\documentclass{amsart}
\usepackage{amsmath}
\usepackage{paralist}
\usepackage{amsfonts}
\usepackage{amssymb}
\usepackage{amsthm}
\usepackage{amscd}
\usepackage{float}
\usepackage{tikz}
\usepackage{pgfplots}
\pgfplotsset{compat=1.15}
\usepackage{graphicx}
\usepackage[colorlinks=true]{hyperref}
\hypersetup{urlcolor=blue, citecolor=red}
\usepackage{hyperref}

\newtheorem{theorem}{Theorem}[section]
\newtheorem{corollary}[theorem]{Corollary}

\newtheorem{lemma}[theorem]{Lemma}
\newtheorem{proposition}[theorem]{Proposition}

\theoremstyle{definition}
\newtheorem{definition}[theorem]{Definition}

\newcommand{\R}{\mathbb{R}}
\newcommand{\Z}{\mathbb{Z}}
\newcommand{\N}{\mathbb{N}}
\newcommand{\B}{\mathcal{B}}
\newcommand{\C}{\mathbb{C}}
\newcommand{\D}{\mathcal{D}}
\newcommand{\E}{\mathcal{E}}
\newcommand{\J}{\mathcal{J}}

\makeatletter
\@namedef{subjclassname@2020}{%
  \textup{2020} Mathematics Subject Classification}
\makeatother

\begin{document}

\title[Maxwell equations with partially anisotropic permittivity]{Strichartz estimates for Maxwell equations in media: The partially anisotropic case}
\author{Robert Schippa}
%\author{Roland Schnaubelt}
\email{robert.schippa@kit.edu}
\address{Fakult\"{a}t f\"{u}r Mathematik, Karlsruher Institut f\"{u}r Technologie, Englerstrasse 2, 76131 Karlsruhe, Germany}
\keywords{Maxwell equations, Strichartz estimates, quasilinear wave equations, rough coefficients, Kerr nonlinearity}
\subjclass[2020]{Primary: 35L45, 35B65, Secondary: 35Q61.}

\begin{abstract}
We prove Strichartz estimates for solutions to Maxwell equations in three dimensions with rough permittivities, which have less than three different eigenvalues. To this end, Maxwell equations are conjugated to half-wave equations in phase space. We use the Strichartz estimates in a known combination with energy estimates to show the new well-posedness results for quasilinear Maxwell equations.
\end{abstract}

\maketitle

\section{Introduction}

In the following Maxwell equations in media in three spatial dimensions, the physically most relevant case (cf. \cite{FeynmanLeightonSands1964, LandauLifschitz1990}), are analyzed. These describe the propagation of \emph{electric and magnetic fields} $(\mathcal{E},\mathcal{B}): \R \times \R^3 \rightarrow \R^3 \times \R^3$, and \emph{displacement and magnetizing fields} $(\mathcal{D},\mathcal{H}): \R \times \R^3 \rightarrow \R^3 \times \R^3$. The system of equations is given by
\begin{equation}
\label{eq:Maxwell3d}
\left\{
 \begin{alignedat}{2}
\partial_t \mathcal{D} &= \nabla \times \mathcal{H} - \J_e, \quad \nabla \cdot \D &&= \rho_e, \\
\partial_t \B &= - \nabla \times \E - \J_m, \quad \nabla \cdot \B &&= \rho_m, \\
\D(0,\cdot) &= \D_0, \quad \qquad \quad \quad \B(0,\cdot) &&= \B_0.
\end{alignedat} \right.
\end{equation}
$(\rho_e,\rho_m): \R \times \R^3 \to \R \times \R$ denote \emph{electric and magnetic charges} and $(\J_e,\J_m): \R \times \R^3 \to \R^3 \times \R^3$ \emph{electric and magnetic currents}. There is no physical evidence for the existence of magnetic charges or magnetic currents, but we include them to highlight a key aspect of the analysis.

The notations follow the previous work \cite{SchippaSchnaubelt2021} on Maxwell equations in two spatial dimensions. We denote space-time coordinates $x=(x^0,x^1,\ldots,x^n) = (t,x^\prime) \in \R \times \R^n$ and the dual variables in Fourier space by $\xi = (\xi^0,\xi^1,\ldots,\xi^n) = (\tau,\xi^\prime) \in \R \times \R^n$.\\
In this work we supplement Maxwell equations with time-instantaneous material laws, relating $\E$ with $\D$ and $\mathcal{H}$ with $\B$:
\begin{equation}
\label{eq:MaterialLaws}
\begin{split}
\D(x) &= \varepsilon(x) \E(x), \quad \varepsilon: \R \times \R^3 \rightarrow \R^{3 \times 3}, \\
\B(x) &= \mu(x) \mathcal{H}(x), \quad \mu: \R \times \R^3 \rightarrow \R^{3 \times 3}. 
\end{split}
\end{equation}
$\varepsilon$ is referred to as \emph{permittivity}, and $\mu$ is referred to as \emph{permeability}. In some cases we shall assume that $\mu \equiv 1$, which means that the considered material is magnetically isotropic. This is a common assumption in nonlinear optics (cf. \cite{MoloneyNewell2004}). Like in the preceding work \cite{SchippaSchnaubelt2021}, we want to describe the propagation in possibly anisotropic and inhomogeneous media. We suppose that $\varepsilon$, $\mu$ are matrix-valued function $\varepsilon, \mu: \R \times \R^3 \rightarrow \R^{3 \times 3}$ with $\Lambda_1, \Lambda_2 > 0$ such that for any $\xi^\prime \in \R^3$ and $x \in \R \times \R^3$
\begin{equation}
\label{eq:UniformEllipticity}
\Lambda_1 |\xi^\prime|^2 \leq \sum_{i,j=1}^3 \kappa^{ij}(x) \xi_i^\prime \xi_j^\prime \leq \Lambda_2 |\xi^\prime|^2, \quad \kappa^{ij}(x) = \kappa^{ji}(x), \quad \kappa \in \{ \varepsilon, \mu \}.
\end{equation}
%Sum convention is in use, e.g.,
%\begin{equation*}
%\varepsilon^{ij}(x) \xi_i^\prime \xi_j^\prime = \sum_{i,j=1}^3 \varepsilon^{ij}(x) \xi_i^\prime \xi_j^\prime.
%\end{equation*}
The case of diagonal $\varepsilon = \text{diag}(\varepsilon_1,\varepsilon_2,\varepsilon_3)$, $\mu = 1_{3 \times 3}$ covers the physically relevant case 
\begin{equation}
\varepsilon(\E) = (1 + |\E|^2) 1_{3 \times 3}
\end{equation}
of the Kerr nonlinearity. The permittivity depends on the electric field itself. We denote
\begin{equation}
\label{eq:RoughSymbolMaxwell}
\mathcal{C}(D) = 
\begin{pmatrix}
0 & -\partial_3 & \partial_2 \\
 \partial_3 & 0 & -\partial_1 \\
-\partial_2 & \partial_1 & 0
\end{pmatrix}
, \quad
P(x,D) = 
\begin{pmatrix}
\partial_t 1_{3 \times 3} & -\mathcal{C}(D) \mu^{-1} \\
\mathcal{C}(D) \varepsilon^{-1} & \partial_t 1_{3 \times 3}
\end{pmatrix}
.
\end{equation}
\eqref{eq:Maxwell3d} becomes
\begin{equation}
\label{eq:Maxwell3dConcise}
P(x,D) \begin{pmatrix} \D \\ \mathcal{H} \end{pmatrix} = - \begin{pmatrix}
\J_e \\ \J_m
\end{pmatrix}
, \quad 
\begin{cases}
 \nabla \cdot \D = \rho_e, \\
 \nabla \cdot \B = \rho_m.
 \end{cases}
\end{equation}

Like for the two-dimensional Maxwell equations covered in \cite{SchippaSchnaubelt2021}, we make use of the FBI transform and analyze the equation in phase space. $P(x,D)$ is conjugated to half-wave equations whose dispersive properties depend on the number of different eigenvalues of $\varepsilon$. This was previously analyzed in the constant-coefficient case by Lucente--Ziliotti \cite{LucenteZiliotti2000} and Liess \cite{Liess1991}; see also \cite{Schippa2022,MandelSchippa2022}. It was proved that for $\varepsilon(x) \equiv \varepsilon$ satisfying \eqref{eq:UniformEllipticity}, solutions to \eqref{eq:Maxwell3dConcise} with $\varepsilon$ having less than three different eigenvalues and $\mu \equiv 1$ decay like solutions to the three-dimensional wave equation. However, if $\varepsilon$ has three different eigenvalues, the decay is weakened to the decay of the two-dimensional wave equation. The fully anisotropic case will be considered separately in \cite{SchippaSchnaubeltFullyAnisotropic}. Presently, we prove the first result for variable rough, possibly anisotropic coefficients. Dumas--Sueur \cite{DumasSueur2012} previously showed Strichartz estimates for smooth scalar coefficients. In the much easier two-dimensional case the eigenvalues of the symbol are always separated in phase space
\begin{equation*}
i (\xi_0, \xi_0 - \| \xi' \|_\varepsilon, \xi_0 + \| \xi' \|_\varepsilon) \text{ for } \| \xi' \|_\varepsilon \sim 1.
\end{equation*}
$\| \xi' \|_{\varepsilon(x)}$ denotes a norm which depends on $\varepsilon(x)$. This separation of the eigenvalues is no longer the case in three dimensions. Roughly speaking, in the isotropic case, the characteristic set is a sphere with multiplicity two and in the partially anisotropic case $\varepsilon(x) = (\varepsilon_1(x),\varepsilon_2(x),\varepsilon_2(x))$, $\varepsilon_1(x) \neq \varepsilon_2(x)$, $\mu = 1_{3 \times 3}$ the characteristic set is described by two ellipsoids intersecting at exactly two points.
The characteristic sets in the partially anisotropic case for constant coefficients were analyzed in detail for the time-harmonic equations in \cite{Schippa2022}. The fact that the ellipsoids are intersecting requires a careful choice of eigenvectors, already in the constant-coefficient case, such that the corresponding Fourier multipliers are $L^p$-bounded.

 It turns out that in the fully anisotropic case $\varepsilon = \text{diag}(\varepsilon_1,\varepsilon_2,\varepsilon_3)$ with $\varepsilon_1 \neq \varepsilon_2 \neq \varepsilon_3 \neq \varepsilon_1$, $\mu = 1_{3 \times 3}$, the characteristic set ceases to be smooth and becomes the Fresnel wave surface with conical singularities.
This is classical and was already pointed out by Darboux \cite{Darboux1993}. The curvature properties were quantified more precisely in \cite{MandelSchippa2022} (see also \cite{Liess1991}). We summarize the properties of the characteristic surface depending on the number of different eigenvalues in Section \ref{subsection:CharacteristicSet}.

\medskip

Below $|D|^\alpha$ and $|D'|^\alpha$ denote Fourier multipliers:
\begin{equation*}
(|D|^\alpha u )\widehat (\xi) = |\xi|^\alpha \hat{u}(\xi), \quad (|D^\prime|^\alpha u) \widehat (\xi) = |\xi^\prime|^\alpha \hat{u}(\xi),
\end{equation*}
and $(\rho,p,q,d)$ is referred to as Strichartz admissible if $d \in \Z_{\geq 2}$, $\rho = d \big( \frac{1}{2}- \frac{1}{q} \big) - \frac{1}{p}$, $p,q \geq 2$, $\frac{2}{p} + \frac{d-1}{q} \leq \frac{d-1}{2}$, and $(p,q,d) \neq (2,\infty,3)$.
We denote the space-time Lebesgue norm of a function $u: \R \times \R^d \to \R$ for $1 \leq p,q < \infty$ by
\begin{equation*}
\| u \|_{L_t^p L_{x'}^q} := \big( \int_{\R} \big( \int_{\R^d} |u(t,x')|^q dx' \big)^{\frac{p}{q}} dt \big)^{\frac{1}{p}}
\end{equation*}
with the usual modifications if $p=\infty$ or $q=\infty$. We recall the following results about Strichartz estimates for wave equations. The sharp range, i.e., global-in-time Strichartz estimates
\begin{equation*}
\| |D'|^{1-\rho} u \|_{L_t^p(\R;L_{x'}^q(\R^d))} \lesssim \| u_0 \|_{\dot{H}^1(\R^d)} + \| u_1 \|_{L^2(\R^d)}
\end{equation*}
for solutions to the Euclidean wave equation
\begin{equation*}
\left\{ \begin{array}{cl}
\partial_t^2 u - \Delta u &= 0, \quad (t,x) \in \R \times \R^d, \; d \geq 2, \\
u(0,\cdot) &= u_0 \in \dot{H}^1(\R^d), \quad \dot{u}(0,\cdot) = u_1 \in L^2(\R^d)
\end{array} \right.
\end{equation*}
with $(\rho,p,q,d)$ Strichartz admissible was covered by Keel--Tao \cite{KeelTao1998}. First results for rough coefficients are due to H. Smith \cite{Smith1998} until Tataru proved the sharp range in a series of papers (cf. \cite{Tataru2000,Tataru2001,Tataru2002}); see also Bahouri--Chemin \cite{BahouriChemin1999} and Klainerman \cite{Klainerman2001}. Tataru recovered the Euclidean Strichartz estimates for $C^2$-coefficients (cf. \cite{Tataru2001}) locally in time and also for coefficients $\| \partial^2 g^{ij} \|_{L^1_t L_{x'}^\infty} < \infty$ (cf. \cite{Tataru2002}). Strichartz estimates for less regular coefficients require additional derivative loss, if one does not impose additional symmetry assumptions on the coefficients as shown in counterexamples by Smith--Tataru \cite{SmithTataru2005}. The Strichartz estimates for coefficients with $\| \partial g \|_{L_t^2 L_{x'}^\infty}< \infty$ can be used to show local well-posedness results for quasilinear wave equations, which improve on the energy method. In the isotropic case we can recover Strichartz estimates for scalar wave equations with rough coefficients.
\begin{theorem}[$C^2$-Strichartz estimates in the isotropic case]
\label{thm:IsotropicStrichartz}
Let $\varepsilon_1, \mu_1 \in C^2(\R \times \R^3;\R)$ and suppose that $\varepsilon = \varepsilon_1 1_{3 \times 3} : \R \times \R^3 \to \R^{3 \times 3}$, $\mu = \mu_1 1_{3 \times 3}: \R \times \R^3 \to \R^{3 \times 3}$ satisfy \eqref{eq:UniformEllipticity}. Let $u= (\D,\mathcal{H}): \R \times \R^3 \to \R^{3 \times 3}$ with $\nabla \cdot \D = \rho_e$ and $\nabla \cdot \mathcal{H} = \rho_m$ and $P$ as in \eqref{eq:RoughSymbolMaxwell}. Denote $\rho_{em} = (\rho_e,\rho_m)$.

If $\| \partial^2 \varepsilon_1 \|_{L^\infty_x} \leq \nu^4$, $\| \partial^2 \mu_1 \|_{L^\infty_x} \leq \nu^4$, then the following estimate holds:
\begin{equation}
\label{eq:StrichartzIsotropicC2}
\| |D|^{-\rho} u \|_{L^p_t L_{x'}^q} \lesssim \nu \| u \|_{L^2_{x}} + \nu^{-1} \| P u \|_{L^2_x} + \| |D|^{-\frac{1}{2}} \rho_{em} \|_{L^2_x}
\end{equation}
provided that the right hand-side is finite and $(\rho,p,q,3)$ is Strichartz admissible.
\end{theorem}
The theorem states that in case of small charges the dispersive properties of wave equations are recovered. Like in the two-dimensional case, note that on the one hand, if
\begin{equation}
\label{eq:LargeCharge}
\| \rho_e \|_{\dot{H}_{x'}^{-\frac{1}{2}}}  \sim \| \D \|_{\dot{H}_{x'}^{\frac{1}{2}}}, \quad \| \rho_m \|_{\dot{H}_{x'}^{-\frac{1}{2}}} \sim \| \B \|_{\dot{H}_{x'}^{\frac{1}{2}}},
\end{equation}
\eqref{eq:StrichartzIsotropicC2} follows from Sobolev embedding. Moreover, we can find stationary solutions $\D = \nabla \varphi$ and $\mathcal{H}=0$ for $\varepsilon = 1_{3 \times 3}$, which would clearly violate \eqref{eq:StrichartzIsotropicC2} 
when omitting the contribution of the charges on the right-hand side in \eqref{eq:StrichartzIsotropicC2}.

Corresponding Strichartz estimates with additional derivative loss under weaker regularity assumptions on $\varepsilon$ and $\mu$ follow by standard means (cf. \cite{Tataru2001, SchippaSchnaubelt2021}).
In the following, for $\lambda \in 2^{\mathbb{Z}}$ we denote Littlewood-Paley projections by
\begin{align*}
(S_\lambda f) \widehat (\xi) = \beta(\lambda^{-1} \| \xi \|) \hat{f}(\xi), \quad (S_\lambda^\prime f) \widehat (\xi) = \beta^\prime(\lambda^{-1} \| \xi' \|) \hat{f}(\xi),
\end{align*}
where $\beta: \R \to \R_{\geq 0}$ denotes a radial function, $\text{supp}(\beta) \subseteq B(0,4) \backslash B(0,1/2)$, which satisfies
\begin{equation}
\label{eq:LPBeta}
\sum_{\lambda \in 2^{\Z}} \beta(\lambda x) = 1 \text{ for } x \neq 0.
\end{equation}

We have the following for $C^s$-coefficients:
\begin{theorem}[$C^s$-Strichartz estimates in the isotropic case]
\label{thm:IsotropicStrichartzEstimatesCsCoefficients}
Let $0<s<2$, $\varepsilon_1,\mu_1 \in C^s(\R \times \R^3;\R)$, $\varepsilon = \text{diag}(\varepsilon_1,\varepsilon_1,\varepsilon_1), \quad \mu = \text{diag}(\mu_1,\mu_1,\mu_1) : \R \times \R^3 \to \R^{3 \times 3}$ be matrix-valued functions with coefficients in $C^s$, $0<s<2$, satisfying \eqref{eq:UniformEllipticity}. Let $u=(\D,\mathcal{H}): \R \times \R^3 \to \R^3 \times \R^3$ with $\nabla \cdot \D = \rho_e$ and $\nabla \cdot \B = \rho_m$, $P$ as in \eqref{eq:RoughSymbolMaxwell}, and write $\rho_{em} = (\rho_e,\rho_m)$. Then, the estimate holds
\begin{equation}
\label{eq:StrichartzEstimatesCscoefficients}
\| |D|^{-\rho-\frac{\sigma}{2}} u \|_{L_t^p L_{x'}^q} \lesssim \nu \| u \|_{L_x^2} + \nu^{-1} \| P u \|_{\dot{H}_x^{-\sigma}} + \| |D|^{-\frac{1}{2}-\frac{\sigma}{2}} \rho_{em} \|_{L_{x}^2} 
\end{equation}
provided that the right hand-side is finite, $(\rho,p,q,3)$ is Strichartz admissible,
\begin{equation*}
\sigma = \frac{2-s}{2+s}, \text{ and } \| (\varepsilon_1,\mu_1) \|_{\dot{C}^s} \leq \nu^4.
\end{equation*}
\end{theorem}
 Moreover, by the arguments from \cite{Tataru2002,SchippaSchnaubelt2021}, Strichartz estimates for coefficients $\partial_x^2 \varepsilon \in L_t^1 L_{x'}^\infty$ (cf. \cite[Theorem~1.3]{SchippaSchnaubelt2021}) and the inhomogeneous equation (cf. \cite[Theorem~1.5]{SchippaSchnaubelt2021}) are proved. We have the following theorem, which is important to treat quasilinear equations. 
\begin{theorem}
\label{thm:StrichartzEstimatesL1LinfCoefficientsIsotropic}
Let $\varepsilon_1,\mu_1 \in C^1(\R \times \R^3; \R)$, and $\varepsilon = \text{diag}(\varepsilon_1,\varepsilon_1,\varepsilon_1) : \R \times \R^3 \to \R^{3 \times 3}$, $\mu = \text{diag}(\mu_1,\mu_1,\mu_1): \R \times \R^3 \to \R^{3 \times 3}$ be matrix-valued functions, which satisfy \eqref{eq:UniformEllipticity} and $\partial_x^2 \varepsilon \in L_t^1 L_{x'}^\infty$, $\partial^2_x \mu \in L^1_t L^\infty_{x'}$. Let $u$, $P$, $\rho_{em}$ be as in Theorem \ref{thm:IsotropicStrichartz}, and $(\rho,p,q,3)$ be Strichartz admissible. Then, the following estimate holds
\begin{equation}
\label{eq:StrichartzEstimatesL1LinfCoefficients}
\begin{split}
\| |D'|^{-\rho} u \|_{L^p(0,T;L_{x'}^q)} &\lesssim \nu^{\frac{1}{p}} \| u \|_{L_t^\infty L_{x'}^2} + \nu^{-\frac{1}{p'}} \| P(x,D) u \|_{L_t^1 L_{x'}^2} \\
&\quad + T^{\frac{1}{p}} ( \| |D'|^{-1+\frac{1}{p}} \rho_{em}(0) \|_{L_{x'}^2(\R^3)} + \| |D'|^{-1+\frac{1}{p}} \partial_t \rho_{em} \|_{L_t^1 L_{x'}^2}),
\end{split}
\end{equation}
whenever the right hand-side is finite, provided that $\nu \geq 1$, and $T \| \partial_x^2 \varepsilon \|_{L_t^1 L_{x'}^\infty} + T \| \partial_x^2 \mu \|_{L_t^1 L_{x'}^\infty} \leq \nu^2$.
\end{theorem} 
 The reason for additional terms $\| |D'|^{-1+\frac{1}{p}} \partial_t \rho \|_{L_t^1 L_{x'}^2}$ compared to \eqref{eq:StrichartzIsotropicC2} is that we use Duhamel's formula in the reductions. For applying the estimates to solve quasilinear equations, $L_t^\infty L_{x'}^2$- and $L_t^1 L_{x'}^2$-norms are to be preferred. We further have to reduce the regularity of $\varepsilon$ to control $\| \partial \varepsilon \|_{L^p L^\infty}$ for energy estimates. We denote homogeneous Besov spaces by $\dot{B}^{pqr}_s$ with norm
 \begin{equation*}
 \| u \|^r_{\dot{B}^{pqr}_s} = \sum_{\lambda \in 2^{\Z}} \lambda^{rs} \| S_\lambda u \|^r_{L_t^p L_{x'}^q}
 \end{equation*}
 with the obvious modification for $r = \infty$. For the coefficients of $\varepsilon$, we use the microlocalizable scale of space (cf. \cite{Tataru2002,SchippaSchnaubelt2021,Taylor1991}):
 \begin{equation*}
 \| v \|_{\mathcal{X}^s} = \sup_{\lambda \in 2^{\Z}} \lambda^s \| S_\lambda v \|_{L_t^1 L_{x'}^\infty}.
 \end{equation*}
 \begin{theorem}
 \label{thm:StrichartzEstimatesXsCoefficients}
 Let $\varepsilon=\text{diag}(\varepsilon_1,\varepsilon_1,\varepsilon_1), \; \mu = \text{diag}(\mu_1,\mu_1,\mu_1) \in \mathcal{X}^s$, $0<s<2$, and $u=(\D,\mathcal{H})$, $(\rho,p,q,3)$, and $\sigma$ as in the assumptions of Theorem \ref{thm:IsotropicStrichartzEstimatesCsCoefficients}. Then, the following estimate holds:
 \begin{equation}
\label{eq:StrichartzEstimatesXsCoefficients}
\begin{split}
\| |D|^{-\rho - \frac{\sigma}{p}} u \|_{\dot{B}^{pq \infty}_0} &\lesssim \nu^{\frac{1}{p}} \| u \|_{L_T^\infty L_{x'}^2} + \nu^{-\frac{1}{p'}} \| |D|^{-\sigma} P u \|_{L_t^1 L_{x'}^2}  \\
&\quad + T^{\frac{1}{p}} ( \| |D'|^{-1+\frac{1}{p}-\frac{\sigma}{p}} \rho_{e m} \|_{L_t^\infty L_{x'}^2} + \| |D'|^{-1+\frac{1}{p}-\frac{\sigma}{p}} \partial_t \rho_{e m} \|_{L_t^1 L_{x'}^2} ) 
 \end{split}
 \end{equation}
 for all $u$ compactly supported in $[0,T]$, and $\nu$, $T$ verifying
 \begin{equation*}
 T^s \| (\varepsilon_1, \mu_1) \|^2_{\mathcal{X}^s} \lesssim \nu^{2+s}.
 \end{equation*}
 \end{theorem}
Further inhomogeneous Strichartz estimates are proved by similar means as in \cite{SchippaSchnaubelt2021}, which is omitted here. In the partially anisotropic case a diagonalization is still possible, but error terms arising from the composition of pseudo-differential operators presently only allow to prove inferior estimates. We show the following:
 \begin{theorem}
 \label{thm:PartiallyAnisotropicStrichartz}
 Let $\varepsilon = (\varepsilon_1,\varepsilon_2,\varepsilon_2): \R \times \R^3 \to \R^{3 \times 3}$ satisfy \eqref{eq:UniformEllipticity}, let $u = (\D, \mathcal{H})$, $\rho_{em} = (\rho_e,\rho_m)$, and $P$ be as in \eqref{eq:Maxwell3dConcise}. Let $T >0$ and $\delta > 0$. 
 \begin{itemize} 
  \item If $\partial \varepsilon \in L_T^\infty L_{x'}^\infty$, then the following estimate holds:
 \begin{equation}
 \label{eq:StrichartzPartiallyAnisotropicLipschitz}
\begin{split}
\| \langle D' \rangle^{-\rho - \frac{1}{3p} - \delta} u \|_{L^p(0,T;L^q(\R^3))} &\lesssim_{T,\delta} \| u_0 \|_{L^2(\R^3)} + \| P u \|_{L_T^1 L_{x'}^2} \\
&\quad + \| \langle D' \rangle^{-\frac{2}{3}} \rho_{em}(0) \|_{L^2_{x'}} + \| \langle D' \rangle^{-\frac{2}{3}} \partial_t \rho_{em} \|_{L_T^1 L_{x'}^2}.
\end{split} 
 \end{equation}
 \item If $\partial \varepsilon \in L_T^2 L_{x'}^\infty$, then the following estimate holds:
\begin{equation}
\label{eq:StrichartzPartiallyAnisotropicL2Lipschitz}
\begin{split}
\| \langle D' \rangle^{-\rho - \frac{1}{2p} - \delta} u \|_{L^p(0,T;L^q(\R^3))} &\lesssim_{T,\delta} \| u_0 \|_{L^2(\R^3)} + \| P u \|_{L_T^1 L_{x'}^2} \\
&\quad + \| \langle D' \rangle^{-\frac{3}{4}} \rho_{em}(0) \|_{L^2_{x'}} + \| \langle D' \rangle^{-\frac{3}{4}} \partial_t \rho_{em} \|_{L_T^1 L_{x'}^2}.
\end{split}
\end{equation} 
 \end{itemize}
 \end{theorem}
 Moreover, the method of proof recovers the estimates from Theorem \ref{thm:IsotropicStrichartz} for $\varepsilon_1(x) = e_1(t,x_1)$ and $\varepsilon_2(x) = e_2(t,x_1)$. In this case the problematic error terms, which arise from composing pseudo-differential operators in the general case, vanish. We have the following:
\begin{theorem}[$C^2$-Strichartz estimates in the structured partially anisotropic case]
\label{thm:StructuredStrichartzPartiallyAnisotropic}
Let $\varepsilon = \text{diag}(\varepsilon_1,\varepsilon_2,\varepsilon_2): \R \times \R^3 \to \R^{3 \times 3}$ satisfy \eqref{eq:UniformEllipticity} and for $i=1,2$ suppose that
\begin{equation*}
\varepsilon_i(x_0,x') = e_i(x_0,x_1) \text{ with } e_i \in C^2(\R \times \R; \R).
\end{equation*}
Let $u=(\mathcal{D},\mathcal{H}): \R \times \R^3 \to \R^3 \times \R^3$ and set $\mu = 1_{3 \times 3}$. If $\| \partial^2 \varepsilon \|_{L^\infty_x} \leq \nu^4$, then the following estimate holds:
\begin{equation*}
\| |D|^{-\rho} u \|_{L_t^p L_{x'}^q} \lesssim \nu \| u \|_{L^2_x} + \nu^{-1} \| P u \|_{L^2_x} + \| |D|^{-\frac{1}{2}} \rho_{em} \|_{L^2_x}
\end{equation*}
provided that the right hand-side is finite and $(\rho,p,q,3)$ is Strichartz admissible.
\end{theorem}
 Strichartz estimates for less regular coefficients like in Theorems \ref{thm:IsotropicStrichartzEstimatesCsCoefficients} and \ref{thm:StrichartzEstimatesL1LinfCoefficientsIsotropic} hold for $C^s$-coefficients or $\partial^2 \varepsilon \in L_T^1 L_{x'}^\infty$ under structural assumptions.

As in \cite{SchippaSchnaubelt2021}, after conjugation of $P(x,D)$ the key ingredient in the proof of Strichartz estimates are estimates for the half-wave equations. We use the following result, shown in \cite{SchippaSchnaubelt2021}:
\begin{proposition}[{\cite[Proposition~1.8]{SchippaSchnaubelt2021}}]
\label{prop:HalfWaveEstimates}
Let $\lambda \in 2^{\mathbb{N}_0},$ $\lambda \gg 1,$ and  $d \geq 2$. Assume $\varepsilon = \varepsilon^{ij}(x)$ satisfies $\varepsilon^{ij} \in C^2$, $\| \partial^2_x \varepsilon \|_{L^\infty} \leq 1$, and \eqref{eq:UniformEllipticity}. Let $Q(x,D)$ denote the pseudo-differential operator with symbol 
\begin{equation*}
Q(x,\xi) = - \xi_0 + \big( \varepsilon^{ij}_{\lambda^\frac12} (x) \xi_i \xi_j \big)^{1/2}.
\end{equation*}
Moreever, let $u:\R \times \R^d \to \R$ decay rapidly outside the unit cube and $(\rho,p,q,d)$ be Strichartz admissible. Then, we find the estimates 
\begin{equation}
\label{eq:StrichartzEstimatesHalfWaveEquationC2Coefficients}
\lambda^{-\rho} \Vert S_\lambda u \Vert_{L^p L^q} \lesssim \Vert S_\lambda u \Vert_{L^2} + \Vert Q(x,D) S_\lambda u \Vert_{L^2}
\end{equation}
to hold with an implicit constant uniform in $\lambda$. For Lipschitz coefficients $\varepsilon^{ij}$ 
with $\| \partial^2_x \varepsilon \|_{L^1 L^\infty} \leq 1$, we obtain
\begin{equation}
\label{eq:StrichartzEstimatesHalfWaveEquationL1LinfCoefficients}
\lambda^{-\rho} \Vert S_\lambda u \Vert_{L^p L^q} \lesssim \Vert S_\lambda u \Vert_{L^\infty L^2} + \Vert Q(x,D) S_\lambda u \Vert_{ L^2}.
\end{equation}
\end{proposition}

We want to use the Strichartz estimates to improve the local well-posedness for quasilinear Maxwell equations:
\begin{equation}
\label{eq:QuasilinearMaxwellKerrNonlinearity}
\left\{ \begin{array}{cl}
P(x,D) (\D,\mathcal{H}) &= 0, \quad \nabla \cdot \D = \nabla \cdot \mathcal{H} = 0, \\
(\D,\mathcal{H})(0) &\in H^s(\R^3;\R)^6,
\end{array} \right.
\end{equation}
where $\varepsilon^{-1}(\D) = \psi(|\D|^2) 1_{3 \times 3}$, and $\psi: \R_{\geq 0} \to \R_{\geq 1}$ is a smooth monotone increasing function with $\psi(0)=1$. This covers the Kerr nonlinearity $\varepsilon = (1+|\E|^2)1_{3 \times 3}$. The energy method (cf. \cite{IfrimTataru2020}) yields local well-posedness for $s>5/2$. We also refer to Spitz's works \cite{SpitzPhdThesis2017,Spitz2019}, where Maxwell equations with Kerr nonlinearity were proved to be locally well-posed in $H^3(\Omega)$ on domains with suitable boundary conditions. 
We compute
\begin{equation*}
\begin{split}
\partial_t (\psi(|\D|^2) \D) &= \psi(|\D|^2) \partial_t \D + (2 \psi'(|\D|^2) \D \otimes \D) \partial_t \D =: \tilde{\psi}_1(\D) \partial_t \D, \\
\nabla \times (\psi(|\D|^2) \D) &=  [ \psi(|\D|^2) \nabla \times + (2 \psi'(|\D|^2) (\D \otimes (\D \times \nabla))^t)] \D =: \tilde{\psi}_2(\D) \D.
\end{split}
\end{equation*}
After a diagonalization in phase space, we shall see that $\tilde{\psi}_1(\D)$ and $\tilde{\psi}_2(\D)$ have at most two different eigenvalues.

Passing to the second order systen yields the system of wave equations:
\begin{equation}
\label{eq:KerrSecondOrderSystem}
\left\{ \begin{array}{cl}
\partial_t^2 \D &= - \nabla \times (\tilde{\psi}_2(\D) \nabla \times \D), \quad \nabla \cdot \D = \nabla \cdot \mathcal{H} = 0, \\
\partial_t^2 \mathcal{H} &= - \nabla \times (\tilde{\psi}_1(\D) \nabla \times \mathcal{H}). 
\end{array} \right.
\end{equation}

We shall first consider the simplified Kerr system, which is obtained by replacing $\tilde{\psi}_i$ with $\psi(|\D|^2)$:
\begin{equation}
\label{eq:SimplifiedKerr}
\partial_t^2 \D = - \nabla \times (\psi(|\D|^2) \nabla \times \D), \quad \nabla \cdot \D = 0.
\end{equation}
In this case we can apply the Strichartz estimates for isotropic permittivity to prove the following:
\begin{theorem}[Local well-posedness for the simplified Kerr system]
\label{thm:LocalWellposednessSimplifiedKerr}
\eqref{eq:SimplifiedKerr} is locally well-posed for $s>\frac{13}{6}$.
\end{theorem}
We remark that we could likewise treat the system
\begin{equation*}
\left\{ \begin{array}{cl}
\partial_t^2 \D &= - \nabla \times (\psi(|\D|^2) \nabla \times \D), \quad \nabla \cdot \D = 0, \\
\partial_t^2 \mathcal{H} &= - \nabla \times (\psi(|\D|^2) \nabla \times \mathcal{H}, \quad \nabla \cdot \mathcal{H} =0.
\end{array} \right.
\end{equation*}
with the additional estimates for $\mathcal{H}$ being carried out in similar spirit.
\vspace*{0.3cm}

In the case of partially anisotropic permittivity, we can use the Strichartz estimates from Theorem \ref{thm:PartiallyAnisotropicStrichartz} directly:
\begin{theorem}[Local well-posedness for Maxwell equations with partially anisotropic permittivity]
\label{thm:PartiallyAnisotropicWellposedness}
Let $\varepsilon^{-1} = \text{diag}(\psi(|\D_1|^2),1,1)$ with $\psi: \R_{\geq 0} \to \R_{\geq 1}$ smooth,  monotone increasing, and $\psi(0) = 1$. Then, the Maxwell system
\begin{equation*}
\left\{ \begin{array}{cl}
\partial_t \D &= \nabla \times \mathcal{H}, \quad \nabla \cdot \D = \nabla \cdot \mathcal{H} = 0, \\
\partial_t \mathcal{H} &= - \nabla \times (\varepsilon^{-1} \D), \; (\D(0),\mathcal{H}(0)) \in H^s(\R^3;\R^6)
\end{array} \right.
\end{equation*}
is locally well-posed for $s>9/4$.
\end{theorem}

%After carrying out a diagonalization in phase space, we obtain the following improvement over energy arguments via Strichartz estimates for the (full) Kerr system:
%\begin{theorem}[Local well-posedness for quasilinear Maxwell equations with Kerr nonlinearity]
%\label{thm:LocalWellposednessKerr}
%\eqref{eq:QuasilinearMaxwellKerrNonlinearity} is locally well-posed for $s>9/4$.
%\end{theorem}
In the two-dimensional case we have shown that the derivative loss for Strichartz estimates with rough coefficients is sharp (cf. \cite[Section~7]{SchippaSchnaubelt2021}). In the three dimensional case we do not have an example showing sharpness. However, the fact that the derivative loss in the isotropic case matches the loss for second order hyperbolic operators indicates sharpness of the Strichartz estimates in the isotropic case.

\vspace{0.5cm}
\emph{Outline of the paper.} In Section \ref{section:Preliminaries} we introduce further notations and recall well-known bounds for pseudo-differential operators and the FBI transform. In Section \ref{section:StrichartzEstimatesIsotropic}, we point out how standard localization arguments reduce Theorems \ref{thm:IsotropicStrichartz} and \ref{thm:StrichartzEstimatesL1LinfCoefficientsIsotropic} to a dyadic estimate with frequency truncated coefficients. Then, the symbol is diagonalized to two degenerate and four non-degenerate half wave equations after an additional localization in phase space. We see that the divergence conditions ameliorate the contribution of the degenerate components as in the two-dimensional case. The estimates for the non-degenerate half-wave equations for $\varepsilon$ having less than three eigenvalues are provided by Proposition \ref{prop:HalfWaveEstimates}. In Section \ref{section:PartiallyAnisotropicCase} we show the Strichartz estimates in Theorem \ref{thm:PartiallyAnisotropicStrichartz} and \ref{thm:StructuredStrichartzPartiallyAnisotropic} for partially anisotropic permittivities with rough coefficients. In Section \ref{section:LocalWellposednessQuasilinearMaxwell} we consider quasilinear Maxwell equations and prove Theorems \ref{thm:LocalWellposednessSimplifiedKerr} and \ref{thm:PartiallyAnisotropicWellposedness}.

\section{Preliminaries}
\label{section:Preliminaries}
In this section we collect basic facts about pseudo-differential operators and the FBI transform to be used in the sequel. 
\subsection{Pseudo-differential operators with rough symbols}
In the following we clarify the quantization and recall the composition formulae for pseudo-differential operators presently considered. We refer to \cite{Hoermander2007,Taylor1974,Taylor1991} for further reading.

Recall the standard H\"ormander class of symbols:
\begin{equation*}
S^m_{\rho,\delta} = \{ a \in C^\infty(\R^m \times \R^m) : |\partial_x^\alpha \partial_\xi^\beta a| \lesssim (1+|\xi|)^{m-\rho |\beta| + |\alpha| \delta} \}
\end{equation*}
for $m \in \R$, $0 \leq \delta \leq \rho \leq 1$. In the following we obtain pseudo-differential operators via the quantization:
\begin{equation*}
a(x,D) f = (2 \pi)^{-m} \int_{\R^m} e^{i x.\xi} a(x,\xi) \hat{f}(\xi) d\xi.
\end{equation*}
The $L^p$-boundedness of $a(x,D)$ with $a \in S^0_{1,\delta}$, $0 \leq \delta < 1$ is standard (cf. \cite[Section~0.11]{Taylor1991}). In the present context of rough coefficients, we shall also consider symbols which are rough in the spatial variable. After a Littlewood-Paley decomposition and a paradifferential decomposition, we can reduce to H\"ormander symbols. We record the following quantification of $L^p L^q$-boundedness for symbols, which are smooth and compactly supported in the fiber variable and possibly rough in the spatial variable:
\begin{lemma}[{\cite[Lemma~2.3]{SchippaSchnaubelt2021}}]
\label{lem:LpLqBoundsPseudos}
Let $1 \leq p,q \leq \infty$ and $a \in C^s_x C^\infty_c(\R^m \times \R^m)$ with $a(x,\xi) = 0$ for $\xi \notin B(0,2)$. Suppose that
\begin{equation*}
\sup_{x \in \R^m} \sum_{0 \leq |\alpha| \leq m+1} \| D_\xi^\alpha a(x,\cdot) \|_{L^1_\xi} \leq C.
\end{equation*}
Then, we find the following estimate to hold:
\begin{equation*}
\| a(x,D) f \|_{L^p L^q} \lesssim C \| f \|_{L^p L^q}.
\end{equation*}
\end{lemma}
We recall the Kohn--Nirenberg theorem on symbol composition. Denote
\begin{equation*}
\partial_x^\alpha = \partial_{x_1}^{\alpha_1} \ldots \partial_{x_m}^{\alpha_m} \text{ and } D_\xi^\alpha = \partial_\xi^\alpha / (i^{|\alpha|})
\end{equation*}
for $\alpha \in \N_0^m$.
\begin{theorem}[{\cite[Proposition~0.3C]{Taylor1991}}]
\label{thm:KohnNirenberg}
Let $m_1,m_2 \in \R$, $0 \leq \delta_i < \rho_i \leq 1$ for $i=1,2$. Given $P(x,\xi) \in S^{m_1}_{\rho_1,\delta_1}$, $Q(x,\xi) \in S^{m_2}_{\rho_2,\delta_2}$, suppose that
\begin{equation*}
0 \leq \delta_2 < \rho \leq 1 \text{ with } \rho = \min(\rho_1,\rho_2).
\end{equation*}
Then, $(P \circ Q)(x,D) \in OPS^{m_1+m_2}_{\rho,\delta}$ with $\delta = \max(\delta_1,\delta_2)$, and $P(x,D) \circ Q(x,D)$ satisfies the asymptotic expansion
\begin{equation}
\label{eq:ExpansionPseudoComposition}
(P \circ Q)(x,D) = \sum_\alpha \frac{1}{\alpha !} (D_\xi^\alpha P \partial_x^\alpha Q)(x,D) +R,
\end{equation}
where $R: \mathcal{S}' \to C^\infty$ is smoothing.
\end{theorem}
Lemma \ref{lem:LpLqBoundsPseudos} quantifies the $L^p L^q$-bounds for the expansion \eqref{eq:ExpansionPseudoComposition} (see \cite[Section~2]{SchippaSchnaubelt2021}). From truncating the expansion to
\begin{equation*}
(P \circ Q)(x,D) = \sum_{|\alpha| \leq N} \frac{1}{\alpha !} (D_\xi^\alpha P \partial_x^\alpha Q)(x,D) + R_N(x,D),
\end{equation*}
we can find error bounds for $R_N$ decaying in $\lambda$. This can be proved again by Lemma \ref{lem:LpLqBoundsPseudos}. We recall the Calderon--Vaillancourt theorem (cf. \cite{CalderonVaillancourt1972,Taylor1974}) to bound $OPS^0_{\rho,\rho}$. The following quantification is due to Kato \cite{Kato1976}:
\begin{theorem}[Calderon--Vaillancourt]
Let $0 \leq \rho < 1$ and $a(x,\xi) \in S^0_{\rho,\rho}(\R^{2d})$ with
\begin{equation*}
|\partial_x^\alpha \partial_\xi^\beta a(x,\xi)| \leq C (1+|\xi|)^{\rho(|\alpha|-|\beta|)}
\end{equation*}
for $|\alpha| \leq \lfloor \frac{d}{2} \rfloor +1$, $|\beta| \leq \lfloor \frac{d}{2} \rfloor + 2$. Then,
\begin{equation*}
\| a(x,D) \|_{L^2 \to L^2} \lesssim_{\rho,d} C.
\end{equation*} 
\end{theorem}

\subsection{The FBI transform} We shall make use of the FBI transform to conjugate the evolution to phase space (cf. \cite{Delort1992,Tataru2002}). For $\lambda \in 2^{\Z}$, we define the FBI transform of $f \in L^1(\R^m; \C)$ by
\begin{equation*}
\begin{split}
T_\lambda f(z) &= C_m \lambda^{\frac{3m}{4}} \int_{\R^m} e^{-\frac{\lambda}{2}(z-y)^2} f(y) dy, \qquad z = x- i \xi \in T^* \R^m \equiv \R^{2m}, \\
C_m &= 2^{-\frac{m}{2}} \pi^{-\frac{3m}{4}}.
\end{split}
\end{equation*}
The FBI transform is an isometric mapping $T_\lambda : L^2(\R^m) \to L^2_\Phi(T^* \R^m)$ with $\Phi(z) = e^{-\lambda \xi^2}$. The range of $T_\lambda$ are holomorphic functions, thus there are many inversion formulae. One is given by the adjoint in $L^2_\Phi$:
\begin{equation*}
T_\lambda^* F(y) = C_m \lambda^{\frac{3m}{4}} \int_{\R^{2m}} e^{-\frac{\lambda}{2}(\bar{z}-y)^2} \Phi(z) F(z) dx d\xi.
\end{equation*}
By decomposing a function into coherent states, the FBI transform allows us to find an approximate conjugate of pseudo-differential operators. Let $s \geq 0$, $a(x,\xi) \in C^s_x C^\infty_c$ be smooth and compactly supported in $\xi$. We assume that
\begin{equation*}
a(x,\xi) = 0 \text{ for } \xi \notin B(0,2).
\end{equation*}
Let $a_\lambda(x,\xi) = a(x,\xi/\lambda)$ denote the scaled symbol and $A_\lambda = a_\lambda(x,D)$ be the corresponding pseudo-differential operator. We have the following asymptotic for analytic symbols:
\begin{equation*}
T_\lambda A_\lambda(x,D) \approx \sum_{\alpha,\beta} (\partial_\xi - \lambda \xi)^\alpha \frac{\partial_x^\alpha \partial_\xi^\beta a(x,\xi)}{|\alpha|! |\beta|! (-i \lambda)^{|\alpha|} \lambda^{|\beta|}} \big( \frac{1}{i} \partial_x - \lambda \xi)^\beta T_\lambda.
\end{equation*}
We consider truncations of the asymptotic expansion. For $s \leq 1$, we let
\begin{equation*}
\tilde{a}^s_\lambda = a,
\end{equation*}
and for $1 < s \leq 2$, let
\begin{equation*}
\tilde{a}^s_\lambda = a + \frac{1}{-i\lambda} a_x(\partial_\xi - \lambda \xi) + \frac{1}{\lambda} a_\xi (\frac{1}{i} \partial_x - \lambda \xi) = a + \frac{2}{\lambda} (\bar{\partial} a)(\partial - i \lambda \xi)
\end{equation*}
with $\partial = \frac{1}{2} (\partial_x + i \partial_\xi)$ and $\bar{\partial} = \frac{1}{2} (\partial_x - i \partial_\xi)$. We define the remainder
\begin{equation*}
R^s_{\lambda,a} = T_\lambda A_\lambda  - \tilde{a}^s_\lambda T_\lambda.
\end{equation*}
Tataru \cite{Tataru2000,Tataru2001} proved the following approximation result:
\begin{theorem}[{\cite[Theorem~5,~p.~393]{Tataru2001}}]
\label{thm:ApproximationTheorem}
Let $0<s\leq 2$, and $a \in C^s_x C^\infty_c$. Then,
\begin{equation*}
\begin{split}
\| R^s_{\lambda,a} \|_{L^2 \to L^2_\Phi} &\lesssim \lambda^{-\frac{s}{2}}, \\
\| (\partial_\xi - \lambda \xi) R^s_{\lambda,a} \|_{L^2 \to L^2_\Phi} &\lesssim \lambda^{\frac{1}{2}-\frac{s}{2}}.
\end{split}
\end{equation*}
Moreover, if $a \in X^1 C^\infty_c$ with $X^1 = \{ f \in L_t^2 L_{x'}^\infty \, : \, \partial f \in L_t^2 L_{x'}^\infty \}$, then
\begin{equation*}
\| R^1_{\lambda,a} \|_{L^\infty L^2 \to L^2_\Phi} \lesssim \lambda^{-\frac{1}{2}}.
\end{equation*}
\end{theorem}
%We use the following multiplier theorem in phase space:
%\begin{proposition}[{\cite[Proposition~2.2]{SchippaSchnaubelt2021}}]
%Let $1 \leq p,q \leq \infty$, $s \geq 0$, $a \in C^s_x C^\infty_c(\R^m \times \R^m)$ with $a(x,\xi) = 0$ for $\xi \notin B(0,2)$, and
%\begin{equation*}
%\sup_{x \in \R^m} \sum_{0 \leq |\alpha| \leq m+1} \| D_\xi^\alpha a(x,\cdot) \|_{L^1_\xi} \leq C.
%\end{equation*}
%Then, we find the following estimate to hold:
%\begin{equation*}
%\| T_\lambda^* a(x,\xi) T_\lambda f \|_{L^p_{x_0} L^q_{x'}} \lesssim C \| f \|_{L^p_{x_0} L^q_{x'}}.
%\end{equation*}
%\end{proposition}

\subsection{The characteristic set depending on the permittivity}
\label{subsection:CharacteristicSet}

In this section we summarize the characteristic set of Maxwell equations depending on the number of different eigenvalues of $\varepsilon=\text{diag}(\varepsilon_1,\varepsilon_2,\varepsilon_3)$ and $\mu = \text{diag}(\mu_1,\mu_2,\mu_3)$. For this discussion suppose that $\varepsilon$ and $\mu$ are homogeneous. The partially anisotropic case (and isotropic case as special case) was detailed in \cite{Schippa2022} and the fully anisotropic case was analyzed in \cite{MandelSchippa2022}.

\subsubsection{Isotropic case}
For $\mu$ and $\varepsilon$ proportional to the unit matrix we can diagonalize the principal symbol to the diagonal matrix as will be carried out in Section \ref{section:StrichartzEstimatesIsotropic}:
\begin{equation*}
d(x,\xi) = i \text{diag}(\xi_0,\xi_0,\xi_0-(\varepsilon \mu)^{-\frac{1}{2}} \| \xi' \|, \xi_0 + (\varepsilon \mu)^{-\frac{1}{2}} \| \xi' \|, \xi_0-(\varepsilon \mu)^{-\frac{1}{2}} \| \xi' \|, \xi_0 + (\varepsilon \mu)^{-\frac{1}{2}} \| \xi' \|).
\end{equation*}
This shows that the characteristic set, without the contribution of the charges, is given by
\begin{equation*}
\{ \xi_0^2 - (\varepsilon \mu)^{-1} \| \xi' \|^2 = 0 \}
\end{equation*}
with multiplicity two.
%\begin{tikzpicture}
%\begin{axis}
%[view={135}{20},%colormap/blackwhite,
%axis lines=center, axis on top,ticks=none,
%set layers=default,axis equal,
%xlabel={$x$}, ylabel={$y$}, zlabel={$z$},
%xlabel style={anchor=south east},
%ylabel style={anchor=south west},
%zlabel style={anchor=south west},
%enlargelimits,
%tick align=inside,
%domain=0:2.00,
%samples=20, 
%z buffer=sort,
%]
%\addplot3 [surf,opacity=0.4,domain=-1:0,
%domain y=0:360] ({sin(y)*sqrt(1-x^2)},{cos(y)*sqrt(1-x^2)},{x});
%\addplot3 [surf,opacity=0.4,domain=0:1,
%domain y=0:360,on layer=axis foreground] ({sin(y)*sqrt(1-x^2)},{cos(y)*sqrt(1-x^2)},{x});
%\end{axis}
%\end{tikzpicture}

\subsubsection{Partially anisotropic case} In the case $\varepsilon = \text{diag}(\varepsilon_1,\varepsilon_2,\varepsilon_2)$, $\varepsilon_1 \neq \varepsilon_2$, $\mu = \mu 1_{3 \times 3}$ the diagonalization in the constant-coefficient case with $L^p$-bounded multipliers is still possible. We obtain the diagonal matrix:
\begin{equation*}
d(x,\xi) = i \text{diag}(\xi_0,\xi_0,\xi_0 - \varepsilon_2^{-\frac{1}{2}} \| \xi' \|, \xi_0 + \varepsilon_2^{-\frac{1}{2}} \| \xi' \|, \xi_0 - \| \xi' \|_\varepsilon, \xi_0 + \| \xi' \|_\varepsilon)
\end{equation*}
with $\| \xi' \|_\varepsilon = (\varepsilon_2^{-1} \xi_1^2 + \varepsilon_1^{-1} \xi_2^2 + \varepsilon_1^{-1} \xi_3^2)^{\frac{1}{2}}$. Clearly,  we have 
\begin{equation*}
\varepsilon_2^{-\frac{1}{2}} \| \xi' \| = \| \xi' \|_\varepsilon \; \Leftrightarrow \; \xi_2' = \xi_3' = 0.
\end{equation*}
The characteristic set is given by the sphere
\begin{equation*}
(\xi_0^2 - \| \xi' \|_\varepsilon^2) (\xi_0^2 - \varepsilon_2^{-1} \| \xi' \|^2) = 0
\end{equation*}
and describes two ellipsoids, which are smoothly intersecting at the $\xi_1$-axis.

%\begin{tikzpicture}
%\begin{axis}
%[view={135}{20},%colormap/blackwhite,
%axis lines=center, axis on top,ticks=none,
%set layers=default,axis equal,
%xlabel={$x$}, ylabel={$y$}, zlabel={$z$},
%xlabel style={anchor=south east},
%ylabel style={anchor=south west},
%zlabel style={anchor=south west},
%enlargelimits,
%tick align=inside,
%domain=0:2.00,
%samples=20, 
%z buffer=sort,
%]
%\addplot3 [surf,opacity=0.2,domain=-1:0,
%domain y=0:360] ({sin(y)*sqrt(1-x^2)},{2*cos(y)*sqrt(1-x^2)},{2*x});
%\addplot3 [surf,opacity=0.2,domain=0:1,
%domain y=0:360,on layer=axis foreground] ({sin(y)*sqrt(1-x^2)},{2*cos(y)*sqrt(1-x^2)},{2*x});
%\addplot3 [surf,opacity=0.6,domain=-1:0,
%domain y=0:360] ({sin(y)*sqrt(1-x^2)},{cos(y)*sqrt(1-x^2)},{x});
%\addplot3 [surf,opacity=0.6,domain=0:1,
%domain y=0:360,on layer=axis foreground] ({sin(y)*sqrt(1-x^2)},{cos(y)*sqrt(1-x^2)},{x});
%\end{axis}
%\end{tikzpicture}

\subsubsection{Fully anisotropic case}
To find the characteristic set in the fully anisotropic case, we symmetrize
\begin{equation*}
\begin{pmatrix}
i \xi_0 & - i \mathcal{C}(\xi') \mu^{-1} \\
i \mathcal{C}(\xi') \varepsilon^{-1} & i \xi_0
\end{pmatrix}
\begin{pmatrix}
\hat{\mathcal{D}} \\ \hat{\mathcal{B}}
\end{pmatrix}
= 0
\end{equation*}
by multiplying with the matrix (cf. \cite[Proposition~1.3,~p.~1835]{MandelSchippa2022})
\begin{equation*}
\begin{pmatrix}
i \xi_0 & i \mathcal{C}(\xi') \mu^{-1} \\
-i \mathcal{C}(\xi') \varepsilon^{-1} & i \xi_0
\end{pmatrix}
\end{equation*}
to find
\begin{equation*}
\begin{pmatrix}
-\xi_0^2 - \mathcal{C}(\xi') \mu^{-1} \mathcal{C}(\xi') \varepsilon^{-1} & 0 \\
0 & -\xi_0^2 - \mathcal{C}(\xi') \varepsilon^{-1} \mathcal{C}(\xi') \mu^{-1}
\end{pmatrix}
\begin{pmatrix}
\hat{\mathcal{D}} \\ \hat{\mathcal{H}}
\end{pmatrix}
= 0.
\end{equation*}
We compute
\begin{equation*}
\begin{split}
p(\xi) &= \det ( -\xi_0^2 - \mathcal{C}(\xi') \mu^{-1} \mathcal{C}(\xi') \varepsilon^{-1}) \\
&= \det (-\xi_0^2 - \mathcal{C}(\xi') \varepsilon^{-1} \mathcal{C}(\xi') \mu^{-1}) = -\xi_0^2 (\xi_0^4 -\xi_0^2 q_0(\xi) + q_1(\xi))
\end{split}
\end{equation*}
with
\begin{equation*}
\begin{split}
q_0(\xi) &= \xi_1^2 \big( \frac{1}{\varepsilon_2 \mu_3} + \frac{1}{\mu_2 \varepsilon_3} \big) + \xi_2^2 \big( \frac{1}{\varepsilon_1 \mu_3} + \frac{1}{\varepsilon_3 \mu_1} \big) + \xi_3^2 \big( \frac{1}{\varepsilon_2 \mu_1} + \frac{1}{\varepsilon_1 \mu_2} \big), \\
q_1(\xi) &= \frac{1}{\varepsilon_1 \varepsilon_2 \varepsilon_3 \mu_1 \mu_2 \mu_3} (\varepsilon_1 \xi_1^2 + \varepsilon_2 \xi_2^2 + \varepsilon_3 \xi_3^2) (\mu_1 \xi_1^2 + \mu_2 \xi_2^2 + \mu_3 \xi_3^2).
\end{split}
\end{equation*}

 It \cite[Section~3]{MandelSchippa2022} was proved that the condition for full anisotropy $\varepsilon = \text{diag}(\varepsilon_1,\varepsilon_2,\varepsilon_3)$ and $\mu = \text{diag}(\mu_1,\mu_2,\mu_3)$ is given by
\begin{equation}
\label{eq:FullAnisotropy} 
 \frac{\varepsilon_1}{\mu_1} \neq \frac{\varepsilon_2}{\mu_2} \neq \frac{\varepsilon_3}{\mu_3} \neq \frac{\varepsilon_1}{\mu_1}.
\end{equation}
 If this fails, then the characteristic set will be like in the isotropic or partially anisotropic case.

If \eqref{eq:FullAnisotropy} holds, then the characteristic set ceases to be smooth and becomes the Fresnel wave surface with conical singularities.
       \begin{figure}
      \includegraphics[scale=0.3]{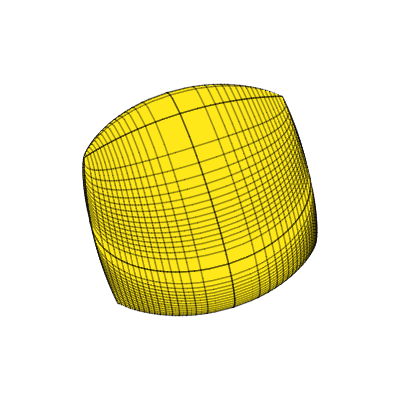}
      \includegraphics[scale=0.2]{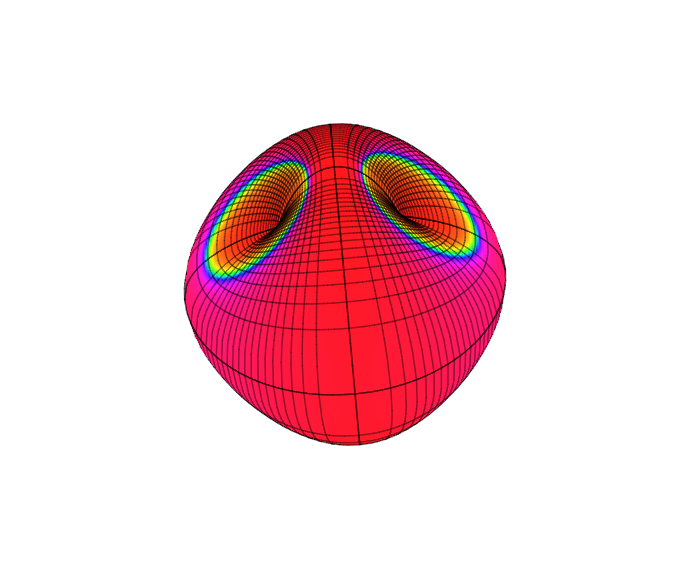}
      \caption{Fresnel surface: inner sheet (left) and outer
      sheet (right) for $\varepsilon_1 = 1$, $\varepsilon_2 = 3$, $\varepsilon_3 = 9$.
      The colours on the outer sheet indicate the Gaussian curvature. The Hamiltonian
      circles with vanishing Gaussian curvature encase the singular points. The images were created with MAPLE\textsuperscript{\texttrademark} with the working sheet \cite{MapleSheet}.
      %The contact of inner (yellow) and half of the outer sheet (red) at two singular points is depicted in the \b{figure below}.
      }
      \label{fig:Fresnel}
    \end{figure}   

It can be conceived as union of three components:
\begin{itemize}
\item a smooth and regular component with two principal curvatures bounded from below,
\item one-dimensional submanifolds with vanishing Gaussian curvature and one principal curvature,
\item neighbourhoods of (four) conical singularities.
\end{itemize}
This is classical and was already pointed out by Darboux \cite{Darboux1993}. The curvature was precisely quantified in \cite{MandelSchippa2022}. Dispersive properties in the constant-coefficient case were first analyzed by Liess \cite{Liess1991}. The conical singularities lead to the dispersive properties of the \emph{2d} wave equation. We prove Strichartz estimates for rough coefficients in the companion paper \cite{SchippaSchnaubeltFullyAnisotropic}.

\section{Strichartz estimates in the isotropic case}
\label{section:StrichartzEstimatesIsotropic}
The purpose of this section is to reduce \eqref{eq:Maxwell3d} to half-wave equations in the isotropic case. The Strichartz estimates then follow from Proposition \ref{prop:HalfWaveEstimates}. The key point is to diagonalize the principal symbol of
\begin{equation*}
P(x,D) = 
\begin{pmatrix}
\partial_t 1_{3 \times 3} & - \mathcal{C}(D) \mu^{-1} \\
\mathcal{C}(D) \varepsilon^{-1} & \partial_t 1_{3 \times 3}
\end{pmatrix}
.
\end{equation*}
The diagonalization argument follows the two-dimensional case, but is more involved. The eigenpairs in the partially anisotropic case had been computed in case of constant coefficients in \cite{Schippa2022}. This suffices for constant-coefficients, but for variable coefficients this diagonalization appears to lose regularity. However, in the isotropic case, we can find a regular diagonalization after an additional microlocalization.

Further reductions are standard, i.e., reduction to high frequencies and localization to a cube of size $1$, reduction to dyadic estimates, truncating frequencies of the coefficients. We start with diagonalizing the principal symbol:

\subsection{Diagonalizing the principal symbol in the isotropic case}
\label{subsection:IsotropicDiagonalization}
%\subsubsection{The isotropic case}
%\label{subsubsection:IsotropicCase}
We begin with the isotropic case $\varepsilon = \text{diag}(\varepsilon_1,\varepsilon_1,\varepsilon_1)$ and $\mu = \text{diag}(\mu_1,\mu_1,\mu_1)$. In the following we abuse notation and write $\varepsilon = \varepsilon_1$ and $\mu = \mu_1$ for sake of brevity.
In the isotropic case the diagonalization is as regular as in two dimensions after an additional localization in phase space. It turns out we have to distinguish one non-degenerate direction to find non-degenerate eigenvectors.

 We use the block matrix structure to find eigenvectors of $p/i$.
We have
\begin{equation*}
p/i = 
\begin{pmatrix}
\xi_0 & - \mathcal{C}(\xi') \mu^{-1} \\
\mathcal{C}(\xi') \varepsilon^{-1} & \xi_0
\end{pmatrix}
\text{ with } \mathcal{C}(\xi')_{ij} = - i \varepsilon_{ijk} \xi_k,
\end{equation*}
where $\varepsilon_{ijk}$ denotes the Levi-Civita symbol.

We find eigenvectors $v = (v_1,v_2) \in \C^3 \times \C^3$ by using the block matrix structure of $p/i$. Let $\lambda \in \R$ (note that $p/i$ is symmetric, which yields real eigenvalues) such that
\begin{equation*}
\begin{pmatrix}
\xi_0 & - \mathcal{C}(\xi') \mu^{-1} \\
\mathcal{C}(\xi') \varepsilon^{-1} & \xi_0
\end{pmatrix}
\begin{pmatrix}
v_1 \\ v_2
\end{pmatrix}
= \lambda
\begin{pmatrix}
v_1 \\ v_2
\end{pmatrix}
.
\end{equation*} 
Let $\lambda' = \lambda - \xi_0$. Then, we find the system of equations
\begin{align}
- \mathcal{C}(\xi') \mu^{-1} v_2 &= \lambda' v_1, \label{eq:EigenvectorA} \\
\mathcal{C}(\xi') \varepsilon^{-1} v_1 &= \lambda' v_2 \label{eq:EigenvectorB}.
\end{align}
In the following let $\xi' \neq 0$ because $p$ is already diagonal for $\xi' = 0$. Denote $\xi_i^* = \xi_i / |\xi'|$. Clearly, for $\lambda = \xi_0$ we find two eigenvectors
\begin{equation*}
\begin{pmatrix}
\xi_1^* & \xi_2^* & \xi_3^* & 0 & 0 & 0
\end{pmatrix}^t
, \qquad
\begin{pmatrix}
0 & 0 & 0 & \xi_1^* & \xi_2^* & \xi_3^*
\end{pmatrix}^t
\end{equation*}
because $\text{span} (\xi') = \ker(\mathcal{C}(\xi'))$. So, in the following suppose that $\lambda \neq \xi_0$, and let $\lambda' = \lambda - \xi_0$. Iterating \eqref{eq:EigenvectorA} and \eqref{eq:EigenvectorB}, we find that $v_1$ and $v_2$ solve the eigenpair equations:
\begin{align}
\label{eq:EigenvectorA1}
-\mathcal{C}(\xi') \mu^{-1} \mathcal{C}(\xi') \varepsilon^{-1} v_1 &= {\lambda'}^2 v_1, \\
-\mathcal{C}(\xi') \varepsilon^{-1} \mathcal{C}(\xi') \mu^{-1} v_2 &= {\lambda'}^2 v_2. \label{eq:EigenvectorB1}
\end{align}
Since $\varepsilon$ and $\mu$ are isotropic and elliptic, we can define
\begin{equation*}
\lambda^* = (\varepsilon \mu)^{\frac{1}{2}} \lambda',
\end{equation*} 
and write \eqref{eq:EigenvectorA1} and \eqref{eq:EigenvectorB1} as
\begin{equation*}
- \mathcal{C}^2(\xi') v_i = {\lambda^*}^2 v_i, \quad i =1,2.
\end{equation*}
We have $\mathcal{C}^2(\xi') = -\| \xi' \|^2 1_{3 \times 3} + \xi' (\xi')^t$. We note that $\langle v_i, \xi' \rangle = 0$, which follows from projecting \eqref{eq:EigenvectorA1} and \eqref{eq:EigenvectorB1} to $\xi'$ and supposing that $\lambda' \neq 0$. We obtain $\lambda^* \in \{ - \| \xi' \|; \| \xi' \| \}$.
We shall construct eigenvectors depending on the non-degenerate direction of $\xi^* = \frac{\xi'}{\| \xi' \|} \in \mathbb{S}^2$. Clearly, there is $i \in \{1,2,3\}$ such that $(\xi_i^*)^2 \geq \frac{1}{3}$. We introduce the notation $\xi_{ij}^2 = \xi_i^2 + \xi_j^2 $ for $i,j \in \{1,2,3\}$.

\medskip

\textbf{Eigenvectors for $|\xi_1^*| \gtrsim 1$:} 
We let
\begin{equation*}
v_1^{(1)} =
\begin{pmatrix}
\frac{\xi_2}{\xi_{12}} \\ - \frac{\xi_1}{\xi_{12}} \\ 0
\end{pmatrix}, \quad
v_2^{(1)} = \frac{\mathcal{C}(\xi') \varepsilon^{-1}}{\lambda'} v_1^{(1)} = \pm \big( \frac{\mu}{\varepsilon} \big)^{\frac{1}{2}} \begin{pmatrix} \frac{\xi_1 \xi_3}{\xi_{12} \| \xi' \|} \\ \frac{\xi_2 \xi_3}{\xi_{12} \| \xi' \|} \\ - \frac{\xi_{12}}{\| \xi' \|} \end{pmatrix}
.
\end{equation*}
The choice of $v_2^{(1)}$ satisfies \eqref{eq:EigenvectorB} and is orthogonal to $\xi^*$. \eqref{eq:EigenvectorA} is satisfied for
\begin{equation*}
\lambda' \in \{ \frac{\| \xi' \|}{(\varepsilon \mu)^{\frac{1}{2}}}, -\frac{\| \xi' \|}{(\varepsilon \mu)^{\frac{1}{2}}} \}.
\end{equation*}
Secondly, we let
\begin{equation*}
v_1^{(2)} =
\begin{pmatrix}
\frac{\xi_3}{\xi_{13}} \\ 0 \\ - \frac{\xi_1 }{\xi_{13}}
\end{pmatrix}
, \quad
v_2^{(2)} = \frac{\mathcal{C}(\xi') \varepsilon^{-1}}{\lambda'} v_1^{(2)} = \pm \big( \frac{\mu}{\varepsilon} \big)^{\frac{1}{2}} \begin{pmatrix}
\frac{- \xi_1 \xi_2}{\xi_{13} \| \xi' \|} \\ \frac{\xi_{13} }{\| \xi' \|} \\ \frac{- \xi_2 \xi_3} { \xi_{13} \| \xi' \| }
\end{pmatrix}
.
\end{equation*} 
$v_1^{(1)}$ and $v_1^{(2)}$ are linearly independent. For the diagonal matrix
\begin{equation}
\label{eq:DiagonalMatrix}
d(x,\xi) = i \text{diag}(\xi_0,\xi_0,\xi_0 - \frac{\| \xi' \|}{(\varepsilon \mu)^{\frac{1}{2}}}, \xi_0 + \frac{\| \xi' \|}{(\varepsilon \mu)^{\frac{1}{2}}}, \xi_0 - \frac{\| \xi' \|}{(\varepsilon \mu)^{\frac{1}{2}}}, \xi_0 + \frac{\| \xi' \|}{(\varepsilon \mu)^{\frac{1}{2}}})
\end{equation}
we have the conjugation matrix of eigenvectors:
\begin{equation*}
m^{(1)}(\xi) =
\begin{pmatrix}
\xi_1^* & 0 & \frac{\xi_2}{\xi_{12}} & \frac{\xi_2}{\xi_{12}} & \frac{\xi_3}{\xi_{13}} & \frac{\xi_3}{\xi_{13}} \\
\xi_2^* & 0 & - \frac{\xi_1}{\xi_{12}} & - \frac{\xi_1}{\xi_{12}} & 0 & 0 \\
\xi_3^* & 0 & 0							& 0							& - \frac{\xi_1}{\xi_{13}} &- \frac{\xi_1}{\xi_{13}} \\
0 & \xi_1^* & \big( \frac{\mu}{\varepsilon} \big)^{\frac{1}{2}} \frac{\xi_1 \xi_3}{\xi_{12} \| \xi' \|} & \big( \frac{\mu}{\varepsilon} \big)^{\frac{1}{2}} \big( \frac{- \xi_1 \xi_3}{\xi_{12} \| \xi' \|} \big) & \big( \frac{\mu}{\varepsilon} \big)^{\frac{1}{2}} \big( \frac{ - \xi_{1} \xi_2}{\xi_{13} \| \xi' \|} \big) & \big( \frac{\mu}{\varepsilon} \big)^{\frac{1}{2}} \frac{ \xi_{1} \xi_2}{\xi_{13} \| \xi' \|} \\
0 & \xi_2^* & \big( \frac{\mu}{\varepsilon} \big)^{\frac{1}{2}} \frac{\xi_2 \xi_3}{\xi_{12} \| \xi' \|} & \big( \frac{\mu}{\varepsilon} \big)^{\frac{1}{2}} \big( \frac{- \xi_2 \xi_3}{\xi_{12} \| \xi' \|} \big) & \big( \frac{\mu}{\varepsilon} \big)^{\frac{1}{2}} \frac{\xi_{13} }{\| \xi' \|} & \big( \frac{\mu}{\varepsilon} \big)^{\frac{1}{2}} \big( \frac{- \xi_{13} }{\| \xi' \|} \big) \\
0 & \xi_3^* & \big( \frac{\mu}{\varepsilon} \big)^{\frac{1}{2}} \big( \frac{- \xi_{12}}{\| \xi' \|} \big) & \big( \frac{\mu}{\varepsilon} \big)^{\frac{1}{2}} \frac{\xi_{12}}{\| \xi' \|} & \big( \frac{\mu}{\varepsilon} \big)^{\frac{1}{2}} \big( \frac{-\xi_2 \xi_3}{\xi_{13} \| \xi' \|} \big) & \big( \frac{\mu}{\varepsilon} \big)^{\frac{1}{2}} \frac{ \xi_2 \xi_3}{\xi_{13} \| \xi' \|}
\end{pmatrix}
.
\end{equation*}
By elimination and using the block matrix structure, the determinant is computed as
\begin{equation*}
|\det m^{(1)}(\xi)| = |  
\begin{vmatrix}
\xi_1^* & \frac{\xi_2}{\xi_{12}} & \frac{\xi_3}{\xi_{13}} \\
\xi_2^* & - \frac{\xi_1}{\xi_{12}} & 0 \\
\xi_3^* & 0 & - \frac{\xi_1}{\xi_{13}}
\end{vmatrix}
\begin{vmatrix}
\xi_1^* & \big( \frac{\mu}{\varepsilon} \big)^{\frac{1}{2}} \frac{\xi_1 \xi_3}{\xi_{12} \| \xi' \|} & \big( \frac{\mu}{\varepsilon} \big)^{\frac{1}{2}} \frac{- \xi_1 \xi_2}{\xi_{13} \| \xi' \|} \\
\xi_2^* & \big( \frac{\mu}{\varepsilon} \big)^{\frac{1}{2}} \frac{\xi_2 \xi_3}{\xi_{12} \| \xi' \|} & \big( \frac{\mu}{\varepsilon} \big)^{\frac{1}{2}} \frac{\xi_{13}}{\| \xi' \|} \\
\xi_3^* & \big( \frac{\mu}{\varepsilon} \big)^{\frac{1}{2}} \frac{-\xi_{12}}{\| \xi' \|} & \big( \frac{\mu}{\varepsilon} \big)^{\frac{1}{2}} \frac{- \xi_2 \xi_3}{\xi_{13} \| \xi' \|}
\end{vmatrix}
|
= M_1 \cdot M_2.
\end{equation*}
For the first determinant we find
\begin{equation*}
\begin{split}
M_1 &=
\begin{vmatrix}
\xi_1^* & \frac{\xi_2}{\xi_{12}} & \frac{\xi_3}{\xi_{13}} \\
\xi_2^* & - \frac{\xi_1}{\xi_{12}} & 0 \\
\xi_3^* & 0 & - \frac{\xi_1}{\xi_{13}}
\end{vmatrix} 
= \frac{1}{\| \xi' \| \xi_{12} \xi_{13}} \begin{vmatrix}
\xi_1 & \xi_2 & \xi_3 \\
\xi_2 & - \xi_1 & 0 \\
\xi_3 & 0 & - \xi_1
\end{vmatrix} \\
&= \frac{1}{\| \xi' \| \xi_{12} \xi_{13}} (\xi_1^3 + \xi_1 \xi_3^2 + \xi_1 \xi_2^2) 
= \frac{\xi_1 \| \xi' \|}{\xi_{12}	\xi_{13}}.
\end{split}
\end{equation*}
For the second determinant, we compute
\begin{equation*}
\begin{split}
M_2 &= 
\frac{\mu}{\varepsilon} \frac{1}{\| \xi' \|^3} 
\begin{vmatrix}
\xi_1 & \frac{\xi_1 \xi_3}{\xi_{12}} & - \frac{\xi_1 \xi_2}{\xi_{13}} \\
\xi_2 & \frac{\xi_2 \xi_3}{\xi_{12}} & \xi_{13} \\
\xi_3 & - \xi_{12} & - \xi_2 \xi_3
\end{vmatrix}
=
\frac{\mu}{\varepsilon} \frac{1}{\| \xi' \|^3 \xi_{12} \xi_{13}} 
\begin{vmatrix}
\xi_1 & 0 & 0 \\
\xi_2 & 0 & \| \xi' \|^2 \\
\xi_3 & - \| \xi' \|^2 & 0
\end{vmatrix}
= \frac{\mu}{\varepsilon}  \frac{\| \xi' \| \xi_1}{\xi_{12} \xi_{13}}.
\end{split}
\end{equation*}
The intermediate equation follows from multiplying the first column with $\xi_3$ and subtracting from the second and multiplying it with $\xi_2$ and adding it to the third column.
We observe that the inverse matrix takes the following form:
\begin{equation*}
{m^{(1)}}^{-1}(x,\xi) =
\begin{pmatrix}
\xi_1^* & \xi_2^* & \xi_3^* & 0 & 0 & 0 \\
0 & 0 & 0 & \xi_1^* & \xi_2^* & \xi_3^* \\
w^{(1)}_{31} & w^{(1)}_{32} & w^{(1)}_{33} & w^{(1)}_{34} & w^{(1)}_{35} & w^{(1)}_{36} \\
w^{(1)}_{41} & w^{(1)}_{42} & w^{(1)}_{43} & w^{(1)}_{44} & w^{(1)}_{45} & w^{(1)}_{46} \\
w^{(1)}_{51} & w^{(1)}_{52} & w^{(1)}_{53} & w^{(1)}_{54} & w^{(1)}_{55} & w^{(1)}_{56} \\
w^{(1)}_{61} & w^{(1)}_{62} & w^{(1)}_{63} & w^{(1)}_{64} & w^{(1)}_{65} & w^{(1)}_{66}
\end{pmatrix}
.
\end{equation*}
By Cramer's rule, the components $w^{(1)}_{ij}$, $3 \leq i \leq 6$, $1 \leq j \leq 6$ are polynomials in the entries of $m^{(1)}_{ij}$ up to the determinant. Hence, for $|\xi_1| \gtrsim 1$, the components of $m^{(1)}$ and $(m^{(1)})^{-1}$ are smooth and zero homogeneous.

\medskip

\textbf{Eigenvectors for $|\xi_2^*| \gtrsim 1$:}
We let like above
\begin{equation*}
v_1^{(1)} =
\begin{pmatrix}
\frac{\xi_2}{\xi_{12}} \\ - \frac{\xi_1}{\xi_{12}} \\ 0
\end{pmatrix}, \quad
v_2^{(1)} = \frac{\mathcal{C}(\xi') \varepsilon^{-1}}{\lambda'} v_1^{(1)} = \pm \big( \frac{\mu}{\varepsilon} \big)^{\frac{1}{2}} \begin{pmatrix} \frac{\xi_1 \xi_3}{\xi_{12} \| \xi' \|} \\ \frac{\xi_2 \xi_3}{\xi_{12} \| \xi' \|} \\ - \frac{\xi_{12}}{\| \xi' \|} \end{pmatrix}
,
\end{equation*}
and
\begin{equation*}
v_1^{(2)} =
\begin{pmatrix}
0 \\ \frac{\xi_3}{\xi_{23}} \\ - \frac{\xi_2}{\xi_{23}}
\end{pmatrix}
, \quad
v_2^{(2)} =
\frac{\mathcal{C}(\xi') \varepsilon^{-1}}{\lambda'} v_1^{(2)} = \pm \big( \frac{\mu}{\varepsilon} \big)^{\frac{1}{2}}
\begin{pmatrix}
- \frac{\xi_{23}}{\| \xi' \|} \\ \frac{\xi_1 \xi_2}{\xi_{23} \| \xi' \|} \\ \frac{\xi_1 \xi_3}{\xi_{23} \| \xi' \|}
\end{pmatrix}
.
\end{equation*}
For $d$ like in \eqref{eq:DiagonalMatrix} we obtain the conjugation matrix of eigenvectors:
\begin{equation*}
m^{(2)}(x,\xi) =
\begin{pmatrix}
\xi_1^* & 0 & \frac{\xi_2}{\xi_{12}} & \frac{\xi_2}{\xi_{12}} & 0 & 0 \\
\xi_2^* & 0 & - \frac{\xi_1}{\xi_{12}} & - \frac{\xi_1}{\xi_{12}} & \frac{\xi_3}{\xi_{23}} & \frac{\xi_3}{\xi_{23}} \\
\xi_3^* & 0 & 0							& 0							& - \frac{\xi_2}{\xi_{23}} &- \frac{\xi_2}{\xi_{23}} \\
0 & \xi_1^* & \big( \frac{\mu}{\varepsilon} \big)^{\frac{1}{2}} \frac{\xi_1 \xi_3}{\xi_{12} \| \xi' \|} & \big( \frac{\mu}{\varepsilon} \big)^{\frac{1}{2}} \big( \frac{- \xi_1 \xi_3}{\xi_{12} \| \xi' \|} \big) & \big( \frac{\mu}{\varepsilon} \big)^{\frac{1}{2}} \big( \frac{ - \xi_{23} }{\| \xi' \|} \big) & \big( \frac{\mu}{\varepsilon} \big)^{\frac{1}{2}} \frac{ \xi_{23} }{\| \xi' \|} \\
0 & \xi_2^* & \big( \frac{\mu}{\varepsilon} \big)^{\frac{1}{2}} \frac{\xi_2 \xi_3}{\xi_{12} \| \xi' \|} & \big( \frac{\mu}{\varepsilon} \big)^{\frac{1}{2}} \big( \frac{- \xi_2 \xi_3}{\xi_{12} \| \xi' \|} \big) & \big( \frac{\mu}{\varepsilon} \big)^{\frac{1}{2}} \frac{\xi_{1} \xi_2}{\xi_{23} \| \xi' \|} & \big( \frac{\mu}{\varepsilon} \big)^{\frac{1}{2}} \frac{\xi_{1} \xi_2}{\xi_{23} \| \xi' \|} \\
0 & \xi_3^* & \big( \frac{\mu}{\varepsilon} \big)^{\frac{1}{2}} \big( \frac{- \xi_{12}}{\| \xi' \|} \big) & \big( \frac{\mu}{\varepsilon} \big)^{\frac{1}{2}} \frac{\xi_{12}}{\| \xi' \|} & \big( \frac{\mu}{\varepsilon} \big)^{\frac{1}{2}} \frac{\xi_1 \xi_3}{\xi_{23} \| \xi' \|} & \big( \frac{\mu}{\varepsilon} \big)^{\frac{1}{2}} \big( \frac{-\xi_1 \xi_3}{\xi_{23} \| \xi' \|} \big)
\end{pmatrix}
.
\end{equation*}
Like above we compute the determinant by using the block matrix structure:
\begin{equation*}
\det m(x,\xi) =
\begin{vmatrix}
\xi_1^* & \frac{\xi_2}{\xi_{12}} & 0 \\
\xi_2^* & - \frac{\xi_1}{\xi_{12}} & \frac{\xi_3}{\xi_{23}} \\
\xi_3^* & 0 & - \frac{\xi_2}{\xi_{23}}
\end{vmatrix}
\begin{vmatrix}
\xi_1^* & \big( \frac{\mu}{\varepsilon} \big)^{\frac{1}{2}} \frac{\xi_1 \xi_3}{\xi_{12} \| \xi' \|} & - \big( \frac{\mu}{\varepsilon} \big)^{\frac{1}{2}} \frac{\xi_{23}}{\| \xi' \|} \\
\xi_2^* & \big( \frac{\mu}{\varepsilon} \big)^{\frac{1}{2}} \frac{\xi_2 \xi_3}{\xi_{12} \| \xi' \|} & \big( \frac{\mu}{\varepsilon} \big)^{\frac{1}{2}} \frac{\xi_1 \xi_2}{\xi_{23} \| \xi' \|} \\
\xi_3^* & - \big( \frac{\mu}{\varepsilon} \big)^{\frac{1}{2}} \frac{\xi_{12}}{\| \xi' \|} & \big( \frac{\mu}{\varepsilon} \big)^{\frac{1}{2}} \frac{\xi_1 \xi_3}{\xi_{23} \| \xi' \|}
\end{vmatrix}
= M_1 \cdot M_2
.
\end{equation*}
The first determinant is computed to be
\begin{equation*}
M_1 = \frac{\xi_2}{\| \xi' \| \xi_{12} \xi_{23}} (\xi_1^2 + \xi_2^2 + \xi_3^2) = \frac{\xi_2 \| \xi' \|}{\xi_{12} \xi_{23}}.
\end{equation*}
We find for the second determinant:
\begin{equation*}
M_2 = \frac{\mu}{\varepsilon \| \xi' \|^3 \xi_{12} \xi_{23}}
\begin{vmatrix}
\xi_1 & \xi_1 \xi_3 & - \xi_{23}^2 \\
\xi_2 & \xi_2 \xi_3 & \xi_1 \xi_2 \\
\xi_3 & - \xi_{12}^2 & \xi_1 \xi_3
\end{vmatrix}
= \frac{\mu}{\varepsilon \| \xi' \|^3 \xi_{12} \xi_{23}}
\begin{vmatrix}
\xi_1 & 0 & - \| \xi' \|^2 \\
\xi_2 & 0 & 0 \\
\xi_3 & -\| \xi' \|^2 & 0
\end{vmatrix}
= \frac{\mu}{\varepsilon} \frac{\xi_2 \| \xi' \|}{\xi_{12} \xi_{23}}.
\end{equation*}
%${m^{(2)}}^{-1}$ takes the same form like in \eqref{eq:InverseConjugationMatrix}.

\medskip

\textbf{Eigenvectors for $|\xi_3^*| \gtrsim 1$:}
We choose
\begin{equation*}
v_1^{(1)} =
\begin{pmatrix}
\frac{\xi_3}{\xi_{13}} \\ 0 \\ -\frac{\xi_1}{\xi_{13}}
\end{pmatrix}
, \quad v_2^{(1)} = 
\frac{\mathcal{C}(\xi') \varepsilon^{-1}}{\lambda'} v_1^{(1)} = \pm \big( \frac{\mu}{\varepsilon} \big)^{\frac{1}{2}} 
\begin{pmatrix}
- \frac{\xi_1 \xi_2}{\xi_{13} \| \xi' \|} \\ \frac{\xi_{13}}{\| \xi' \|} \\ \frac{- \xi_2 \xi_3}{\xi_{13} \| \xi' \|}
\end{pmatrix}
\end{equation*}
and let
\begin{equation*}
v_1^{(2)} =
\begin{pmatrix}
 0 \\ \frac{\xi_3}{\xi_{23}} \\ - \frac{\xi_2}{\xi_{23}}
\end{pmatrix}, \quad
v_2^{(2)} = \frac{\mathcal{C}(\xi') \varepsilon^{-1}}{\lambda'} v_1^{(2)} = \pm \big( \frac{\mu}{\varepsilon} \big)^{\frac{1}{2}} 
\begin{pmatrix}
\frac{- \xi_{23}}{\| \xi' \|} \\ \frac{\xi_1 \xi_2}{\xi_{23} \| \xi' \|} \\ \frac{\xi_1 \xi_3}{\xi_{23} \| \xi' \|}
\end{pmatrix}
.
\end{equation*}
The conjugation matrix of eigenvectors for $d$ like in \eqref{eq:DiagonalMatrix} is given by
\begin{equation*}
m^{(3)}(\xi) = 
\begin{pmatrix}
\xi_1^* & 0 & \frac{\xi_3}{\xi_{13}} & \frac{\xi_3}{\xi_{13}} & 0 & 0 \\
\xi_2^* & 0 & 0 & 0 & \frac{\xi_3}{\xi_{23}} & \frac{\xi_3}{\xi_{23}} \\
\xi_3^* & 0 & - \frac{\xi_1}{\xi_{13}} & - \frac{\xi_1}{\xi_{13}} & - \frac{\xi_2}{\xi_{23}} & - \frac{\xi_2}{\xi_{23}} \\
0 & \xi_1^* & \big( \frac{\mu}{\varepsilon} \big)^{\frac{1}{2}} \big( \frac{- \xi_1 \xi_2}{\xi_{13} \| \xi' \|} \big) & \big( \frac{\mu}{\varepsilon} \big)^{\frac{1}{2}} \frac{\xi_1 \xi_2}{\xi_{13} \| \xi' \|} & \big( \frac{\mu}{\varepsilon} \big)^{\frac{1}{2}} \big( \frac{-\xi_{23}}{\| \xi' \|} \big) & \big( \frac{\mu}{\varepsilon} \big)^{\frac{1}{2}} \big( \frac{\xi_{23 }}{ \| \xi' \|} \big) \\
0 & \xi_2^* & \big( \frac{\mu}{\varepsilon} \big)^{\frac{1}{2}} \frac{\xi_{13}}{\| \xi' \|} & \big( \frac{\mu}{\varepsilon} \big)^{\frac{1}{2}} \big( \frac{- \xi_{13}}{\| \xi' \|} \big) & \big( \frac{\mu}{\varepsilon} \big)^{\frac{1}{2}} \frac{\xi_1 \xi_2}{\xi_{23} \| \xi' \|} & \big( \frac{\mu}{\varepsilon} \big)^{\frac{1}{2}} \big( \frac{- \xi_1 \xi_2}{\xi_{23} \| \xi' \|} \big) \\
0 & \xi_3^* & \big( \frac{\mu}{\varepsilon} \big)^{\frac{1}{2}} \big( \frac{- \xi_2 \xi_3}{\xi_{13} \| \xi' \|} \big) & \big( \frac{\mu}{\varepsilon} \big)^{\frac{1}{2}} \frac{\xi_2 \xi_3}{\xi_{13} \| \xi' \|} & \big( \frac{\mu}{\varepsilon} \big)^{\frac{1}{2}} \frac{\xi_1 \xi_3}{\xi_{23} \| \xi' \|} & \big( \frac{\mu}{\varepsilon} \big)^{\frac{1}{2}} \big( \frac{- \xi_1 \xi_3}{\xi_{23} \| \xi' \|} \big)
\end{pmatrix}
.
\end{equation*}
In this case we compute the determinant to be
\begin{equation*}
| \det m(\xi) | =
\begin{vmatrix}
\xi_1^* & \frac{\xi_3}{\xi_{13}} & 0 \\
\xi_2^* & 0 & \frac{\xi_3}{\xi_{23}} \\
\xi_3^* & -\frac{\xi_1}{\xi_{13}} & - \frac{\xi_2}{\xi_{23}}
\end{vmatrix}
\begin{vmatrix}
\xi_1^* & \big( \frac{\mu}{\varepsilon} \big)^{\frac{1}{2}} \big( \frac{- \xi_1 \xi_2}{\xi_{13} \| \xi' \|} \big) & \big( \frac{\mu}{\varepsilon} \big)^{\frac{1}{2}} \big( \frac{-\xi_{23}}{\| \xi' \|} \big) \\
\xi_2^* & \big( \frac{\mu}{\varepsilon} \big)^{\frac{1}{2}} \frac{\xi_{13}}{\| \xi' \|} & \big( \frac{\mu}{\varepsilon} \big)^{\frac{1}{2}} \frac{\xi_1 \xi_2}{\xi_{23} \| \xi' \|} \\
\xi_3^* & \big( \frac{\mu}{\varepsilon} \big)^{\frac{1}{2}} \big( \frac{- \xi_2 \xi_3}{\xi_{13} \| \xi' \|} \big) & \big( \frac{\mu}{\varepsilon} \big)^{\frac{1}{2}} \frac{\xi_1 \xi_3}{\xi_{23} \| \xi' \|}
\end{vmatrix}
.
\end{equation*}
For the first determinant we find
\begin{equation*}
M_1 = \frac{\xi_3^3}{\xi_{13} \xi_{23} \| \xi' \|} + \frac{\xi_1^2 \xi_3}{\xi_{13} \xi_{23} \| \xi' \|} + \frac{\xi_2^2 \xi_3}{\xi_{13} \xi_{23} \| \xi' \|} = \frac{\xi_3 \| \xi' \|}{\xi_{13} \xi_{23}}.
\end{equation*}
For the second determinant we compute
\begin{equation*}
M_2 = \frac{\mu}{\varepsilon} \frac{1}{\| \xi' \|^3 \xi_{13} \xi_{23}} 
\begin{vmatrix}
\xi_1 & - \xi_1 \xi_2 & -\xi_{23}^2 \\
\xi_2 & \xi_{13}^2 & \xi_1 \xi_2 \\
\xi_3 & - \xi_2 \xi_3 & \xi_1 \xi_3
\end{vmatrix}
= \frac{\mu}{\varepsilon} \frac{\xi_3 \| \xi' \|}{\xi_{13} \xi_{23}}.
\end{equation*}
For $i=1,2,3$, we summarize that ${m^{(i)}}^{-1}$ takes the form
\begin{equation}
\label{eq:InverseConjugationMatrix}
{m^{(i)}}^{-1}(x,\xi) =
\begin{pmatrix}
\xi_1^* & \xi_2^* & \xi_3^* & 0 & 0 & 0 \\
0 & 0 & 0 & \xi_1^* & \xi_2^* & \xi_3^* \\
w^{(i)}_{31} & w^{(i)}_{32} & w^{(i)}_{33} & w^{(i)}_{34} & w^{(i)}_{35} & w^{(i)}_{36} \\
w^{(i)}_{41} & w^{(i)}_{42} & w^{(i)}_{43} & w^{(i)}_{44} & w^{(i)}_{45} & w^{(i)}_{46} \\
w^{(i)}_{51} & w^{(i)}_{52} & w^{(i)}_{53} & w^{(i)}_{54} & w^{(i)}_{55} & w^{(i)}_{56} \\
w^{(i)}_{61} & w^{(i)}_{62} & w^{(i)}_{63} & w^{(i)}_{64} & w^{(i)}_{65} & w^{(i)}_{66}
\end{pmatrix}
\end{equation}
with $w^{(i)}_{mn}$ zero-homogeneous and smooth in $\xi'$ for $|\xi_i| \gtrsim 1$.

\subsection{Reductions for $C^2$-coefficients}
\label{subsection:ReductionsC2Isotropic}
Next, we carry out reductions as in \cite{SchippaSchnaubelt2021} for the proof of Theorem \ref{thm:IsotropicStrichartz}. Precisely, we apply the following:
\begin{itemize}
\item Reduction to high frequencies and localization to a cube of size $1$,
\item Reduction to dyadic estimates,
\item Truncating the coefficients of $P$ at frequency $\lambda^{\frac{1}{2}}$,
\item Reduction to half-wave equations.
\end{itemize}
To begin with, by scaling we suppose that $\| \partial_x^2 \varepsilon \|_{L^\infty} \leq 1$, $\| \partial_x^2 \mu \|_{L^\infty} \leq 1$, and $\nu = 1$. Note that the ellipticity constraint \eqref{eq:UniformEllipticity} implies by the Gagliardo--Nirenberg inequality
\begin{equation*}
\| \partial_x \varepsilon \|_{L^\infty} + \| \partial_x \mu \|_{L^\infty} \lesssim 1.
\end{equation*}
\subsubsection{Reduction to high frequencies and localization to a cube of size $1$}
Let $\beta \in C^\infty_c$ like in \eqref{eq:LPBeta}, and let $s(\xi) = \beta(\| \xi \|)$ denote a symbol supported in $B(0,2) \backslash B(0,1/2)$ such that
\begin{equation*}
\sum_{j \in \Z} s(2^{-j} \xi) = 1, \qquad \xi \in \R^4 \backslash \{ 0 \}.
\end{equation*}
For $\lambda \in 2^{\mathbb{N}_0}$, let $S_\lambda = S(D/\lambda)$ be the Littlewood-Paley multiplier and $S_{\lesssim 1} = 1 - \sum_{j \geq 0} S_{2^j}$. Let $u = S_{\lesssim 1} u + (1-S_{\lesssim 1}) u$. 
We estimate the low frequencies as follows:
Write
\begin{equation*}
S_{\lesssim 1} = \sum_{K, L \leq 8} S^\tau_L S^{\xi'}_K S_{\lesssim 1}
\end{equation*}
with $S^\tau_M$, $S^{\xi'}_N$ denoting Littlewood-Paley projectors only in $\tau$ or $\xi'$. We use Bernstein's and Minkowski's inequality to find
\begin{equation*}
\| |D|^{-\rho} S_{\lesssim 1} u \|_{L^p L^q}^2 \leq \sum_{K, L \leq 8} \| |D|^{-\rho} S^\tau_L S^{\xi'}_K u \|_{L^p L^q}^2.
\end{equation*}
For $K \leq L$, we have by Bernstein's inequality and Plancherel's theorem
\begin{equation*}
\begin{split}
\sum_{K \leq L \leq 8} \| |D|^{-\rho} S^\tau_L S^{\xi'}_K u \|^2_{L^p L^q} &\lesssim \sum_{K \leq L \leq 1} L^{-2 \rho - \frac{1}{2}} \| S^\tau_L S^{\xi'}_K u \|^2_{L^p L^q} \\
&\lesssim \sum_{K \leq L \leq 8} L^{-2 \rho - \frac{1}{2}} L^{2 \big(\frac{1}{2}-\frac{1}{p} \big)} K^{6 \big( \frac{1}{2} - \frac{1}{q} \big)} \| S^\tau_L S^{\xi'}_K u \|^2_{L^2} \\
&\lesssim \| u \|_{L^2}^2.
\end{split}
\end{equation*}
The estimate for $L \leq K$ follows \emph{mutatis mutandis}.

\medskip

It remains to prove the claim for the inhomogeneous norm for the high frequencies:
\begin{equation*}
\| \langle D \rangle^{-\rho} u \|_{L^p L^q} \lesssim \| u \|_{L^2} + \| P u \|_{L^2} + \| \langle D' \rangle^{-\frac{1}{2}} \rho_{e m} \|_{L^2}
\end{equation*}
with $\langle D \rangle = OP((1+\|\xi\|^2)^{\frac{1}{2}})$ and $\langle D' \rangle = OP((1+\| \xi' \|^2)^{\frac{1}{2}})$. To localize $u$ to the unit cube, we introduce a smooth partition of unity in space-time:
\begin{equation*}
 1 = \sum_{j \in \mathbb{Z}^{4}} \chi_j(x), \quad \chi_j(x) = \chi(x-j), \quad \text{supp} \chi \subseteq B(0,2).
\end{equation*}
Let 
\begin{align*}
\rho_{ej} = \partial_1 (\chi_j u_1) + \partial_2 (\chi_j u_2) + \partial_3 (\chi_j u_3), \quad
\rho_{mj} = \partial_1 (\chi_j u_4) + \partial_2 (\chi_j u_5) + \partial_3 (\chi_j u_6).
\end{align*}
By commutator estimates, we find
\begin{equation*}
\sum_j \| \chi_j u \|^2_{L^2} + \| P (\chi_j u) \|_{L^2}^2 \lesssim \| u \|^2_{L^2} + \| P u \|^2_{L^2}.
\end{equation*}
Moreover, as proved in \cite[Eq.~(36),~(37)]{SchippaSchnaubelt2021}, we have
\begin{align*}
\| \langle D \rangle^{-\rho} u \|^2_{L^p L^q} &\lesssim \sum_j \| \langle D \rangle^{-\rho} \chi_j u \|^2_{L^p L^q}, \\
\sum_j \| \langle D \rangle^{-\frac{1}{2}} (\rho_{ej}, \rho_{mj}) \|_{L^2}^2 &\lesssim \| \langle D \rangle^{-\frac{1}{2}} \rho_{em} \|_{L^2}^2.
\end{align*}
This concludes the reduction to $u$ being supported in the unit cube.

\subsubsection{Reduction to dyadic estimates}
We shall see that it is enough to prove
\begin{equation}
\label{eq:FrequencyLocalizedStrichartz}
\lambda^{-\rho} \| S_\lambda u \|_{L^p L^q} \lesssim \| S_\lambda u \|_{L^2} + \| P S_\lambda u \|_{L^2} + \lambda^{-\frac{1}{2}} \| S_\lambda \rho_{e m} \|_{L^2}.
\end{equation}
We can assume that $2 \leq p,q < \infty$ because it is enough to prove the claim for sharp Strichartz exponents. The point $(p,q)=(\infty,2)$ is covered by the energy estimate.

By Littlewood-Paley theory (here we use that $2 \leq p,q < \infty$), we can estimate
\begin{equation*}
\| u \|_{L^p L^q} \lesssim \big\| \big( \sum_{\lambda \in 2^{\N_0}} |S_\lambda u |^2 \big)^{\frac{1}{2}} \big\|_{L^p L^q} \lesssim \big( \sum_{\lambda \in 2^{\N_0}} \| S_\lambda u \|^2_{L^p L^q} \big)^{\frac{1}{2}}.
\end{equation*}
To carry out the square sum over the right hand-side, we require the commutator estimate
\begin{equation}
\label{eq:CommutatorEstimate}
\big( \sum_{\lambda \in 2^{\N_0}} \| [P, S_\lambda] u \|_{L^2_x}^2 \big)^{\frac{1}{2}} \lesssim \| u \|_{L^2}.
\end{equation}
The square sum over the remaining terms is straightforward. \eqref{eq:CommutatorEstimate} is proved on \cite[p.~21]{SchippaSchnaubelt2021}.

\subsubsection{Truncating the coefficients of $P$ at frequency $\lambda^{\frac{1}{2}}$}
Finally, we reduce \eqref{eq:FrequencyLocalizedStrichartz} to $\varepsilon$ and $\mu$ having Fourier transform supported in $\{ |\xi| \leq \lambda^{\frac{1}{2}} \}$. Note that for $\lambda \gg 1$, $\varepsilon_{\lambda^{\frac{1}{2}}}$, $\mu_{\lambda^{\frac{1}{2}}}$ denoting the Fourier truncated coefficients, is still uniformly elliptic because
\begin{equation*}
\| \varepsilon - \varepsilon_{\lambda^{\frac{1}{2}}} \|_{L^\infty} \lesssim \lambda^{-1} \| \partial \varepsilon_{ \geq \lambda^{\frac{1}{2}}} \|_{L^\infty}.
\end{equation*}

It is enough to show
\begin{equation}
\label{eq:FrequencyLocalizedTruncatedStrichartz}
\lambda^{-\rho} \| S_\lambda u \|_{L^p L^q} \lesssim \| S_\lambda u \|_{L^2} + \| P_\lambda S_\lambda u \|_{L^2} + \lambda^{-\frac{1}{2}} \| S_\lambda \rho_{e m} \|_{L^2},
\end{equation}
where\footnote{We first truncate the frequencies and then take the inverse.}
\begin{equation}
\label{eq:TruncatedMaxwell}
P_\lambda = \begin{pmatrix}
\partial_t 1_{3 \times 3} & - \mathcal{C}(D) \mu^{-1}_{\lambda^{\frac{1}{2}}} \\
\mathcal{C}(D) \varepsilon_{\lambda^{\frac{1}{2}}}^{-1} & \partial_t 1_{3 \times 3}
\end{pmatrix}
.
\end{equation}
Let $u^{(1)} = (u_1,u_2,u_3)$ and $u^{(2)} = (u_4,u_5,u_6)$. The error estimate follows from
\begin{equation*}
\begin{split}
&\quad \| (P - P_\lambda) S_\lambda u \|_{L^2} \\
 &\leq \| \nabla \times ((\mu^{-1} - (\mu_{\leq \lambda^{\frac{1}{2}}})^{-1}) S_\lambda u^{(2)}) \|_{L^2} + \| \nabla \times ((\varepsilon^{-1} - (\varepsilon_{\leq \lambda^{\frac{1}{2}}})^{-1}) S_\lambda u^{(1)}) \|_{L^2} \\
&\lesssim (\| \partial_x \mu \|_{L^\infty} + \| \partial_x \varepsilon \|_{L^\infty}) \| S_\lambda u \|_{L^2} + \| (\mu^{-1} - (\mu_{\leq \lambda^{\frac{1}{2}}})^{-1}) \nabla \times S_\lambda u^{(1)} \|_{L^2} \\
&\qquad + \| (\varepsilon^{-1} - (\varepsilon_{\leq \lambda^{\frac{1}{2}}})^{-1})) \nabla \times S_\lambda u^{(2)} \|_{L^2} \\
&\lesssim \| S_\lambda u \|_{L^2} + \lambda ( \| \varepsilon_{\geq \lambda^{\frac{1}{2}}} \|_{L^\infty} + \| \mu_{\geq \lambda^{\frac{1}{2}}} \|_{L^\infty} ) \| S_\lambda u \|_{L^2} \\
&\lesssim (1+ \| \partial_x^2 \varepsilon \|_{L^\infty} + \| \partial_x^2 \mu \|_{L^\infty}) \| S_\lambda u \|_{L^2}.
\end{split}
\end{equation*}
We used that
\begin{equation*}
\| \varepsilon^{-1} - (\varepsilon_{\leq \lambda^{\frac{1}{2}}})^{-1} \|_{L^\infty} \leq \frac{\| \varepsilon - \varepsilon_{\leq \lambda^{\frac{1}{2}}} \|_{L^\infty}}{ \inf_{x \in \R^4} | \varepsilon(x) \varepsilon_{\leq \lambda^{\frac{1}{2}}(x)} |} \lesssim 
\| \varepsilon_{\geq \lambda^{\frac{1}{2}}} \|_{L^\infty}
\end{equation*}
and
\begin{equation*}
\lambda \| \varepsilon_{\geq \lambda^{\frac{1}{2}}} \|_{L^\infty} \lesssim \| \partial_x^2 \varepsilon \|_{L^\infty}.
\end{equation*}
The corresponding estimates also hold for $\mu$.

\subsubsection{Diagonalizing the Maxwell operator}
After truncating the coefficients, we obtain
\begin{equation*}
p(x,\xi) \chi_\lambda(\xi) = i
\begin{pmatrix}
\xi_0 & -\mathcal{C}(\xi') \mu_{\leq \lambda^{\frac{1}{2}}}^{-1} \\
\mathcal{C}(\xi') \varepsilon_{\leq \lambda^{\frac{1}{2}}}^{-1} & \xi_0
\end{pmatrix}
s(\xi / \lambda) \in S^1_{1,\frac{1}{2}}.
\end{equation*}
By microlocal analysis, we extend the formal diagonalization from Section \ref{subsection:IsotropicDiagonalization} to pseudo-differential operators diagonalizing the symbol.
\begin{proposition}
\label{prop:Diagonalization}
Let $\varepsilon, \mu \in C^1$. For $\lambda \gg 1$, there are operators $\mathcal{M}_\lambda^{(i)}$, $\mathcal{N}_\lambda^{(i)}$, $\mathcal{D}_\lambda$ and $S_{\lambda i}$ for $i \in \{1,2,3\}$ such that $S_\lambda S'_\lambda = S_{\lambda 1} + S_{\lambda 2} + S_{\lambda 3}$ and
\begin{equation*}
P_\lambda S_{\lambda i} = \mathcal{M}_\lambda^{(i)} \mathcal{D}_\lambda \mathcal{N}_\lambda^{(i)} S_{\lambda i } + E_\lambda^{(i)}
\end{equation*}
for $i \in \{1,2,3\}$ with $\| E_\lambda^{(i)} \|_{L^2 \to L^2} \lesssim 1$.
\end{proposition}
\begin{proof}
We quantize the diagonalization carried out in Subsection \ref{subsection:IsotropicDiagonalization}. To this end, we decompose $\mathbb{S}^2$ with a smooth partition of unity
\begin{equation*}
1 = \sum_{i=1}^3 s_i(\theta), \qquad s_i \in C^\infty_c(\mathbb{S}^2;\R_{\geq 0})
\end{equation*}
such that $s_i(\theta) = 1$ for $|\theta_i| \gtrsim 1$. Note that $s_{\lambda i}(\xi) = s(\xi/\lambda) s_i(\xi'/\| \xi' \|) \beta(\| \xi' \|/\lambda) \in S^0_{1,0}$. We let $S_{\lambda i} = OP(s_{\lambda i}(\xi))$ and note that $m_i^{-1}(x,\xi) s_{\lambda i}(\xi) \in S^0_{1,\frac{1}{2}}$, $m_i(x,\xi) s_{\lambda i}(\xi) \in S^0_{1,\frac{1}{2}}$, $d(x,\xi) s_\lambda(\xi) \in S^1_{1,\frac{1}{2}}$.

 We quantize
\begin{equation*}
\begin{split}
\mathcal{N}^{(i)}(x,D) &= OP(m^{-1}_i(x,\xi) s_{\lambda i}(\xi)), \quad \mathcal{M}^{(i)}(x,D) = OP(m_i(x,\xi) s_{\lambda i}(\xi)), \\ \mathcal{D}(x,D) &= OP(i \xi_0, i \xi_0, i \xi_0 + \frac{\| \xi' \|}{(\varepsilon \mu)^{\frac{1}{2}}}, i \xi_0 - \frac{\| \xi' \|}{(\varepsilon \mu)^{\frac{1}{2}}}, i \xi_0 + \frac{\| \xi' \|}{(\varepsilon \mu)^{\frac{1}{2}}}, i \xi_0 - \frac{\| \xi' \|}{(\varepsilon \mu)^{\frac{1}{2}}}).
\end{split}
\end{equation*}
Symbol composition holds by Theorem \ref{thm:KohnNirenberg} and the asymptotic expansion gives
\begin{equation*}
OP(m_i^{-1}(x,\xi) s_{\lambda i}(\xi)) = \tilde{S}_{\lambda i} OP (m_i^{-1}(x,\xi) s_{\lambda i}(\xi)) + O_{L^2} (\lambda^{-\infty}),
\end{equation*}
where $\tilde{S}_{\lambda i} = OP( \tilde{s}_{\lambda i}(\xi))$ and $\tilde{s}_{\lambda i}$ denotes a function like $s_{\lambda i}(\xi)$ with mildly enlarged support. Likewise we find that $\mathcal{M}_\lambda$ and $\mathcal{D}_\lambda$ do not significantly change frequency localization. Hence, to compute compositions, we can harmlessly insert additional frequency projections
\begin{equation*}
\mathcal{M}_\lambda^{(i)} \tilde{S}_{\lambda i} \mathcal{D}_\lambda \tilde{S}_{\lambda i} \mathcal{N}^{(i)}_\lambda S_{\lambda i } + O_{L^2}(\lambda^{-\infty}).
\end{equation*}
Now we use Theorem \ref{thm:KohnNirenberg} to argue that 
\begin{equation*}
\mathcal{M}_\lambda^{(i)} \tilde{S}_{\lambda i } \mathcal{D}_\lambda \tilde{S}_{\lambda i } \mathcal{N}_\lambda^{(i)} S_{\lambda i} + E_\lambda
\end{equation*}
with $E_\lambda = OP(e_\lambda)$ with $e_\lambda \in S^0_{1,\frac{1}{2}}$. The reason for $e_\lambda$ being better than suggested by Theorem \ref{thm:KohnNirenberg} is that the coefficients are $C^1$. Hence, inspecting the asymptotic expansion from Theorem \ref{thm:KohnNirenberg} reveals that the leading order term is in $S^0_{1,\frac{1}{2}}$. By \cite[Theorem~6.3]{Taylor1974} this is $L^2$-bounded.
\end{proof}

\subsubsection{Reduction to half-wave equations}
We consider the two regions $\{| \xi_0 | \gg \|\xi' \| \}$ and $\{ |\xi_0| \lesssim \|\xi' \| \}$. The first region is away from the characteristic surface. Hence, $P$ is elliptic in this region. The contribution can be estimated by Sobolev embedding. To make the argument precise, we use the FBI transform. By applying Theorem \ref{thm:ApproximationTheorem}, we find
\begin{equation*}
\| T_\lambda \big( \frac{P(x,D)}{\lambda} S_\lambda u \big) - p_\lambda(x,\xi) T_\lambda S_\lambda u \|_{L^2_\Phi} \lesssim \lambda^{-\frac{1}{2}} \| S_\lambda u \|_{L^2}.
\end{equation*}
Denote $v_\lambda = T_\lambda S_\lambda u$, and we observe for $|\xi_0| \gg \| \xi' \|$ that
\begin{equation*}
\| p(x,\xi) v_\lambda \|_{L^2_\Phi} \gtrsim \| v_\lambda \|_{L^2_\Phi}.
\end{equation*}
This is argued as follows. Write $v_\lambda = (v_1,v_2)$. Indeed, for $\| v_1 \|_{L^2_\Phi} \gtrsim \| v_2 \|_{L^2_\Phi}$, we find for some $c_0 \ll 1$
\begin{equation*}
\| \xi_0 v_1 - \mathcal{C}(\xi') \mu^{-1} v_2 \|_{L^2_\Phi} \geq \xi_0 ( \| v_1 \|_{L^2_\Phi} - c_0 \| v_2 \|_{L^2_\Phi} ) \gtrsim \| v_1 \|_{L^2_\Phi} \gtrsim \| v \|_{L^2_\Phi}.
\end{equation*}
If $\| v_2 \|_{L^2_\Phi} \gtrsim \| v_1 \|_{L^2_\Phi}$, then
\begin{equation*}
\| \xi_0 v_2 + \mathcal{C}(\xi') \varepsilon^{-1} v_1 \|_{L^2_\Phi} \geq \xi_0 ( \| v_2 \|_{L^2_\Phi} - c_0 \| v_1 \|_{L^2_\Phi} ) \gtrsim \| v_2 \|_{L^2_\Phi} \gtrsim \| v \|_{L^2_\Phi}.
\end{equation*}
Let $S_\lambda \tilde{u}$ denote the part of $S_\lambda u$ with Fourier transform in $\{ |\xi_0| \gg \| \xi' \| \}$.
By non-stationary phase, $T_\lambda S_\lambda \tilde{u}$ is essentially supported in $\{1 \sim |\xi_0| \gg \| \xi' \| \}$ up to arbitrary high gain of derivatives. We can write $T_\lambda S_\lambda \tilde{u} = p_\lambda^{-1}(x,\xi) p_\lambda(x,\xi) T_\lambda S_\lambda \tilde{u}$ because $p_\lambda(x,\xi)$ is invertible in $\{1 \sim |\xi_0| \gg \| \xi' \|\}$
 and conclude by $L^2$-$L^2_\Phi$-isometry of $T_\lambda$ and Theorem \ref{thm:ApproximationTheorem}:
\begin{equation*}
\begin{split}
\| T_\lambda S_\lambda \tilde{u} \|_{L^2_\Phi} &\lesssim \| p_\lambda(x,\xi) T_\lambda S_\lambda \tilde{u} \|_{L^2_\Phi} \\
&\lesssim \big\| \big( T_\lambda \frac{P_\lambda(x,D)}{\lambda} - p_\lambda(x,\xi) T_\lambda \big) S_\lambda \tilde{u} \|_{L^2_\Phi} + \big\| \frac{T_\lambda P_\lambda(x,D)}{\lambda} S_\lambda \tilde{u} \big\|_{L^2_\Phi} \\
&\lesssim \lambda^{-\frac{1}{2}} \| S_\lambda \tilde{u} \|_{L^2} + \lambda^{-1} \| P_\lambda(x,D) S_\lambda \tilde{u} \|_{L^2}
.
\end{split}
\end{equation*}

We handle the main contribution coming from $\{|\xi_0| \lesssim \| \xi' \| \}$ by invoking the diagonalization from Proposition \ref{prop:Diagonalization}. In the following assume that the space-time Fourier transform of $u$ is supported in $\{|\xi_0| \lesssim \| \xi' \| \}$. We treat the regions $\{ |\xi_i^* | \gtrsim 1 \}$ separately. We start with the proof of
\begin{equation}
\label{eq:DiagonalStrichartzEstimate}
\lambda^{-\rho} \| S_\lambda w \|_{L_t^p L_{x'}^q} \lesssim \| S_\lambda u \|_{L_x^2} + \| \mathcal{D}_\lambda S_\lambda w \|_{L_x^2} + \lambda^{-\frac{1}{2}} \| S_\lambda \rho_{e m} \|_{L^2},
\end{equation}
where\footnote{We omit the frequency truncation in the coefficients to lighten the notation.}
\begin{equation*}
\mathcal{D} = \text{diag}(\partial_t, \partial_t, \partial_t - i \frac{1}{(\varepsilon \mu)^{\frac{1}{2}}} |D'|, \partial_t + i \frac{1}{(\varepsilon \mu)^{\frac{1}{2}}} |D'|, \partial_t - i \frac{1}{(\varepsilon \mu)^{\frac{1}{2}}} |D'|, \partial_t + i  \frac{1}{(\varepsilon \mu)^{\frac{1}{2}}} |D'|)
\end{equation*}
and $w = \tilde{S}_{\lambda i} \mathcal{N}_\lambda^{(i)} S_{\lambda i} u$. $\tilde{S}_\lambda = \sum_{|j| \leq 2} S_{2^j \lambda}$ denotes a mildly enlarged version of $S_{\lambda i}$.

The first and second component are estimated by Sobolev embedding and the definition of $\mathcal{N}_\lambda^{(i)}$:
\begin{equation*}
\begin{split}
\| \tilde{S}_{\lambda i} w_1 \|_{L_t^p L_{x'}^q} &\lesssim \lambda^{\rho + \frac{1}{2}} \| \tilde{S}_{\lambda i} w_1 \|_{L_t^p L_{x'}^q} \lesssim \lambda^{\rho + \frac{1}{2}} \| \frac{1}{|D'|} \tilde{S}_{\lambda i} (\partial_1 u_1 + \partial_2 u_2 + \partial_3 u_3) \|_{L^2_x} \\
&\lesssim \lambda^{\rho - \frac{1}{2}} \| \rho_e \|_{L^2_x}.
\end{split}
\end{equation*}
The estimate of the second component in terms of $\rho_m$ follows \emph{mutatis mutandis}.
The third to sixth component are estimated by Proposition \ref{prop:HalfWaveEstimates}. This finishes the proof of \eqref{eq:DiagonalStrichartzEstimate}. To conclude the proof of Theorem \ref{thm:IsotropicStrichartz}, we show the following lemma:
\begin{lemma}
\label{lem:AuxDiagonalizationIsotropic}
With notations like above, we find the following estimates to hold:
\begin{align}
\label{eq:DiagonalizationA}
\lambda^{-\rho} \| S_{\lambda i } v \|_{L^p L^q} &\lesssim \lambda^{-\rho} \| \tilde{S}_{\lambda i} \mathcal{N}_\lambda^{(i)} S_{\lambda i} v \|_{L^p L^q} + \| S_{\lambda i } v \|_{L^2_x}, \\
\| S_{\lambda i} v \|_{L^2} &\lesssim \| \mathcal{M}_\lambda^{(i)} S_{\lambda i} v \|_{L^2_x}. \label{eq:DiagonalizationB}
\end{align}
\end{lemma}
\begin{proof}
We begin with the proof of \eqref{eq:DiagonalizationA}. Write by symbol composition
\begin{equation*}
\mathcal{M}_\lambda^{(i)} \tilde{S}_{\lambda i} \mathcal{N}_\lambda^{(i)} S_{\lambda i} v = (1+ R^{(i)}_\lambda) S_{\lambda i} v
\end{equation*}
with $\| R_{\lambda}^{(i)} S_{\lambda i} v \|_{L^2_x} \lesssim \lambda^{-\frac{1}{2}} \| S_{\lambda i} v \|_{L^2_x}$. By Sobolev embedding and Minkowski's inequality, we find
\begin{equation*}
\begin{split}
\lambda^{-\rho} \| S_{\lambda i } v \|_{L^p L^q} &\leq \lambda^{-\rho} \| \mathcal{M}_\lambda^{(i)} \tilde{S}_{\lambda i} \mathcal{N}_\lambda^{(i)} v \|_{L^p L^q} + \lambda^{-\rho} \| R_\lambda^{(i)} S_{\lambda i} v \|_{L^p L^q} \\
&\lesssim \lambda^{-\rho} \| \mathcal{N}_\lambda^{(i)} S_{\lambda i} v \|_{L^p L^q} + \| S_{\lambda i} v \|_{L^2}.
\end{split}
\end{equation*}
In the ultimate estimate we use boundedness of $\mathcal{M}_\lambda^{(i)} \tilde{S}_{\lambda i}$ due to Lemma \ref{lem:LpLqBoundsPseudos}. 

For the proof of \eqref{eq:DiagonalizationB} we write
\begin{equation*}
\mathcal{N}_\lambda^{(i)} \tilde{S}_{\lambda i} \mathcal{M}_\lambda^{(i)} S_{\lambda i} v = S_{\lambda i} v + R_\lambda^{(i)} S_{\lambda i} v
\end{equation*}
with $\| R_\lambda^{(i)} S_{\lambda i} v \|_{L^2} \lesssim \lambda^{-1} \| S_{\lambda i} v \|_{L^2}$. Therefore,
\begin{equation*}
\begin{split}
\| S_{\lambda i} v \|_{L^2} &\leq \| \mathcal{N}_\lambda^{(i)} \tilde{S}_{\lambda i} \mathcal{M}_\lambda^{(i)} S_{\lambda i} v \|_{L^2} + \| R^{(i)}_{\lambda} S_{\lambda i} v \|_{L^2} \\
&\leq \| \mathcal{N}_\lambda^{(i)} \tilde{S}_{\lambda i} \mathcal{M}_\lambda^{(i)} S_{\lambda i} v \|_{L^2} + C \lambda^{-1} \| S_{\lambda i} v \|_{L^2}.
\end{split}
\end{equation*}
For $\lambda$ large enough, we can absorb the error term into the left hand-side to find
\begin{equation*}
\| S_{\lambda i} v \|_{L^2} \lesssim \| \mathcal{N}_\lambda^{(i)} \tilde{S}_{\lambda i} \mathcal{M}_\lambda^{(i)} S_{\lambda i} v \|_{L^2} \lesssim \| \mathcal{M}_\lambda^{(i)} S_{\lambda i} v \|_{L^2}.
\end{equation*}
For the ultimate estimate, we invoke Lemma \ref{lem:LpLqBoundsPseudos} to bound $\mathcal{N}_\lambda^{(i)} \tilde{S}_{\lambda i}$ on $L^2$.
\end{proof}

We are ready to conclude the proof of Theorem \ref{thm:IsotropicStrichartz}. By \eqref{eq:DiagonalizationA} and \eqref{eq:DiagonalStrichartzEstimate}, we find
\begin{equation*}
\begin{split}
\lambda^{-\rho} \| S_{\lambda i} u \|_{L^p L^q} &\lesssim \lambda^{-\rho} \| \tilde{S}_{\lambda i} \mathcal{N}_{\lambda}^{(i)} S_{\lambda i} u \|_{L^p L^q} \\
&\lesssim \| \mathcal{N}_{\lambda i} S_{\lambda i} u \|_{L^2} + \| \tilde{S}_{\lambda i} \mathcal{D}_{\lambda} \tilde{S}_{\lambda i} \mathcal{N}_{\lambda i} S_{\lambda i} u \|_{L^2} + \lambda^{-\frac{1}{2}} \| \rho_{em} \|_{L^2_x}.
\end{split}
\end{equation*}
We can bound $\mathcal{N}_{\lambda i} S_{\lambda i}$ in the first term by appealing to Lemma \ref{lem:LpLqBoundsPseudos}.
We further apply \eqref{eq:DiagonalizationB} to the second term to find
\begin{equation*}
\begin{split}
&\quad \| \mathcal{N}_{\lambda i} S_{\lambda i} u \|_{L^2} + \| \tilde{S}_{\lambda i} \mathcal{D}_{\lambda} \tilde{S}_{\lambda i} \mathcal{N}_{\lambda i} S_{\lambda i} u \|_{L^2} + \lambda^{-\frac{1}{2}} \| \rho_{em} \|_{L^2_x} \\
&\lesssim \| S_{\lambda i} u \|_{L^2} + \| \mathcal{M}_\lambda^{(i)} \tilde{S}_{\lambda i} \mathcal{D}_{\lambda} \tilde{S}_{\lambda i} \mathcal{N}_{\lambda i} S_{\lambda i} u \|_{L^2} + \lambda^{-\frac{1}{2}} \| \rho_{em} \|_{L^2_x} \\
&\lesssim \| S_{\lambda i} u \|_{L^2} + \| P_\lambda S_{\lambda i} u \|_{L^2} + \| E_\lambda^{(i)} S_{\lambda i} u \|_{L^2}+ 
\lambda^{-\frac{1}{2}} \| \rho_{em} \|_{L^2_x}.
\end{split}
\end{equation*}
In the ultimate estimate we invoked Proposition \ref{prop:Diagonalization}, which further allows us to bound $E_\lambda^{(i)}$ in $L^2$. The proof of Theorem \ref{thm:IsotropicStrichartz} is complete. $\hfill \Box$

\subsection{Proof of Theorem \ref{thm:StrichartzEstimatesL1LinfCoefficientsIsotropic}}

We carry out the following steps to reduce Theorem \ref{thm:StrichartzEstimatesL1LinfCoefficientsIsotropic} to the dyadic estimates
\begin{equation}
\label{eq:DyadicEstimateL1Linfty}
\lambda^{-\rho} \| S_\lambda u \|_{L^p L^q} \lesssim \| S_\lambda u \|_{L^\infty L^2} + \| P_\lambda S_\lambda u \|_{L^1 L^2} + \lambda^{-1+\frac{1}{p}} \| S_\lambda \rho_{e m} \|_{L^\infty_t L_{x'}^2}
\end{equation}
where $\lambda \gg 1$, the Fourier support of $\varepsilon$ and $\mu$ is contained in $\{ |\xi| \leq \lambda^{\frac{1}{2}} \}$ and $u$ is essentially supported in the unit cube with space-time Fourier transform supported in $\{ |\xi_0| \lesssim \| \xi' \| \}$. For this purpose, we carry out the following steps:
\begin{itemize}
\item reduction to the case $\nu = 1$,
\item confining the support of $u$ to the unit cube and the frequency support to large frequencies,
\item estimate away from the characteristic surface,
\item reduction to dyadic estimates, 
\item truncating the coefficients at frequency $\lambda^{\frac{1}{2}}$.
\end{itemize}
This can be accomplished like in \cite[Subsection~3.4]{SchippaSchnaubelt2021} and in the previous paragraph. We omit the details to avoid repetition.

The dyadic estimate is again proved via diagonalization. Like above, we carry out an additional microlocalization $\{ |\xi_i^*| \gtrsim 1 \}$ in the region $\{ |\xi_0| \lesssim \| \xi' \| \}$. The estimate
\begin{equation}
\label{eq:DiagonalStrichartzIsotropicL1Linf}
\begin{split}
\lambda^{-\rho} \| \tilde{S}_{\lambda i} \mathcal{N}_\lambda^{(i)} S_{\lambda i} u \|_{L_t^p L_{x'}^q} &\lesssim \| \tilde{S}_{\lambda i} \mathcal{N}_\lambda^{(i)} S_{\lambda i} u \|_{L_t^{\infty} L_{x'}^2} + \| \mathcal{D}_\lambda \tilde{S}_{\lambda i} \mathcal{N}_\lambda^{(i)} S_{\lambda i} u \|_{L^2_x} \\
&\quad + \lambda^{-1 + \frac{1}{p}} (\| S_\lambda' \rho_{em}(0) \|_{L^2_{x'}} + \| \partial_t S_\lambda' \rho_{em} \|_{L_t^1 L_{x'}^2})
\end{split}
\end{equation}
is proved component-wise. For the components $[\tilde{S}_{\lambda i} \mathcal{N}_\lambda^{(i)} S_{\lambda i} u ]_j$, $j=3,\ldots,6$, we find by invoking Proposition \ref{prop:HalfWaveEstimates}:
\begin{equation*}
\lambda^{-\rho} \| [\tilde{S}_{\lambda i} \mathcal{N}_\lambda^{(i)} S_{\lambda i} u ]_j \|_{L_t^p L_{x'}^q} \lesssim \| [\tilde{S}_{\lambda i} \mathcal{N}_\lambda^{(i)} S_{\lambda i} u]_j \|_{L_t^\infty L_{x'}^2} + \| [\mathcal{D}_\lambda \tilde{S}_{\lambda i} \mathcal{N}_\lambda^{(i)} S_{\lambda i} u ]_j \|_{L^2_x}.
\end{equation*}
The first and second component are estimated by Sobolev embedding and H\"older's inequality:
\begin{equation*}
\begin{split}
\lambda^{-\rho} \| [\tilde{S}_{\lambda i} \mathcal{N}_\lambda^{(i)} S_{\lambda i} u ]_1 \|_{L_t^p L_{x'}^q} &= \| \frac{1}{|D'|} S_{\lambda i} \rho_e \|_{L_t^p L_{x'}^q} \\
&\lesssim \lambda^{-1 + \frac{1}{p}} \| S_{\lambda i} \rho_e \|_{L_t^\infty L_{x'}^2} \lesssim \lambda^{-1 + \frac{1}{p}} \| S'_\lambda \rho_{em} \|_{L_t^\infty L_{x'}^2}.
\end{split}
\end{equation*}
By the fundamental theorem and Minkowski's inequality, we find
\begin{equation*}
\| S'_\lambda \rho_{em} \|_{L_t^\infty L_{x'}^2} \lesssim \| S_\lambda' \rho_e(0) \|_{L^2_{x'}} + \| S_\lambda' \partial_t \rho_e \|_{L_t^1 L_{x'}^2}.
\end{equation*}
For the second component, we obtain similarly
\begin{equation*}
\lambda^{-\rho} \| [\tilde{S}_{\lambda i} \mathcal{N}_\lambda^{(i)} S_{\lambda i} u ]_2 \|_{L_t^p L_{x'}^q} \lesssim \lambda^{-1 + \frac{1}{p}} \big( \| S_\lambda' \rho_m(0) \|_{L^2_{x'}} + \| \partial_t S_\lambda' \rho_m \|_{L_t^1 L_{x'}^2} \big).
\end{equation*}
This yields \eqref{eq:DiagonalStrichartzIsotropicL1Linf}. The sequence of estimates
\begin{equation*}
\begin{split}
\lambda^{-\rho} \| S_{\lambda i} u \|_{L^p L^q} &\lesssim \lambda^{-\rho} \| \tilde{S}_{\lambda i} \mathcal{N}_\lambda^{(i)} S_{\lambda i} u \|_{L^p L^q} \\
&\lesssim \| \tilde{S}_{\lambda i} \mathcal{N}_\lambda^{(i)} S_{\lambda i} u \|_{L_t^\infty L_{x'}^2} + \| \mathcal{D}_\lambda \tilde{S}_{\lambda i} \mathcal{N}_\lambda^{(i)} S_{\lambda i} u \|_{L^2_{t,x}} \\
&\quad + \lambda^{-1 + \frac{1}{p}} ( \| S'_\lambda \rho_{em}(0) \|_{L^2_{x'}} + \| \partial_t S'_\lambda \rho_{em} \|_{L_t^1 L_{x'}^2} ) \\
&\lesssim \| S_{\lambda i} u \|_{L_t^\infty L_{x'}^2} + \| \mathcal{M}_\lambda^{(i)} \tilde{S}_{\lambda i} \mathcal{D}_\lambda \tilde{S}_{\lambda i} \mathcal{N}_\lambda^{(i)} S_{\lambda i} u \|_{L^2_{t,x}} \\
&\quad + \lambda^{-1 + \frac{1}{p}} ( \| S'_\lambda \rho_{em}(0) \|_{L^2_{x'}} + \| \partial_t S'_\lambda \rho_{em} \|_{L_t^1 L_{x'}^2} ) \\
&\lesssim \| S_{\lambda i} u \|_{L_t^\infty L_{x'}^2} + \| P_\lambda S_{\lambda i} u \|_{L^2_x} \\
&\quad + \lambda^{-1+\frac{1}{p}} ( \| S'_\lambda \rho_{em}(0) \|_{L^2_{x'}} + \| \partial_t S'_\lambda \rho_{em} \|_{L_t^1 L_{x'}^2}
\end{split}
\end{equation*}
follows like at the end of Section \ref{subsection:ReductionsC2Isotropic}. The proof of Theorem \ref{thm:StrichartzEstimatesL1LinfCoefficientsIsotropic} is complete. $\hfill \Box$

\section{Strichartz estimates in the partially anisotropic case}
\label{section:PartiallyAnisotropicCase}
We turn to the partially anisotropic case. The conjugation matrices take a more difficult form because the additional microlocalization to regions $\{|\xi_i^*| \gtrsim 1 \}$ does not allow to choose eigenvectors with improved regularity. In the constant-coefficient case we can argue that we have $L^p$-bounded Fourier multipliers nonetheless by the H\"ormander-Mikhlin theorem. For variable coefficients, this does not appear to be possible in the general case, but only under additional structural assumptions.
\subsection{Diagonalizing the principal symbol}
 Like in the previous section, we begin with diagonalizing the principal symbol. 
Recall that this is given by 
\begin{equation*}
\tilde{p}(x,\xi)/i =
\begin{pmatrix}
 \xi_0 1_{3 \times 3} & -\mathcal{C}(\xi') \\
\mathcal{C}(\xi') \varepsilon^{-1}(x) & \xi_0 1_{3 \times 3}
\end{pmatrix}
, \quad \mathcal{C}(\xi') = 
\begin{pmatrix}
0 & - \xi_3 & \xi_2 \\
\xi_3 & 0 & - \xi_1 \\
- \xi_2 & \xi_1 & 0
\end{pmatrix}
.
\end{equation*}
The diagonalization in the constant-coefficient case was previously computed in \cite{Schippa2022}.

We suppose that $\varepsilon^{-1} = \text{diag}(a,b,b)$. Let 
\begin{align*}
\| \xi' \|^2 &= \xi_1^2 + \xi_2^2 + \xi_3^2, \quad \| \xi' \|_\varepsilon^2 = b(x) \xi_1^2 + a(x) \xi_2^2 + a(x) \xi_3^2, \\
\xi_i^* &= \xi_i / \| \xi \|, \qquad \tilde{\xi}_i = \xi_i / \| \xi \|_\varepsilon, \qquad i = 1,2,3.
\end{align*}

The eigenvalues of $\tilde{p}(x,\xi)$ are
\begin{align*}
\lambda_{1,2} = i \xi_0, \quad \lambda_{3,4} = i \xi_0 \mp i \sqrt{b(x)} \| \xi' \|, \quad
\lambda_{5,6} = i \xi_0 \mp i \| \xi' \|_\varepsilon.
\end{align*}
Let
\begin{equation}
\label{eq:DiagonalMatrixPartiallyAnisotropic}
d(x,\xi) = i \text{diag} (\xi_0,\xi_0,\xi_0 - \sqrt{b(x)} \| \xi' \|, \xi_0 + \sqrt{b(x)} \| \xi' \|, \xi_0 - \| \xi' \|_{\varepsilon}, \xi_0 + \| \xi' \|_{\varepsilon}).
\end{equation}
We find the following corresponding eigenvectors, which are normalized to zero-homogeneous entries. Eigenvectors to $i \xi_0$ are
\begin{align*}
v_1^t &= \big(0,0,0, \xi_1^* , \xi_2^*, \xi_3^* \big), \\
v_2^t &= \big(\frac{\tilde{\xi}_1}{a }, \frac{\tilde{\xi}_2}{b }, \frac{\tilde{\xi}_3}{b }, 0, 0, 0 \big).
\end{align*}
Eigenvectors to $i\xi_0 \pm i \sqrt{b(x)} \| \xi \|$ are given by
\begin{align*}
v_3^t &= \big(0,- \frac{\xi_3^* }{\sqrt{b}}, \frac{\xi_2^* }{\sqrt{b}}, - ({\xi_2^*}^2 + {\xi_3^*}^2), \xi_1^* \xi_2^*, \xi_1^* \xi_3^* \big), \\
v_4^t &= \big(0, \frac{\xi_3^* }{\sqrt{b} }, - \frac{\xi_2^* }{\sqrt{b}}, - ({\xi_2^*}^2 + {\xi_3^*}^2), \xi_1^* \xi_2^*, \xi_1^* \xi_3^* \big).
\end{align*}
Eigenvectors to $i \xi_0 \pm i \| \xi \|_{\varepsilon}$ are given by
\begin{align*}
v_5^t &= \big( \tilde{\xi}_2^2 + \tilde{\xi}_3^2, - \tilde{\xi}_1 \tilde{\xi}_2, - \tilde{\xi}_1 \tilde{\xi}_3, 0 , - \tilde{\xi}_3, \tilde{\xi}_2 \big),\\
v_6^t &= \big(- (\tilde{\xi}_2^2+ \tilde{\xi}_3^2), \tilde{\xi}_1 \tilde{\xi}_2, \tilde{\xi}_1 \tilde{\xi}_3, 0, -\tilde{\xi}_3, \tilde{\xi}_2 \big).
\end{align*}
Set
\begin{equation*}
m(x,\xi) = (v_1, \ldots, v_6).
\end{equation*}
We find
\begin{equation*}
\begin{split}
&m^{-1}(x,\xi) = \\
&\begin{pmatrix}
0 & 0 & 0 &  \xi_1^* & \xi_2^* & \xi_3^* \\
ab \tilde{\xi}_1 & ab \tilde{\xi}_2 & ab \tilde{\xi}_3 & 0 & 0 & 0 \\
0 & - \frac{\sqrt{b} \| \xi \|}{2 \| \xi \|_\varepsilon} \frac{\tilde{\xi}_3}{\tilde{\xi}_2^2 + \tilde{\xi}_3^2} &  
\frac{\sqrt{b} \| \xi \|}{2 \| \xi \|_\varepsilon} \frac{\tilde{\xi}_2}{\tilde{\xi}_2^2 + \tilde{\xi}_3^2} & - \frac{1}{2} & \frac{\xi_1^* \xi_2^*}{2({\xi_2^*}^2 + {\xi_3^*}^2)} & \frac{\xi_1^* \xi_3^*}{2({\xi_2^*}^2 + {\xi_3^*}^2)} \\
0 & \frac{\sqrt{b}\| \xi \| }{2 \| \xi \|_\varepsilon} \frac{\tilde{\xi}_3}{\tilde{\xi}_2^2 + \tilde{\xi}_3^2} & - \frac{\sqrt{b} \| \xi \| }{2 \| \xi \|_\varepsilon} \frac{\tilde{\xi}_2}{\tilde{\xi}_2^2 + \tilde{\xi}_3^2} & - \frac{1}{2} & \frac{\xi_1^* \xi_2^*}{2({\xi_2^*}^2 + {\xi_3^*}^2)} & \frac{\xi_1^* \xi_3^*}{2({\xi_2^*}^2 + {\xi_3^*}^2)} \\
\frac{a}{2} & - \frac{b \tilde{\xi}_1 \tilde{\xi}_2}{2(\tilde{\xi}_2^2 + \tilde{\xi}_3^2)} & - \frac{b \tilde{\xi}_1 \tilde{\xi}_3}{2(\tilde{\xi}_2^2 + \tilde{\xi}_3^2)} & 0 & - \frac{\xi_3^* \| \xi \|_\varepsilon}{2 \| \xi \| ({\xi_2^*}^2 + {\xi_3^*}^2)} & \frac{\| \xi \|_\varepsilon \xi_2^* }{2 \| \xi \| ({\xi_2^*}^2 + {\xi_3^*}^2)} \\
-\frac{a}{2} & \frac{b \tilde{\xi}_1 \tilde{\xi}_2}{2(\tilde{\xi}_2^2 + \tilde{\xi}_3^2)} &  \frac{b \tilde{\xi}_1 \tilde{\xi}_3}{2(\tilde{\xi}_2	^2 + \tilde{\xi}_3^2)} & 0 & - \frac{\| \xi \|_\varepsilon}{2 \| \xi \|} \frac{\xi_3^*}{({\xi_2^*}^2 + {\xi_3^*}^2)} & \frac{\| \xi \|_\varepsilon \xi_2^*}{2 \| \xi \| ({\xi_2^*}^2 + {\xi_3^*}^2)}
\end{pmatrix}
.
\end{split}
\end{equation*}
In the constant-coefficient case, Lucente--Ziliotti \cite{LucenteZiliotti2000} used a similar argument, but did not give the eigenvectors. It turns out that these have to be normalized carefully to find uniformly $L^p$-bounded conjugation operators. More precisely, note that the matrix becomes singular for $|\xi_2| + |\xi_3| \to 0$. The remedy is to renormalize $v_3,\ldots,v_6$ with 
\begin{equation}
\label{eq:Factor}
\alpha(x,\xi) = \frac{(\xi_2^2 + \xi_3^2)^{\frac{1}{2}}}{(\| \xi' \| \| \xi' \|_\varepsilon)^{\frac{1}{2}}}.
\end{equation}
In fact, we find by elementary matrix operations, that is adding and subtracting the third and fourth, and fifth and sixth eigenvector, that
\begin{equation*}
| \det m(x,\xi) | \sim_\varepsilon
\begin{vmatrix}
0 & 0 & 0 & 1 & 0 & 0 \\
0 & 0 & 0 & 0 & \xi_2^* & \xi_3^* \\
0 & 0 & 0 & 0 & -\tilde{\xi}_3 & \tilde{\xi}_2 \\
1 & 0 & 0 & 0 & 0 & 0 \\
0 & \xi_3^* & - \xi_2^* & 0 & 0 & 0 \\
0 & \tilde{\xi}_2 & \tilde{\xi}_3 & 0 & 0 & 0
\end{vmatrix}
\sim (\xi_2^* \tilde{\xi}_2 + \xi_3^* \tilde{\xi}_3 )^2 = \alpha^4(x,\xi).
\end{equation*}
This suggests renormalizing the eigenvectors from above with \eqref{eq:Factor}, as for the associated eigenvectors of $v_3/\alpha(x,\xi),\ldots,v_6/\alpha(x,\xi)$ we can verify $L^p L^q$-boundedness.
We give the details. Let $\delta = \| \xi \|/\| \xi \|_\varepsilon$. Note that
\begin{equation*}
\alpha(x,\xi) = \frac{(\xi_2^2 + \xi_3^2)^{\frac{1}{2}}}{(\| \xi \| \| \xi \|_\varepsilon)^{\frac{1}{2}}} = \frac{(\tilde{\xi}_2^2 + \tilde{\xi}_3^2)^{\frac{1}{2}}}{\delta^{\frac{1}{2}}} = (\delta(\xi_2'^2 + \xi_3'^2))^{\frac{1}{2}}. 
\end{equation*}
We find
\begin{equation}
\label{eq:ConjugationMatrixPartiallyAnisotropic}
\begin{split}
&\tilde{m}(x,\xi) = \\
&\begin{pmatrix}
0 & \frac{\tilde{\xi}_1}{a} & 0 & 0 & (\delta (\tilde{\xi}_2^2 + \tilde{\xi}_3^2))^{\frac{1}{2}} & - (\delta(\tilde{\xi}_2^2 + \tilde{\xi}_3^2))^{\frac{1}{2}} \\
0 & \frac{\tilde{\xi}_2}{b} & - \frac{\xi_3'}{\sqrt{b} (\delta( \xi_2'^2 + \xi_3'^2))^{\frac{1}{2}}} & \frac{\xi_3'}{\sqrt{b} (\delta(\xi_2'^2 + \xi_3'^2))^{\frac{1}{2}}} & - \frac{\delta^{\frac{1}{2}} \tilde{\xi}_1 \tilde{\xi}_2}{(\tilde{\xi}_2^2 + \tilde{\xi}_3^2)^{1/2}} & \frac{\delta^{\frac{1}{2}} \tilde{\xi}_1 \tilde{\xi}_2}{(\tilde{\xi}_2^2 + \tilde{\xi}_3^2)^{\frac{1}{2}}} \\
0 & \frac{\tilde{\xi}_3}{b} & \frac{\xi_2'}{\sqrt{b} (\delta (\xi_2'^2 + \xi_3'^2))^{\frac{1}{2}}} & - \frac{\xi_2'}{\sqrt{b} (\delta(\xi_2'^2 + \xi_3'^2))^{\frac{1}{2}}} & - \frac{\delta^{\frac{1}{2}} \tilde{\xi}_1 \tilde{\xi}_3}{(\tilde{\xi}_2^2 + \tilde{\xi}_3^2)^{\frac{1}{2}}} & \frac{\delta^{\frac{1}{2}} \tilde{\xi}_1 \tilde{\xi}_3}{(\tilde{\xi}_2^2 + \tilde{\xi}_3^2)^{\frac{1}{2}}} \\
\xi_1' & 0 &- \frac{(\xi_2'^2 + \xi_3'^2)^{\frac{1}{2}}}{\delta^{\frac{1}{2}}} & - \frac{(\xi_2'^2 + \xi_3'^2)^{\frac{1}{2}}}{\delta^{\frac{1}{2}}} & 0 & 0 \\
\xi_2' & 0 & \frac{\xi_1' \xi_2' }{(\delta(\xi_2'^2 + \xi_3'^2))^{\frac{1}{2}}} & \frac{\xi_1' \xi_2' }{(\delta (\xi_2'^2 + \xi_3'^2))^{\frac{1}{2}}} & - \frac{\delta^{\frac{1}{2}} \tilde{\xi}_3}{(\tilde{\xi}_2^2 + \tilde{\xi}_3^2)^{\frac{1}{2}}} & - \frac{\delta^{\frac{1}{2}} \tilde{\xi}_3}{(\tilde{\xi}_2^2 + \tilde{\xi}_3^2)^{\frac{1}{2}}} \\
\xi_3' & 0 & \frac{\xi_1' \xi_3'}{(\delta( \xi_2'^2 + \xi_3'^2))^{\frac{1}{2}}} & \frac{\xi_1' \xi_3' }{ (\delta(\xi_2'^2 + \xi_3'^2))^{\frac{1}{2}}} & \frac{\tilde{\xi}_2 \delta^{\frac{1}{2}}}{(\tilde{\xi}_2^2 + \tilde{\xi}_3^2)^{\frac{1}{2}}} & \frac{\tilde{\xi}_2 \delta^{\frac{1}{2}}}{(\tilde{\xi}_2^2 + \tilde{\xi}_3^2)^{\frac{1}{2}}}
\end{pmatrix}.
\end{split}
\end{equation}

By Cramer's rule, we find $\tilde{m}(x,\xi)^{-1}$ from $m^{-1}(x,\xi)$ by modifying the rows 3-6:
\begin{equation}
\label{eq:InverseConjugationMatrixPartiallyAnisotropic}
\begin{split}
&\tilde{m}^{-1}(x,\xi) = \\
&\begin{pmatrix}
0 & 0 & 0 & \xi_1' & \xi_2' & \xi_3' \\
ab \tilde{\xi}_1 & ab \tilde{\xi}_2 & ab \tilde{\xi}_3 & 0 & 0 & 0 \\
0 & - \frac{\sqrt{b} \delta^{\frac{1}{2}} \tilde{\xi}_3}{2 (\tilde{\xi}_2^2+ \tilde{\xi}_3^2)^{\frac{1}{2}}} &  \frac{\sqrt{b} \delta^{\frac{1}{2}} \tilde{\xi}_2}{2 (\tilde{\xi}_2^2 + \tilde{\xi}_3^2)^{\frac{1}{2}}} & - \frac{ (\tilde{\xi}_2^2 + \tilde{\xi}_3^2)^{\frac{1}{2}}}{2 \delta^{\frac{1}{2}}} & \frac{\xi_1' \xi_2' \delta^{\frac{1}{2}}}{2(\xi_2'^2 + \xi_3'^2)^{\frac{1}{2}}} & \frac{\xi_1' \xi_3' \delta^{\frac{1}{2}}}{2(\xi_2'^2 + \xi_3'^2)^{\frac{1}{2}}} \\
0 &  \frac{\sqrt{b} \delta^{\frac{1}{2}} \tilde{\xi}_3}{2 (\tilde{\xi}_2^2+\tilde{\xi}_3^2)^{\frac{1}{2}}} & - \frac{\sqrt{b} \delta^{\frac{1}{2}} \tilde{\xi}_2}{2 (\tilde{\xi}_2^2 + \tilde{\xi}_3^2)^{1/2}} & - \frac{ (\tilde{\xi}_2^2 + \tilde{\xi}_3^2)^{\frac{1}{2}}}{2 \delta^{\frac{1}{2}}} & \frac{\delta^{\frac{1}{2}} \xi_1' \xi_2'}{2(\xi_2'^2 + \xi_3'^2)^{\frac{1}{2}}} & \frac{\delta^{\frac{1}{2}} \xi_1' \xi_3'}{2(\xi_2'^2 + \xi_3'^2)^{\frac{1}{2}}} \\
\frac{a (\tilde{\xi}_2 + \tilde{\xi}_3^2)^\frac{1}{2}}{2 \delta^{\frac{1}{2}}} & - \frac{ b \tilde{\xi}_1 \tilde{\xi}_2}{2 (\delta(\tilde{\xi}_2^2 + \tilde{\xi}_3^2))^{\frac{1}{2}}} & - \frac{b \tilde{\xi}_1 \tilde{\xi}_3}{ 2 (\delta(\tilde{\xi}_2^2 + \tilde{\xi}_3^2))^{\frac{1}{2}}} & 0 & - \frac{\xi_3'}{2(\delta(\xi_2'^2 + \xi_3'^2))^{\frac{1}{2}}} & \frac{\xi_2'}{2  (\delta(\xi_2'^2 + \xi_3'^2))^{\frac{1}{2}}} \\
- \frac{a (\tilde{\xi}_2 + \tilde{\xi}_3^2)^\frac{1}{2}}{2 \delta^{\frac{1}{2}}} & \frac{b \tilde{\xi}_1 \tilde{\xi}_2}{2 (\delta(\tilde{\xi}_2^2 + \tilde{\xi}_3^2))^{\frac{1}{2}}} & \frac{b \tilde{\xi}_1 \tilde{\xi}_3}{2 (\delta(\tilde{\xi}_2^2 + \tilde{\xi}_3^2))^{\frac{1}{2}}} & 0 & - \frac{\xi_3'}{2(\delta (\xi_2'^2 + \xi_3'^2))^{\frac{1}{2}}} & \frac{\xi_2'}{2(\delta( \xi_2'^2 + \xi_3'^2))^{\frac{1}{2}}}
\end{pmatrix}
.
\end{split}
\end{equation}
In conclusion, we find
\begin{equation*}
\tilde{p}(x,\xi) = \tilde{m}(x,\xi) d(x,\xi) \tilde{m}^{-1}(x,\xi).
\end{equation*}

In the following we associate pseudo-differential operators with the symbols. To obtain admissible symbols, we localize frequencies $\{\| \xi' \| \sim \lambda \}$ away from the $\xi_1$-axis to the region $\{|(\xi_2,\xi_3)| \gtrsim \lambda^{\alpha }$ with $\frac{1}{2} \leq \alpha < 1$. The contribution of $|(\xi_2,\xi_3)| \lesssim \lambda^{\alpha}$ can be estimated directly via Bernstein's inequality. Since we shall truncate the coefficients to frequencies of size $\lambda^{\beta}$, $\beta \leq \alpha$, this leads to symbols $S^m_{\alpha,\beta}$. For $m=0$, these are bounded in $L^2$ by the Calderon--Vaillancourt theorem. To compute bounds in $L^p L^q$, we use symbol composition to write it as composition of Riesz transforms and pseudo-differential operators, which allow a straight-forward estimate in $L^2$. The error terms are sufficiently smoothing to be estimated via Sobolev embedding in $L^2$. The choice of $\alpha$ and $\beta$ depends on the regularity of the coefficients:
\begin{itemize}
\item In the case of structured coefficients like in Theorem \ref{thm:StructuredStrichartzPartiallyAnisotropic} we choose $\alpha = \beta = \frac{1}{2}$. This allows for the proof of Strichartz estimates with same derivative loss like in the free case.
\item In the case of coefficients, which satisfy $\partial \varepsilon \in L_T^2 L_{x'}^\infty$ we choose $\alpha = \frac{3}{4}$ and $\beta = \frac{1}{2}$. This proves Strichartz estimates with $1/4+\varepsilon$ additional derivative loss close to the forbidden endpoint $(p,q) = (2,\infty)$ compared to Euclidean Strichartz estimates.
\item In the case of Lipschitz coefficients, we choose $\alpha = \frac{2}{3}+\varepsilon$ and $\beta = \frac{2}{3}$. This proves Strichartz estimates with $1/6+\varepsilon$ additional derivative loss close to the forbidden endpoint $(p,q) = (2,\infty)$. This is up to $\varepsilon$ the same additional derivative loss like for scalar wave equations with Lipschitz coefficients close to the forbidden endpoint $(p,q) = (2,\infty)$.
\end{itemize}

A direct estimate in $L^p L^q$ is unclear because $L^p_x$-boundedness of $Op(S^0_{\rho,0})$ for $0 \leq \rho < 1$ fails in general (cf. \cite[Chapter~XI]{Taylor1974}).

\subsection{Proof of Theorem \ref{thm:StructuredStrichartzPartiallyAnisotropic}}
\label{subsection:StructuredAnisotropic}

In this section we prove Strichartz estimates under structural assumptions on the coefficients by conjugating the Maxwell operator to half-wave equations. The first reductions are like in Section \ref{subsection:ReductionsC2Isotropic} and the details are omitted.
\subsubsection{Localization arguments}
 By scaling we can suppose that $\| \partial^2 \varepsilon \|_{L^\infty} \leq 1$, $\nu =1$. We carry out the following reductions like in Section \ref{subsection:ReductionsC2Isotropic}:
\begin{itemize}
\item Reduction to high frequencies and localization to a cube of size $1$,
\item Reduction to dyadic estimates,
\item Truncating the coefficients of $P$ at frequency $\lambda^{\frac{1}{2}}$,
\item Microlocal estimate away from the characteristic surface.
\end{itemize}

After these steps, it suffices to prove the following dyadic estimate:
\begin{equation}
\label{eq:DyadicStrichartzEstimatesPartiallyAnisotropicStructured}
\lambda^{-\rho} \| S_\lambda u \|_{L^p L^q} \lesssim \| S_\lambda u \|_{L^2_x} + \| P_\lambda S_\lambda u \|_{L^2_x} + \lambda^{-\frac{1}{2}} \| S_\lambda' \rho_{em} \|_{L^2_x}
\end{equation}
for $\lambda \gg 1$, $u$ having Fourier support in $\{|\xi_0| \lesssim \| \xi' \| \sim \lambda \}$ and being essentially supported in a space-time unit cube. $P_\lambda$ denotes the time-dependent Maxwell operator with coefficients truncated at frequencies $\lesssim \lambda^{\frac{1}{2}}$.
\subsubsection{Estimate without diagonalization}
We estimate directly the contribution of the spatial frequencies $\{ |(\xi_2,\xi_3)| \lesssim \lambda^{\frac{1}{2}} \}$ by Bernstein's inequality. Let $\chi_A(\xi')$ denote a smooth version of the indicator function of
\begin{equation*}
A = \{ \| \xi' \| \sim \lambda \} \cap \{ |(\xi_2,\xi_3)| \lesssim \lambda^{\frac{1}{2}} \}.
\end{equation*}
We estimate $|A| \lesssim \lambda^{2 }$. Hence,
\begin{equation*}
\lambda^{-\rho} \| S_\lambda S_A u \|_{L_t^2 L_{x'}^\infty} \lesssim \| S_\lambda S_A u \|_{L^2_x}.
\end{equation*}
By interpolation with the energy estimate, the contribution of frequencies in $A$ is estimated.
\subsubsection{Estimate via diagonalization of Maxwell operator}
Let $S_B$ denote the smooth frequency projection to
\begin{equation*}
\{ |\xi_0| \lesssim \| \xi' \| \sim \lambda \} \cap \{ |(\xi_2,\xi_3)| \gg \lambda^{\frac{1}{2}} \}
\end{equation*}
with symbol $s_B(\xi)$. Then $P_{\lambda} S_B$ admits diagonalization by quantizing $\tilde{m}(x,\xi) s_B(\xi)$, $d(x,\xi) s_B(\xi)$, and $\tilde{m}^{-1}(x,\xi) s_B(\xi)$ as given in \eqref{eq:ConjugationMatrixPartiallyAnisotropic}, \eqref{eq:DiagonalMatrixPartiallyAnisotropic}, and \eqref{eq:InverseConjugationMatrixPartiallyAnisotropic}. For this purpose note that $\tilde{m}(x,\xi) s_B(\xi)$ and $\tilde{m}^{-1}(x,\xi) s_B(\xi) \in S^0_{\frac{1}{2},\frac{1}{2}}$, and $d(x,\xi) \in S^1_{1,\frac{1}{2}}$.
We shall prove the dyadic estimate
\begin{equation*}
\lambda^{-\rho} \| S_B u \|_{L_t^p L_{x'}^q} \lesssim \| S_\lambda u \|_{L^2_x} + \| P_\lambda S_B u \|_{L^2_x} + \lambda^{-\frac{1}{2}} \| S'_\lambda \rho_{em} \|_{L^2_x}
\end{equation*}
with $u$ having the same properties like in \eqref{eq:DyadicStrichartzEstimatesPartiallyAnisotropicStructured}. This reduction requires an additional commutator estimate for the localization $S_B$. Note that $\| P_\lambda S_B u \|_{L^2_x} \lesssim \| P_\lambda S_B u \|_{L^2_x}$ because the projection on $\xi_2$ and $\xi_3$ commutes with $P_\lambda$.

 Symbol composition (Theorem \ref{thm:KohnNirenberg}) holds to first order because the coefficients are Lipschitz. We shall see that we have the following improved error estimate compared to standard symbol composition:
\begin{proposition}
\label{prop:StructuredDiagonalizationPartiallyAnisotropic}
With above notations, let
\begin{equation}
\label{eq:QuantizationPartiallyAnisotropic}
\mathcal{M}_\lambda = OP(\tilde{m}(x,\xi) \tilde{\chi}_B(\xi)), \; \mathcal{D}_\lambda = OP(d(x,\xi) \tilde{\chi}_B(\xi)), \; \mathcal{N}_\lambda = OP(\tilde{m}^{-1}(x,\xi) \chi_B(\xi)).
\end{equation}
Then, we find the following identity to hold:
\begin{equation}
\label{eq:StructuredDiagonalizationPartiallyAnisotropic}
P_\lambda S_B = \mathcal{M}_\lambda \tilde{S}_B \mathcal{D}_\lambda \tilde{S}_B \mathcal{N}_\lambda S_B + E_\lambda
\end{equation}
with $\| E_\lambda \|_{L^2 \to L^2} \lesssim 1$.
\end{proposition}
\begin{proof}
We inspect the asymptotic expansion to obtain the error estimate. First order symbol composition gives
\begin{equation*}
\mathcal{M}_\lambda \tilde{S}_B \mathcal{D}_\lambda \tilde{S}_B \mathcal{N}_\lambda S_B = P_\lambda S_B + E_\lambda
\end{equation*}
with $E_\lambda = OP(e_\lambda)$, $e_\lambda \in S^{\frac{1}{2}}_{\frac{1}{2},\frac{1}{2}}$.

To improve on the bounds for $E_\lambda$, we first consider the composition of $\mathcal{D}_\lambda \tilde{S}_B$ and $\mathcal{N}_\lambda S_B$. We obtain by Theorem \ref{thm:KohnNirenberg}
\begin{equation*}
\mathcal{D}_\lambda \tilde{S}_B \mathcal{N}_\lambda \tilde{S}_B = OP(d(x,\xi) \tilde{m}^{-1}(x,\xi) \chi_B(\xi)) + E_\lambda 
\end{equation*}
with asymptotic expansion of the symbol of $E_\lambda = OP(e_\lambda)$ given by 
\begin{equation*}
e_\lambda = \sum_{|\alpha| \geq 1} \frac{1}{\alpha !} OP((D_\xi^\alpha (d(x,\xi) \tilde{\chi}_B(\xi)) (\partial_x^\alpha \tilde{m}^{-1}(x,\xi) \chi_B(\xi)).
\end{equation*}
We shall see that the asymptotic expansion converges although $d(x,\xi) \tilde{\chi}_B(\xi) \in S^1_{\frac{1}{2},\frac{1}{2}}$, $\tilde{m}(x,\xi) \chi_B(\xi) \in S^0_{\frac{1}{2},\frac{1}{2}}$.

Let $\alpha = (\alpha_0,\alpha_1,\alpha_2,\alpha_3)$. By the structural assumptions, the terms with $|\alpha_2| + |\alpha_3| > 0$ are vanishing. But the derivatives in $\xi_0$ and $\xi_1$ applied to $\tilde{m}_{ij}(x,\xi) \tilde{\chi}_B(\xi)$ gain factors of $\lambda^{-1}$. Hence, we obtain
\begin{equation*}
\mathcal{D}_\lambda \tilde{S}_B \mathcal{N}_\lambda S_B = OP(d(x,\xi) \tilde{m}^{-1}(x,\xi)  \tilde{\chi}^2(\xi)) + E_\lambda
\end{equation*}
with $E_\lambda = OP(e_\lambda)$, $e_\lambda \in S^0_{\frac{1}{2},\frac{1}{2}}$.

By the similar argument, we obtain
\begin{equation*}
\mathcal{M}_\lambda \tilde{S}_B \mathcal{D}_\lambda \tilde{S}_B \mathcal{N}_\lambda S_B = OP(\tilde{m}(x,\xi) d(x,\xi) \tilde{m}^{-1}(x,\xi) \chi_B(\xi)) + E_\lambda = P_\lambda S_B + E_\lambda
\end{equation*}
with $E_\lambda \in OP S^0_{\frac{1}{2},\frac{1}{2}}$. $E_\lambda$ is bounded in $L^2$ by the Calderon--Vaillancourt theorem. The proof is complete.
\end{proof}

To conclude the proof of Theorem \ref{thm:StructuredStrichartzPartiallyAnisotropic} by using the diagonalization, we argue like at the end of Section \ref{section:StrichartzEstimatesIsotropic}. The symbol composition is more delicate in the present case. We can still show the following lemma:
\begin{lemma}
\label{lem:AuxDiagonalizationStructuredAnisotropic}
With above notations, we find the following estimates to hold:
\begin{align}
\label{eq:AuxDiagonalizationI}
\lambda^{-\rho} \| S_B u \|_{L^p L^q} &\lesssim \lambda^{-\rho} \| \mathcal{N}_\lambda S_B u \|_{L^p L^q} + \| S_B u \|_{L^2_x}, \\
\label{eq:AuxDiagonalizationII}
\lambda^{-\rho} \| \mathcal{N}_\lambda S_B u \|_{L^p L^q} &\lesssim \| \mathcal{D}_\lambda \tilde{S}_B \mathcal{N}_\lambda S_B u \|_{L^2_x} + \| \mathcal{N}_\lambda S_B u \|_{L^2_x} + \lambda^{-\frac{1}{2}} \| S'_\lambda \rho_{em} \|_{L^2_x}, \\
\label{eq:AuxDiagonalizationIII}
\| S_B v \|_{L^2_x} &\lesssim \| \mathcal{M}_\lambda S_B v \|_{L^2_x}.
\end{align}
\end{lemma}

Before we turn to the proof of Lemma \ref{lem:AuxDiagonalizationStructuredAnisotropic}, we shall see how to conclude the proof of Theorem \ref{thm:StructuredStrichartzPartiallyAnisotropic} at its disposal. 
\begin{proof}[Conclusion of the proof of Theorem \ref{thm:StructuredStrichartzPartiallyAnisotropic}]
By appealing to \eqref{eq:AuxDiagonalizationI} and \eqref{eq:AuxDiagonalizationII}, we find
\begin{equation*}
\begin{split}
\lambda^{-\rho} \| S_B u \|_{L^p L^q} &\lesssim \lambda^{-\rho} \| \mathcal{N}_\lambda S_B u \|_{L^p L^q} + \| S_B u \|_{L^2_x} \\
&\lesssim \| \mathcal{D}_\lambda \tilde{S}_B \mathcal{N}_\lambda S_B u \|_{L^2_x} + \lambda^{-\frac{1}{2}} \| S'_\lambda \rho_{em} \|_{L^2_x} + \| \mathcal{N}_\lambda S_B u \|_{L^2_x}.
\end{split}
\end{equation*}
We apply \eqref{eq:AuxDiagonalizationI} to the first term and since $\mathcal{N}_\lambda S_B \in OP S^0_{\frac{1}{2},\frac{1}{2}}$, this is bounded in $L^2_x$ by the Calderon--Vaillancourt theorem:
\begin{equation*}
\begin{split}
&\lesssim \| \mathcal{M}_\lambda \tilde{S}_B \mathcal{D}_\lambda \tilde{S}_B \mathcal{N}_\lambda S_B u \|_{L^2_x} + \lambda^{-\frac{1}{2}} \| S'_\lambda \rho_{em} \|_{L^2_{x}} + \| S_\lambda u \|_{L^2_x} \\
&\lesssim \| P_\lambda S_B u \|_{L^2_x} + \| S_\lambda u \|_{L^2_x} + \lambda^{-\frac{1}{2}} \| S'_\lambda \rho_{em} \|_{L^2_{x}}.
\end{split}
\end{equation*}
The ultimate estimate is a consequence of Proposition \ref{prop:StructuredDiagonalizationPartiallyAnisotropic}.
\end{proof}

We turn to the proof of Lemma \ref{lem:AuxDiagonalizationStructuredAnisotropic}.
\begin{proof}[Proof of Lemma \ref{lem:AuxDiagonalizationStructuredAnisotropic}]
For the proof of \eqref{eq:AuxDiagonalizationI}, we use symbol composition to write
\begin{equation*}
\mathcal{M}_\lambda \tilde{S}_B \mathcal{N}_\lambda S_B = S_B + E_\lambda
\end{equation*}
with $E_\lambda = OP(e_\lambda)$ and $e_\lambda$ given by the asymptotic expansion
\begin{equation*}
e_\lambda = \frac{1}{\alpha !} \sum_{|\alpha| \geq 1} D_\xi^\alpha (\tilde{m}(x,\xi) \tilde{\chi}_B(\xi)) \partial_x^\alpha (\tilde{m}^{-1}(x,\xi) \chi_B(\xi)).
\end{equation*}
This converges for similar reasons as in the proof of Proposition \ref{prop:StructuredDiagonalizationPartiallyAnisotropic}. Derivatives $\partial_x^\alpha \tilde{m}^{-1}(x,\xi)$ vanish for $\alpha = (\alpha_0,\alpha_1,\alpha_2,\alpha_3)$ with $|\alpha_2| + |\alpha_3| > 0$. We have
\begin{equation*}
|D_\xi^\alpha (\tilde{m}(x,\xi) \tilde{\chi}_B(\xi))| \lesssim \lambda^{-|\alpha|} \text{ for } |\alpha_2| = |\alpha_3| = 0
\end{equation*}
and moreover,
\begin{equation*}
|\partial_x^\alpha \tilde{m}^{-1}(x,\xi)| \lesssim \lambda^{\frac{|\alpha|}{2}}.
\end{equation*}
We have for the leading order term
\begin{equation*}
\sum_{|\alpha| = 1} (D_\xi^\alpha \tilde{m}(x,\xi) \tilde{\chi}_B(\xi)) (\partial_x^\alpha \tilde{m}^{-1}(x,\xi) \chi_B(\xi)) \in S^{-1}_{\frac{1}{2},\frac{1}{2}}
\end{equation*}
because the coefficients are Lipschitz. Hence, $e_\lambda \in S^{-1}_{\frac{1}{2},\frac{1}{2}}$, and we obtain
\begin{equation*}
\lambda^{-\rho} \| S_B u \|_{L^p L^q} \lesssim \lambda^{-\rho} \| \mathcal{M}_\lambda \tilde{S}_B \mathcal{N}_\lambda S_B u \|_{L^p L^q} + \lambda^{-\rho} \| E_\lambda u \|_{L^p L^q}.
\end{equation*} 
By Sobolev embedding and the Calderon--Vaillancourt theorem, we have
\begin{equation*}
\lambda^{-\rho} \| E_\lambda u \|_{L^p L^q} \lesssim \lambda^{\frac{1}{2}} \| E_\lambda u \|_{L^2_x} \lesssim \lambda^{-\frac{1}{2}} \| S_\lambda u \|_{L^2_x}.
\end{equation*}
We still have to estimate
\begin{equation*}
\lambda^{-\rho} \| \mathcal{M}_\lambda \tilde{S}_B \mathcal{N}_\lambda S_B u \|_{L^p L^q} \lesssim \lambda^{-\rho} \| \mathcal{N}_\lambda S_B u \|_{L^p L^q} + \| S_B u \|_{L^2_x}.
\end{equation*}
For the proof of $L^p L^q$-bounds for $[\mathcal{M}_\lambda]_{ij}$ we write the components as composition of operators, for which $L^p L^q$-bounds are straight-forward because these are differential operators, Riesz transforms, or amenable to Lemma \ref{lem:LpLqBoundsPseudos}. The error terms, however, gain a factor $\lambda^{-\frac{1}{2}}$, and therefore can be estimated by Sobolev embedding. We note that the components of $\mathcal{M}_{ij}$for $(i,j) \in \{1,\ldots,6\}^2 \backslash \{(1,5),(1,6),(4,3),(4,4) \}$ can be written as linear combinations of products of symbols in $S^0_{1,\frac{1}{2}}$ and $\frac{\partial_i}{D_{23}} S_B$ for $i=2,3$. E.g.,
\begin{equation*}
\mathcal{M}_{23} S_B = \frac{i}{D^{\frac{1}{2}}} D_\varepsilon^{\frac{1}{2}} \frac{\partial_3}{D_{23}} \big( \frac{1}{\sqrt{b_{\leq \lambda^{\frac{1}{2}}}}} \cdot \big) S_B = OP(a_1) OP(a_2) OP(a_3) + E_1
\end{equation*}
with
\begin{equation*}
a_1 = i \frac{\| \xi' \|^{\frac{1}{2}}_\varepsilon(x)}{\| \xi'\|^{\frac{1}{2}} } \chi_B(\xi'), \quad a_2 = \frac{i \xi_3}{\sqrt{\xi_2^2 + \xi_3^2}} \chi_B(\xi'), \quad a_3 = \frac{1}{\sqrt{b_{\leq \lambda^{\frac{1}{2}}}}} \chi_B(\xi').
\end{equation*}
The boundedness of $a_1$ in $L_t^p L_{x'}^q$ follows from Lemma \ref{lem:LpLqBoundsPseudos}. For $a_2$ this follows from boundedness of Riesz transforms in $L^p$ for $1<p<\infty$ and for $a_3$ this is trivial.

The error terms obtained in $E_1$ are of the form $\lambda^{-\frac{1}{2}} OP(e_1)$ with $e_1 \in S^0_{\frac{1}{2},\frac{1}{2}}$. The additional factor $\lambda^{-\frac{1}{2}}$ comes from the coefficients being Lipschitz. Therefore, we may estimate
\begin{equation*}
\| OP(a_1) OP(a_2) OP(a_3) f \|_{L_t^p L_{x'}^q} \lesssim \| f \|_{L_t^p L_{x'}^q}
\end{equation*}
and for the error term
\begin{equation*}
\lambda^{-\rho} \| E_1 f \|_{L_t^p L_{x'}^q} \lesssim \| f \|_{L^2_x}.
\end{equation*}
This shows
\begin{equation*}
\lambda^{-\rho} \| \mathcal{M}_\lambda \tilde{S}_B \mathcal{N}_\lambda S_B u \|_{L_t^p L_{x'}^q} \lesssim \lambda^{-\rho} \| \mathcal{N}_\lambda S_B u \|_{L_t^p L_{x'}^q} + \| S_B u \|_{L^2_x}.
\end{equation*}

We turn to the proof of \eqref{eq:AuxDiagonalizationII}. This is shown component-wise. For the first and second component $i=1,2$ we compute by Sobolev embedding
\begin{equation}
\label{eq:ChargeEstimate}
\lambda^{-\rho} \| [\mathcal{N}_\lambda S_B u]_i \|_{L_t^p L_{x'}^q} \lesssim \lambda^{\frac{1}{2}} \| [\mathcal{N}_\lambda S_B u]_i \|_{L^2_x} \lesssim \lambda^{-\frac{1}{2}} \| \rho_{em} \|_{L^2_x}.
\end{equation}
The ultimate estimate follows from $[\mathcal{N}_\lambda S_B u]_1 = \frac{1}{|D'|} \nabla_{x'} \cdot S'_\lambda S_B \mathcal{D}$ and $[\mathcal{N}_\lambda S_B u]_2 = \frac{1}{|D'_\varepsilon|} \nabla_{x'} \cdot S'_\lambda S_B \mathcal{H}$. The estimate for $j=3,\ldots,6$ is a consequence of Proposition \ref{prop:HalfWaveEstimates}:
\begin{equation}
\label{eq:NondegenerateEstimate}
\lambda^{-\rho} \| [\mathcal{N}_\lambda S_B u]_j \|_{L_t^p L_{x'}^q} \lesssim \| [ \mathcal{D}_\lambda \tilde{S}_B \mathcal{N}_\lambda S_B u]_j \|_{L^2_x} + \| [\mathcal{N}_\lambda S_B u]_j \|_{L^2_x}.
\end{equation}
Taking \eqref{eq:ChargeEstimate} and \eqref{eq:NondegenerateEstimate} together yields \eqref{eq:AuxDiagonalizationII}.

Finally, we show \eqref{eq:AuxDiagonalizationIII}. By a similar argument as in the proof of \eqref{eq:AuxDiagonalizationI}, we find
\begin{equation*}
\mathcal{N}_\lambda \tilde{S}_B \mathcal{M}_\lambda S_B = S_B + E_\lambda \text{ with } \| E_\lambda \|_{L^2 \to L^2} \lesssim \lambda^{-1}.
\end{equation*}
Since $\mathcal{N}_\lambda \tilde{S}_B \in OP S^0_{\frac{1}{2},\frac{1}{2}}$, we can apply the Calderon--Vaillancourt theorem to find
\begin{equation*}
\| S_B v \|_{L^2_x} \lesssim \| \mathcal{N}_\lambda \tilde{S}_B \mathcal{M}_\lambda S_B v \|_{L^2_x} + \| E_\lambda S_B v \|_{L^2_x} \lesssim \| \mathcal{M}_\lambda S_B v \|_{L^2_x} + \lambda^{-1} \| S_B v \|_{L^2_x}.
\end{equation*}
Absorbing $\lambda^{-1} \| S_B v \|_{L^2_x}$ into the left hand-side finishes the proof.
\end{proof}

\subsection{Proof of Theorem \ref{thm:PartiallyAnisotropicStrichartz}}

This subsection is devoted to the proof of Theorem \ref{thm:PartiallyAnisotropicStrichartz}.

\begin{proof}[Proof~of~Theorem~\ref{thm:PartiallyAnisotropicStrichartz}]
We carry out the following steps to reduce to a dyadic estimate:
\begin{itemize}
\item Reduction to high frequencies,
\item Microlocal estimate away from the characteristic surface,
\item Reduction to dyadic estimate and truncating the frequencies of the coefficients to $\lambda^{\frac{1}{2}}$.
\end{itemize}
We first handle the more difficult case $\partial \varepsilon \in L_T^2 L_{x'}^\infty$ and then turn to $\partial \varepsilon \in L_T^\infty L_{x'}^\infty$.
\subsubsection{$\partial \varepsilon \in L_T^2 L_{x'}^\infty$}
\label{subsubsection:IntegrableCoefficients}
$\quad$

\emph{Reduction to dyadic estimate and frequency truncation.}

\noindent It suffices to prove estimates close to the forbidden endpoint $(p,q) = (2,\infty)$, $\delta = 0$:
\begin{equation}
\label{eq:DyadicEstimateAnisotropic}
\lambda^{-\frac{5}{4} - \delta} \| S_\lambda S'_\lambda u \|_{L_T^p L^q_{x'}} \lesssim_{T,\delta} \| S_\lambda S'_\lambda u \|_{L^\infty_T L^2_{x'}} + \lambda^{-\frac{1}{2}} \| P_{\lambda} S_\lambda S'_\lambda u \|_{L_T^2 L^2_{x'}} + \lambda^{-\frac{3}{4}} \| S'_\lambda \rho_{em} \|_{L_t^\infty L_{x'}^2}.
\end{equation}
In the above display the coefficients of $P$ are truncated at frequencies $\lesssim \lambda^{\frac{1}{2}}$, $(p,q)$ denotes a sharp Strichartz pair and $\delta > 0$.

To see how \eqref{eq:DyadicEstimateAnisotropic} implies \eqref{eq:StrichartzPartiallyAnisotropicL2Lipschitz}, we note that
\begin{equation*}
\tilde{S}_\lambda \tilde{S}'_\lambda P_{\lambda} S_\lambda S'_\lambda u = P_\lambda S_\lambda S'_\lambda u.
\end{equation*}
By $\tilde{S}_{\lambda}$ the mildly enlarged frequency projection is denoted, likewise for $S'_\lambda$. Now we write
\begin{equation*}
\tilde{S}_\lambda \tilde{S}'_\lambda P_\lambda S_\lambda S'_\lambda = \tilde{S}_\lambda \tilde{S}'_\lambda P S_\lambda S'_\lambda  - \tilde{S}_\lambda \tilde{S}'_\lambda P_{\lambda^{\frac{1}{2}} \lesssim \cdot \lesssim \lambda} S_\lambda S'_\lambda.
\end{equation*}
The contribution of the second term is estimated by
\begin{equation*}
\begin{split}
\| \nabla \times (\varepsilon^{-1}_{\lambda^{\frac{1}{2}} \lesssim \cdot \lesssim \lambda} S_\lambda S'_\lambda \mathcal{D})) \|_{L^2_x} &\lesssim \| \partial \varepsilon^{-1} \|_{L_t^2 L_{x'}^\infty} \| S_\lambda S'_\lambda u \|_{L_t^\infty L_{x'}^2} + \lambda \| \varepsilon^{-1}_{\gtrsim \lambda^{\frac{1}{2}}} \|_{L_t^2 L_{x'}^\infty} \| S_\lambda S'_\lambda u \|_{L_t^\infty L_{x'}^2} \\
&\lesssim \lambda^{\frac{1}{2}} (1+ \| \partial \varepsilon \|_{L_t^2 L_{x'}^\infty} ) \| S_\lambda S'_\lambda u \|_{L_t^\infty L_{x'}^2}.
\end{split}
\end{equation*}
For the first term, we note that
\begin{equation*}
\begin{split}
\| \tilde{S}_\lambda \tilde{S}'_\lambda P S_\lambda S'_\lambda u \|_{L^2_x} &\leq \| \tilde{S}_\lambda \tilde{S}'_\lambda P u \|_{L^2_x} + \| \tilde{S}_\lambda \tilde{S}'_\lambda \nabla \times ([\varepsilon^{-1},S_\lambda S'_\lambda] \mathcal{D}) \|_{L^2_x} \\
&\lesssim \| \tilde{S}_\lambda \tilde{S}'_\lambda P u \|_{L^2_x} + \| u \|_{L_t^\infty L_{x'}^2}.
\end{split}
\end{equation*}
The second estimate follows from
\begin{equation*}
\| \tilde{S}_\lambda \tilde{S}'_\lambda \nabla \times ([\varepsilon^{-1},S_\lambda S'_\lambda] \mathcal{D}) \|_{L^2_x} \lesssim \lambda \| [\varepsilon^{-1},S_\lambda S'_\lambda] \mathcal{D} \|_{L^2_x} \lesssim \| \mathcal{D} \|_{L_t^\infty L_{x'}^2},
\end{equation*}
which is based on the commutator estimate
\begin{equation*}
\| [\varepsilon^{-1} , S_\lambda S'_\lambda ] \|_{L_t^\infty L_{x'}^2 \to L^2_x} \lesssim \lambda^{-1} \| \partial \varepsilon^{-1} \|_{L_t^2 L_{x'}^\infty}
\end{equation*}
as a consequence of the kernel estimate (\cite[Lemma~2.3]{Tataru2000}). Hence, taking the supremum gives
\begin{equation*}
\sup_{\lambda \geq 1} \big( \lambda^{-\frac{5}{4}-\delta} \| S'_\lambda S_\lambda u \|_{L_T^p L_{x'}^q} \big) \lesssim \| u \|_{L_T^\infty L_{x'}^2} + \| |D|^{-\frac{1}{2}} P u \|_{L^2_x} + \| \langle D' \rangle^{-\frac{3}{4}} \rho_{em} \|_{L_T^\infty L_{x'}^2}.
\end{equation*}
Recall the estimate for the contribution away from the characteristic surface:
\begin{equation*}
\| \langle D' \rangle^{-\rho} S_{ |\tau| \gg \| \xi' \|} u \|_{L_T^p L_{x'}^q} \lesssim \| u \|_{L_T^\infty L_{x'}^2} + \| Pu \|_{L_T^2 L_{x'}^2}.
\end{equation*}
Since $\delta > 0$ was arbitrary, we find
\begin{equation*}
\| \langle D' \rangle^{-\frac{5}{4}-\delta} S_{|\tau| \lesssim \| \xi' \|} u \|_{L_T^p L_{x'}^q} \lesssim \| u \|_{L_T^\infty L_{x'}^2} + \| |D|^{-\frac{1}{2}} P u \|_{L_T^2 L_{x'}^2} + \| \langle D' \rangle^{-\frac{3}{4}} \rho_{em} \|_{L_T^\infty L_{x'}^2}
\end{equation*}
Applying this to homogeneous solutions (together with the better estimate away from the characteristic surface), we find
\begin{equation*}
\| \langle D' \rangle^{-\frac{5}{4}-\delta} u \|_{L_T^p L_{x'}^q} \lesssim \| u(0) \|_{L_{x'}^2} + \| \langle D' \rangle^{-\frac{3}{4}} \rho_{em}(0) \|_{L_{x'}^2}.
\end{equation*}
By Duhamel's formula and Minkowski's inequality, we find
\begin{equation*}
\begin{split}
\| \langle D' \rangle^{-\frac{5}{4}-\delta} u \|_{L_T^p L_{x'}^q} &\lesssim_{T,\delta} \| u \|_{L_T^\infty L_{x'}^2} + \| P u \|_{L_T^1 L^2_{x'}} \\
&\quad + \| \langle D' \rangle^{-\frac{3}{4}} \rho_{em}(0) \|_{L_{x'}^2} + \| \langle D' \rangle^{-\frac{3}{4}} \partial_t \rho_{em} \|_{L_T^1 L_{x'}^2}.
\end{split}
\end{equation*}

We interpolate the above display for the sharp Strichartz exponents with $p=2+\varepsilon$ with the energy estimate
\begin{equation*}
\| u \|_{L_T^\infty L_{x'}^2} \lesssim \| u(0) \|_{L_{x'}^2} + \| P u \|_{L_T^1 L_{x'}^2}
\end{equation*}
to find
\begin{equation*}
\begin{split}
\| \langle D' \rangle^{-\rho-\frac{1}{2p} -\delta} u \|_{L_T^p L_{x'}^q} &\lesssim \| u \|_{L_T^\infty L_{x'}^2} + \| P u \|_{L_T^1 L^2_{x'}} \\
&\quad + \| \langle D' \rangle^{-\frac{3}{4}} \rho_{em}(0) \|_{L_{x'}^2} + \| \langle D' \rangle^{-\frac{3}{4}} \partial_t \rho_{em} \|_{L_T^1 L_{x'}^2}.
\end{split}
\end{equation*}

\emph{Estimate without diagonalization.} The diagonalization becomes singular for \\
$|(\xi_2,\xi_3)| \ll \lambda$. We estimate the contribution of $A = \{ \xi' \in \R^3 : \| \xi' \| \sim \lambda, \quad |(\xi_2,\xi_3)| \lesssim \lambda^{\frac{3}{4}} \}$ directly by Bernstein's inequality. The volume of $A$ is given by $|A| \lesssim \lambda^{\frac{5}{2}}$. Let $S_A$ denote the corresponding smooth projection in Fourier space. Applying Bernstein's inequality gives
\begin{equation*}
\lambda^{- \frac{5}{4} - \delta} \| S_A S_\lambda S'_\lambda u \|_{L_T^2 L_{x'}^\infty} \lesssim \lambda^{-\frac{5}{4}-\delta} \lambda^{\frac{5}{4}} \| S_A S_\lambda S'_\lambda u \|_{L^2_x}.
\end{equation*}
We suppose in the following that $\mathcal{F}_{x} u$ is supported in $A^c \cap \{ \| \xi' \| \sim \lambda \} \cap \{ |\xi_0| \lesssim \lambda \}$.

\vspace*{0.3cm}

\emph{Estimate with diagonalization.} We denote by $S_B$ the frequency projection to $\{ |\xi_0| \lesssim \| \xi' \| \sim \lambda \} \cap \{|(\xi_2,\xi_3)| \gtrsim \lambda^{\frac{3}{4}} \}$. It suffices to show
\begin{equation}
\label{eq:SecondLocalization}
\lambda^{- \frac{5}{4} - \delta} \| S_B u \|_{L_T^p L_{x'}^q} \lesssim_{T,\delta} \| S_B u \|_{L_T^\infty L^2_{x'}} + \lambda^{-\frac{1}{2}} \| P_{\lambda} S_B u \|_{L^2_T L^2_{x'}} + \lambda^{-\frac{3}{4}} \| \rho_{em} \|_{L_T^\infty L_{x'}^2}.
\end{equation}
This requires an additional commutator estimate for $S_{|(\xi_2,\xi_3)| \gtrsim \lambda^{\frac{3}{4}}} = S''_{\gtrsim \lambda^{\frac{3}{4}}}$ with $\varepsilon$. Write $S''_{\gtrsim \lambda^{\frac{3}{4}}} = Id - S''_{\lesssim \lambda^{\frac{3}{4}}}$. This way we find
\begin{equation*}
\begin{split}
\| [\varepsilon, S''_{\lesssim \lambda^{\frac{3}{4}}}] f \|_{L^2_x} &= \| \| [\varepsilon(t,\cdot),S''_{\lesssim \lambda^{\frac{3}{4}}}] f(t,\cdot) \|_{L^2_{x'}} \|_{L_t^2} \\
&\lesssim \| \| \partial \varepsilon(t) \|_{L_{x'}^\infty} \lambda^{-\frac{3}{4}} \| f(t,\cdot) \|_{L^2_{x'}} \|_{L_t^2} \\
&\lesssim \lambda^{-\frac{3}{4}} \| \partial \varepsilon \|_{L_t^2 L_{x'}^\infty} \| f \|_{L_t^\infty L_{x'}^2}.
\end{split}
\end{equation*}
With the extra smoothing of $\lambda^{-\frac{1}{2}}$ of $Pu$ we see that it suffices to prove \eqref{eq:SecondLocalization}.

Due to this frequency truncation and localization away from the singular set, we can use the diagonalization because
\begin{equation*}
\tilde{m}(x,\xi) \chi_B(\xi), \quad \tilde{m}^{-1}(x,\xi) \chi_B(\xi) \in S^0_{\frac{3}{4},\frac{1}{2}}.
\end{equation*}
Indeed, taking derivatives in $\xi_2$ and $\xi_3$ gains factors of $\lambda^{-\frac{3}{4}}$ (derivatives in $\xi_1$ are better behaved and gain factors of $\lambda^{-1}$) and derivatives in $x$ yield factors of $\lambda^{\frac{1}{2}}$ because
\begin{equation*}
\| \partial^\alpha \varepsilon_{\leq \lambda^{\frac{1}{2}}} \|_{L^\infty_x} \lesssim \lambda^{\frac{|\alpha|}{2}} \| \varepsilon_{\leq \lambda^{\frac{1}{2}}} \|_{L^\infty_x}.
\end{equation*}
It is important to realize that for the first derivative we find by Bernstein's inequality the better estimate
\begin{equation}
\label{eq:ImprovedDerivativeEstimate}
\| \partial \varepsilon_{\leq \lambda^{\frac{1}{2}}} \|_{L^\infty_x} \lesssim \lambda^{\frac{1}{4}} \| \partial \varepsilon_{\leq \lambda^{\frac{1}{2}}} \|_{L^2_t L^\infty_{x'}}.
\end{equation}

We want to apply the diagonalization for $|(\xi_2,\xi_3)| \gtrsim \lambda^{\frac{3}{4}}$. For this purpose, we show the following lemma:
\begin{lemma}
\label{lem:AuxDiagonalizationL2Lipschitz}
With the notations from above, we find the following estimates to hold:
\begin{align}
\label{eq:DiagonalizationAnisotropicI}
\lambda^{-\frac{5}{4}-\delta} \| S_B u \|_{L^p_t L^q_{x'}} &\lesssim \lambda^{-\frac{5}{4}-\delta} \| \tilde{S}_B \mathcal{N}_\lambda S_B u \|_{L^p_t L^q_{x'}} + \| S_B u \|_{L^2_x}, \\
\label{eq:DiagonalizationAnisotropicII}
\lambda^{-\frac{5}{4}-\delta} \| \tilde{S}_B \mathcal{N}_\lambda S_B u \|_{L_t^p L_{x'}^q} &\lesssim \| S_B u \|_{L_t^\infty L_{x'}^2} + \lambda^{-\frac{1}{2}} \| \mathcal{D}_\lambda \tilde{S}_B \mathcal{N}_\lambda S_B u \|_{L^2_x} \\
&\quad + \lambda^{-\frac{3}{4}} \| S'_\lambda \rho_{em} \|_{L^\infty_t L^2_{x'}}, \nonumber \\
\label{eq:DiagonalizationAnisotropicIII}
\lambda^{-\frac{1}{2}} \| \mathcal{D}_{\lambda} \tilde{S}_B \mathcal{N}_\lambda S_B u \|_{L^2_x} &\lesssim \lambda^{-\frac{1}{2}} \| P_{\lambda} S_B u \|_{L^2_x} + \| S_B u \|_{L^\infty_t L^2_{x'}}.
\end{align}
\end{lemma}
The lemma is the analog of Lemma \ref{lem:AuxDiagonalizationStructuredAnisotropic}. The present asymptotic expansions are worse compared to Section \ref{subsection:StructuredAnisotropic}, which is mitigated by the additional smoothing factor of $\lambda^{-\frac{1}{4}}$ on the left hand side and $\lambda^{-\frac{1}{2}}$ for the forcing term.
\begin{proof}
For the proof of \eqref{eq:DiagonalizationAnisotropicI} we use symbol composition and the asymptotic expansion of $\mathcal{M}_\lambda \tilde{S}_B \mathcal{N}_\lambda S_B = S_B + \tilde{E}_\lambda$. We have $\mathcal{M}_\lambda \tilde{S}_B \in Op(S^0_{\frac{3}{4},\frac{1}{2}})$, $\mathcal{N}_\lambda S_B \in Op(S^0_{\frac{3}{4},\frac{1}{2}})$. Hence, symbol composition holds and we compute for the leading order term
\begin{equation*}
\tilde{E}_\lambda = \lambda^{-\frac{1}{2}} E_\lambda \text{ with } E_\lambda \in Op(S^0_{\frac{3}{4},\frac{1}{2}}).
\end{equation*}
The additional gain comes from \eqref{eq:ImprovedDerivativeEstimate}. Thus, the error term can be estimated by Sobolev embedding:
\begin{equation*}
\lambda^{-\frac{5}{4}-\delta} \| \tilde{E}_\lambda S_B u \|_{L^p_t L^q_{x'}} \lesssim \lambda^{-\delta} \| E_\lambda S_B u \|_{L^2_x} \lesssim \| S_B u \|_{L^2_x}.
\end{equation*}
The ultimate estimate follows from the Calderon--Vaillancourt theorem.

Thus, for the proof of \eqref{eq:DiagonalizationAnisotropicI} we have yet to show
\begin{equation*}
\lambda^{-\frac{5}{4}-\delta} \| \mathcal{M}_\lambda \tilde{S}_B \mathcal{N}_\lambda S_B u \|_{L_t^p L_{x'}^q} \lesssim \lambda^{-\frac{5}{4}-\delta} \| \mathcal{N}_\lambda S_B u \|_{L_t^p L_{x'}^q} + \| S_B u \|_{L^2_x}.
\end{equation*}
To this end, we write $\mathcal{M}_\lambda S_B$ as composition of pseudo-differential operators, which can be bounded on $L_t^p L_{x'}^q$ for $2 \leq p,q < \infty$ because these are Riesz transforms, differential operators, or by Lemma \ref{lem:LpLqBoundsPseudos}. The error terms arising in symbol composition gain $\frac{1}{4}$ derivatives (again essentially due to \eqref{eq:ImprovedDerivativeEstimate}). We note that the components of $\mathcal{M}_{ij}$ for $(i,j) \in \{1,\ldots,6\}^2 \backslash \{(1,5),(1,6),(4,3),(4,4) \}$ can be written as linear combinations of products of symbols in $S^0_{1,\frac{1}{2}}$ and $\frac{\partial_i}{D_{23}} S_B$ for $i=2,3$. An appropriate splitting of components of $\mathcal{M}$ and $\mathcal{N}$ for this argument is provided in the Appendix. E.g.,
\begin{equation*}
\mathcal{M}_{23} S_B = \frac{i}{D^{\frac{1}{2}}} D_\varepsilon^{\frac{1}{2}} \frac{\partial_3}{D_{23}} \big( \frac{1}{\sqrt{b_{\leq \lambda^{\frac{1}{2}}}}} \cdot \big) S_B = OP(a_1) OP(a_2) OP(a_3) + E_1
\end{equation*}
with
\begin{equation*}
a_1 = i \frac{\| \xi' \|_{\varepsilon(x)}}{\| \xi'\|} \chi_\lambda(\xi_0) \chi_\lambda(\xi'), \quad a_2 = \frac{i \xi_3}{\sqrt{\xi_2^2 + \xi_3^2}} \chi_B(\xi'), \quad a_3 = \frac{1}{\sqrt{b_{\leq \lambda^{\frac{1}{2}}}}} \chi_B(\xi').
\end{equation*}
The boundedness of $a_1$ in $L_t^p L_{x'}^q$ follows from Lemma \ref{lem:LpLqBoundsPseudos}. For $a_2$ this follows from boundedness of Riesz transforms in $L^q$ for $1<q<\infty$ and for $a_3$ this is trivial.

The error terms obtained in $E_1$ are of the form $\lambda^{-\frac{1}{4}} OP(e_1)$ with $e_1 \in S^0_{\frac{3}{4},\frac{1}{2}}$. The additional $\lambda^{-\frac{1}{4}}$ gain follows from \eqref{eq:ImprovedDerivativeEstimate} and derivatives in $\xi'$ at least yield factors $\lambda^{-\frac{3}{4}}$. Therefore, we may estimate
\begin{equation*} 
\| OP(a_1) OP(a_2) OP(a_3) f \|_{L_t^p L_{x'}^q} \lesssim \| f \|_{L_t^p L_{x'}^q}
\end{equation*}
and by Sobolev embedding and the Calderon--Vaillancourt theorem we find
\begin{equation*}
\lambda^{-\frac{5}{4}} \| E_1 f \|_{L_t^p L_{x'}^q} \lesssim \| e_1 f \|_{L^2_x} \lesssim \| f \|_{L^2_x}.
\end{equation*}
It remains to check the contributions of $\mathcal{M}_{15}$, $\mathcal{M}_{16}$, $\mathcal{M}_{43}$, $\mathcal{M}_{44}$. Here we write
\begin{equation*}
\mathcal{M}_{15} S_B = OP(a_1 \tilde{\chi}_B(\xi) a_2 \chi_B(\xi)) = OP(a_1 \tilde{\chi}_B) OP(a_2 \chi_B) + \lambda^{-\frac{1}{4}} OP(e_1)
\end{equation*}
with $e_1 \in S^0_{\frac{3}{4},\frac{1}{2}}$. Clearly,
\begin{equation*}
\| OP(a_1 \tilde{\chi}_B) \|_{L_t^p L_{x'}^q \to L_t^p L_{x'}^q} \lesssim \lambda^{-1}, \quad \| OP(a_2 \tilde{\chi}_B) \|_{L_t^p L_{x'}^q \to L_t^p L_{x'}^q} \lesssim \lambda,
\end{equation*}
which estimates the leading order term. The error term is estimated like above via Sobolev embedding and the Calderon--Vaillancourt theorem. The estimates of $\mathcal{M}_{16}$, $\mathcal{M}_{43}$, $\mathcal{M}_{44}$ follow likewise.

We turn to the proof of \eqref{eq:DiagonalizationAnisotropicIII}:
We use symbol composition to write
\begin{equation*}
S_B v_\lambda = \mathcal{N}_\lambda \tilde{S}_B \mathcal{M}_\lambda S_B v_\lambda + E_\lambda S_B v_\lambda
\end{equation*}
with $E_\lambda = \lambda^{-\frac{1}{2}} Op(e_\lambda)$, $e_\lambda \in S^0_{\frac{3}{4},\frac{1}{2}}$. Hence, we can estimate by Minkowski's inequality and the Calderon--Vaillancourt theorem:
\begin{equation*}
\begin{split}
\| S_B v_\lambda \|_{L^2_x} &\lesssim \| \mathcal{N}_\lambda \tilde{S}_B \mathcal{M}_\lambda S_B v_\lambda \|_{L^2_x} + \| E_\lambda S_B v_\lambda \|_{L^2_x} \\
&\lesssim \| \mathcal{N}_\lambda \tilde{S}_B \mathcal{M}_\lambda S_B v_\lambda \|_{L^2}+ \lambda^{-\frac{1}{2}} \| S_B v_\lambda \|_{L^2_x}.
\end{split}
\end{equation*}

We absorb the error term into the left hand-side to find
\begin{equation*}
\| S_B v_\lambda \|_{L^2_x} \lesssim \| \mathcal{N}_\lambda \tilde{S}_B \mathcal{M}_\lambda S_B v_\lambda \|_{L^2_x}.
\end{equation*}
$L^2$-boundedness of $\mathcal{N}_\lambda \tilde{S}_B$ follows because its symbol is in $S^0_{\frac{3}{4},\frac{1}{2}}$.
We have argued that
\begin{equation*}
\lambda^{-\frac{1}{2}} \| \mathcal{D}_\lambda \tilde{S}_B \mathcal{N}_\lambda S_B u \|_{L^2_x} \lesssim \lambda^{-\frac{1}{2}} \| \mathcal{M}_\lambda \tilde{S}_B \mathcal{D}_\lambda \tilde{S}_B \mathcal{N}_\lambda S_B u \|_{L^2_x} + \| S_B u \|_{L^2_x}.
\end{equation*}
To conclude, we shall show that 
\begin{equation*}
\mathcal{M}_\lambda \tilde{S}_B \mathcal{D}_\lambda \tilde{S}_B \mathcal{N}_\lambda S_B = P_{\lambda} S_B + E_\lambda
\end{equation*}
with $\| E_\lambda \|_{L^2 \to L^2} \lesssim \lambda^{\frac{1}{2}}$.

We apply symbol composition by Theorem \ref{thm:KohnNirenberg} to find that $E_\lambda = Op(\lambda^{-\frac{1}{4}} e_\lambda)$ with $e_\lambda \in S^\frac{3}{4}_{\frac{3}{4},\frac{1}{2}}$. The additional gain of $\lambda^{-\frac{1}{4}}$ stems from \eqref{eq:ImprovedDerivativeEstimate}. Since $\lambda^{-\frac{1}{2}} E_\lambda \in S^0_{\frac{3}{4},\frac{1}{2}}$, we can finish the proof by appealing to the Calderon--Vaillancourt theorem.

We turn to the proof of \eqref{eq:DiagonalizationAnisotropicII}:
\begin{equation*}
\begin{split}
\lambda^{-\frac{5}{4}-\delta} \| \tilde{S}_B \mathcal{N}_\lambda S_B u \|_{L_t^2 L_{x'}^\infty} &\lesssim \| \tilde{S}_B \mathcal{N}_\lambda S_B u \|_{L_t^\infty L^2_{x'}} + \lambda^{-\frac{1}{2}} \| \tilde{S}_B \mathcal{D}_\lambda \tilde{S}_B \mathcal{N}_\lambda S_B u \|_{L^2_x} \\
&\quad + \lambda^{-\frac{3}{4}} \| S'_\lambda \rho_{em} \|_{L^\infty_t L^2_{x'}}.
\end{split}
\end{equation*}
The first two components are estimated by Sobolev embedding and H\"older's inequality in time like in the isotropic case. For the remaining four components we shall prove
\begin{equation*}
\lambda^{-1-\delta} \| [\tilde{S}_B \mathcal{N}_\lambda S_B u]_j \|_{L_t^2 L_{x'}^q} \lesssim \lambda^{-\frac{1}{4}} \| [\tilde{S}_B \mathcal{N}_\lambda S_B u]_j \|_{L_t^\infty L_{x'}^2} + \lambda^{-\frac{1}{4}} \| [\tilde{S}_B \mathcal{D}_\lambda S_B \mathcal{N}_\lambda S_B u]_j \|_{L^2_x}.
\end{equation*}
For this purpose, we apply \cite[Theorem~5]{Tataru2000} on the level of half-wave equations.

The proof is complete.
\end{proof}
With the lemma at hand, we can conclude the proof of Theorem \ref{thm:PartiallyAnisotropicStrichartz} for $\partial \varepsilon \in L_T^2 L_{x'}^\infty$ in the similar spirit as for Theorems \ref{thm:IsotropicStrichartz} and \ref{thm:StructuredStrichartzPartiallyAnisotropic}. We omit the details to avoid repetition.

\subsubsection{Lipschitz coefficients}

Now we shall see how to modify the above argument to deal with Lipschitz coefficients and show estimates with slightly less derivative loss. After the usual reductions, we shall prove the dyadic estimate
\begin{equation}
\label{eq:DyadicEstimateLipschitzCoefficients}
\lambda^{-\frac{7}{6}-\delta} \| S_\lambda S'_\lambda u \|_{L_T^p L_{x'}^q} \lesssim_T \| S_\lambda S'_\lambda u \|_{L_T^2 L_{x'}^2} + \lambda^{-\frac{1}{3}} \| P_{\lambda^{\frac{2}{3}}} S_\lambda S'_\lambda u \|_{L_T^2 L_{x'}^2} + \lambda^{-\frac{2}{3}} \| S'_\lambda \rho_{em} \|_{L_T^\infty L_{x'}^2}.
\end{equation}
The coefficients of $P$ are truncated at frequencies $\lambda^{\frac{2}{3}}$. The frequency truncation at $\lambda^{\frac{2}{3}}$ will not be emphasized in the following anymore, and we simply write $P_\lambda$. We observe like above
\begin{equation*}
\tilde{S}_\lambda \tilde{S}'_\lambda P_\lambda S_\lambda S'_\lambda u = P_\lambda S_\lambda S'_\lambda u
\end{equation*}
and write
\begin{equation*}
\tilde{S}_\lambda \tilde{S}'_\lambda P_\lambda S_\lambda S'_\lambda = \tilde{S}_\lambda \tilde{S}'_\lambda P S_\lambda S'_\lambda - \tilde{S}_\lambda \tilde{S}'_\lambda P_{\lambda^{\frac{2}{3}}\lesssim \cdot \lesssim \lambda} S_\lambda S'_\lambda.
\end{equation*}

We estimate the contribution of the second term by
\begin{equation*}
\begin{split}
\| \nabla \times (\varepsilon^{-1}_{\lambda^{\frac{2}{3}} \lesssim \cdot \lesssim \lambda} S_\lambda S'_\lambda \mathcal{D}) \|_{L^2_x} &\lesssim \| \partial \varepsilon^{-1} \|_{L^\infty_x} \| S_\lambda S'_\lambda u \|_{L^2_x} \\
&\quad + \lambda \| \varepsilon^{-1}_{\geq \lambda^{\frac{2}{3}}} \|_{L^\infty_{x}} \| S_\lambda S'_\lambda u \|_{L^2_x} \\
&\lesssim \lambda^{\frac{1}{3}} ( 1 + \| \partial \varepsilon \|_{L_{T,x'}^\infty} ) \| S_\lambda S'_\lambda u \|_{L^2_x}.
\end{split}
\end{equation*}
For the first term we use a commutator argument
\begin{equation*}
\begin{split}
\| \tilde{S}_\lambda \tilde{S}'_\lambda P S_\lambda S'_\lambda u \|_{L^2_x} &\leq \| \tilde{S}_\lambda \tilde{S}'_\lambda P u \|_{L^2_x} + \| \tilde{S}_\lambda \tilde{S}'_\lambda \nabla \times ([\varepsilon^{-1}, S_\lambda S'_\lambda]) \mathcal{D} )\|_{L^2_x}
\\
&\lesssim \| \tilde{S}_\lambda \tilde{S}_\lambda' P u \|_{L^2_x} + \| u \|_{L^2_x}.
\end{split}
\end{equation*}
The second estimate follows by
\begin{equation*}
\| \tilde{S}_\lambda \tilde{S}'_\lambda \nabla \times ([\varepsilon^{-1}, S_\lambda S'_\lambda] \D) \|_{L^2_x} \lesssim \lambda \| [\varepsilon^{-1}, S_\lambda S'_\lambda] \D \|_{L^2_x} \lesssim \| \D \|_{L^2_{x}},
\end{equation*}
which is a consequence of
\begin{equation*}
\| [\varepsilon^{-1} , S_\lambda S'_\lambda ] \|_{L^2_x \to L^2_x} \lesssim \lambda^{-1} \| \partial \varepsilon \|_{L^\infty_x}
\end{equation*}
from a kernel estimate (cf. \cite[Eq.~(3.21)]{Tataru2000}). Hence, we find like in the beginning of Subsection \ref{subsubsection:IntegrableCoefficients}:
\begin{equation*}
\| \langle D' \rangle^{-\frac{7}{6}-\delta} S_{|\tau| \lesssim \| \xi' \|} u \|_{L_T^p L_{x'}^q} \lesssim_{T,\delta} \| u \|_{L_T^2 L_{x'}} + \| |D|^{-\frac{1}{3}} P u \|_{L_T^2 L_{x'}^2} + \| \langle D' \rangle^{-\frac{3}{4}} \rho_{em} \|_{L_T^\infty L_{x'}^2}.
\end{equation*}
Together with the better estimate for the contribution away from the characteristic surface, we find like in Subsection \ref{subsubsection:IntegrableCoefficients}:
% Hence, square summing \eqref{eq:DyadicEstimateLipschitzCoefficients} over $\lambda$ yields by Littlewood-Paley theory (in $x'$) the estimate
\begin{equation*}
\begin{split}
\| \langle D' \rangle^{-\rho-\frac{1}{2p}-\delta} u \|_{L_T^p L_{x'}^q} &\lesssim \| u(0) \|_{L_{x'}^2} + \| P u \|_{L_T^1 L_{x'}^2} \\
&\quad + \| \langle D' \rangle^{-\frac{2}{3}} \rho_{em}(0) \|_{L_{x'}^2} + \| \langle D' \rangle^{-\frac{2}{3}} \partial_t \rho_{em} \|_{L_T^1 L_{x'}^2}.
\end{split}
\end{equation*}

\emph{Estimate without diagonalization}.  The decreased additional smoothing of $\lambda^{-\frac{1}{6}}$ compared to $\lambda^{-\frac{1}{4}}$ compared to the previous case allows only to estimate the contribution
\begin{equation*}
A = \{ \xi' \in \R^3 : \| \xi' \| \sim \lambda, \; |(\xi_2,\xi_3)| \lesssim \lambda^{\frac{2}{3}+\varepsilon} \}
\end{equation*}
directly by Bernstein's inequality.
We suppose henceforth that $\mathcal{F}_x u$ is supported in $\{|(\xi_2,\xi_3)| \gtrsim \lambda^{\frac{2}{3}+\varepsilon}, \; \| \xi' \| \sim \lambda, \; |\xi_0| \lesssim \lambda \}$.

\emph{Estimate with diagonalization.} We denote by $S_B$ the frequency projection to $\{|\xi_0| \lesssim \| \xi' \| \sim \lambda \} \cap \{ |(\xi_2,\xi_3)| \gtrsim \lambda^{\frac{2}{3}+\varepsilon} \}$. It suffices to show
\begin{equation*}
\lambda^{-\frac{7}{6}-\delta} \| S_B u \|_{L_T^p L_{x'}^q} \lesssim \| S_B u \|_{L_{T,x}^2} + \lambda^{-\frac{1}{3}} \| P_\lambda S_B u \|_{L^2_T L_{x'}^2} + \lambda^{-\frac{2}{3}} \| S'_\lambda \rho_{em} \|_{L_T^\infty L_{x'}^2}.
\end{equation*}
The additional commutator estimate for $S''_{\gtrsim \lambda^{\frac{2}{3}}}$ gains $\lambda^{-\frac{2}{3}}$. By the frequency truncation and localization away from the singular set, we can use the diagonalization because
\begin{equation*}
\tilde{m}(x,\xi) \chi_B(\xi), \qquad \tilde{m}^{-1}(x,\xi) \chi_B(\xi) \in S^{0}_{\frac{2}{3}+\varepsilon,\frac{2}{3}}.
\end{equation*}
For Lipschitz coefficients we have the bound
\begin{equation}
\label{eq:ImprovedDerivativeBoundLipschitz}
\| \partial^\alpha \varepsilon_{\leq \lambda^{\frac{2}{3}}} \|_{L_x^\infty} \lesssim \lambda^{\frac{2}{3} (|\alpha| -1)_+} \| \varepsilon_{\leq \lambda^{\frac{2}{3}}} \|_{L_x^\infty}.
\end{equation}
Roughly speaking, error terms arising in first order symbol composition give smoothing factors of $\lambda^{-\frac{2}{3}}$. Together with the weight $\lambda^{-\frac{1}{3}}$ this allows to recover the whole derivative. The counterpart of Lemma \ref{lem:AuxDiagonalizationStructuredAnisotropic} reads as follows:
\begin{lemma}
\label{lem:AuxDiagonalizationLipschitz}
With the notations from above, we find the following estimate to hold:
\begin{align}
\label{eq:AuxDiagonalizationLipschitzI}
\lambda^{-\frac{7}{6}-\delta} \| S_B u \|_{L_T^p L_{x'}^q} &\lesssim \lambda^{-\frac{7}{6}-\delta} \| \tilde{S}_B \mathcal{N}_\lambda S_B u \|_{L_t^p L_{x'}^q} + \| S_B u \|_{L^2_x}, \\
\label{eq:AuxDiagonalizationLipschitzII}
\lambda^{-\frac{7}{6}-\delta} \| \tilde{S}_B \mathcal{N}_\lambda S_B u \|_{L_T^p L_{x'}^q} &\lesssim \| S_B u \|_{L_x^2} + \lambda^{-\frac{1}{3}} \| \mathcal{D}_\lambda \tilde{S}_B \mathcal{N}_\lambda S_B u \|_{L^2_x} \\
&\quad + \lambda^{-\frac{2}{3}} \| S'_\lambda \rho_{em} \|_{L_t^\infty L_{x'}^2}, \nonumber \\
\label{eq:AuxDiagonalizationLipschitzIII}
\lambda^{-\frac{1}{3}} \| \D_\lambda \tilde{S}_B \mathcal{N}_\lambda S_B u \|_{L^2_x} &\lesssim \lambda^{-\frac{1}{3}} \| P_\lambda S_B u \|_{L^2_x} + \| S_B u \|_{L^2_x}.
\end{align}
\end{lemma}
\begin{proof}
This is a reprise of the proof of Lemma \ref{lem:AuxDiagonalizationL2Lipschitz}. For the proof of \eqref{eq:AuxDiagonalizationLipschitzI} we again use symbol composition and the asymptotic expansion of $\mathcal{M}_\lambda \tilde{S}_B \mathcal{N}_\lambda S_B = S_B + \tilde{E}_\lambda $. We have $\mathcal{M}_\lambda \tilde{S}_B \in OPS^0_{\frac{2}{3}+\varepsilon,\frac{2}{3}}$, $\mathcal{N}_\lambda S_B \in OPS^0_{\frac{2}{3}+\varepsilon,\frac{2}{3}}$. Hence, symbol composition holds, and we compute for the leading order term $\tilde{E}_\lambda = \lambda^{-\frac{2}{3}-\varepsilon} E_\lambda$ with $E_\lambda \in OPS^0_{\frac{2}{3}+\varepsilon,\frac{2}{3}}$. The additional gain stems from \eqref{eq:ImprovedDerivativeBoundLipschitz}. Thus, the error term can be estimated by
\begin{equation}
\label{eq:ErrorTermEstimateLipschitzDiagonalization}
\lambda^{-\frac{7}{6}-\delta} \| \tilde{E}_\lambda S_B u \|_{L_t^p L_{x'}^q} \lesssim \lambda^{-\delta - \varepsilon} \| E_\lambda S_B u \|_{L^2_x} \lesssim \| S_B u \|_{L^2_x}.
\end{equation}
Note that for $\frac{2}{p} + \frac{2}{q} = 1$ there are at most $\frac{3}{2}$ derivatives required:
\begin{equation*}
\| S_\lambda u \|_{L_t^p L_{x'}^q} \lesssim \lambda^{\frac{3}{2}} \| S_\lambda u \|_{L^2_x}.
\end{equation*}
The ultimate estimate in \eqref{eq:ErrorTermEstimateLipschitzDiagonalization} follows from the Calderon--Vaillancourt therom. For the proof of \eqref{eq:AuxDiagonalizationLipschitzI} we have to show
\begin{equation*}
\lambda^{-\frac{7}{6}-\delta} \| \mathcal{M}_\lambda \tilde{S}_B \mathcal{N}_\lambda S_B u \|_{L_t^p L_{x'}^q} \lesssim \lambda^{-\frac{7}{6}-\delta} \| \mathcal{N}_\lambda S_B u \|_{L_t^p L_{x'}^q} + \| S_B u \|_{L^2_x}.
\end{equation*}
Like in the proof of \eqref{eq:DiagonalizationAnisotropicI}, we write $\mathcal{M}_\lambda S_B$ as composition of operators, which can be bounded on $L_t^p L_{x'}^q$ in a straight-forward way. The error terms arising in symbol composition gain $\frac{2}{3}+\varepsilon$ derivatives, which then suffices to estimate the remainder by Sobolev embedding. This finishes the proof of \eqref{eq:AuxDiagonalizationLipschitzI}.

We turn to the proof of \eqref{eq:AuxDiagonalizationLipschitzIII}. We write by symbol composition
\begin{equation*}
S_B v_\lambda  = \mathcal{N}_\lambda \tilde{S}_B \mathcal{M}_\lambda S_B v_\lambda + E_\lambda S_B v_\lambda
\end{equation*}
with $E_\lambda = \lambda^{-\frac{2}{3}-\varepsilon} Op(e_\lambda)$, $e_\lambda \in S^0_{\frac{2}{3}+\varepsilon,\frac{2}{3}}$. Like in the proof of \eqref{eq:DiagonalizationAnisotropicIII}, we can absorb the error term into the left hand-side to find
\begin{equation*}
\| S_B v_\lambda \|_{L^2_x} \lesssim \| \mathcal{N}_\lambda \tilde{S}_B \mathcal{M}_\lambda S_B v_\lambda \|_{L^2_x}.
\end{equation*}
$L^2$-boundedness of $\mathcal{N}_\lambda \tilde{S}_B$ follows again by the Calderon--Vaillancourt theorem. We have proved
\begin{equation*}
\lambda^{-\frac{1}{3}} \| \mathcal{D}_\lambda \tilde{S}_B \mathcal{N}_\lambda S_B u \|_{L^2_x} \lesssim \lambda^{-\frac{1}{3}} \| \mathcal{M}_\lambda \tilde{S}_B \mathcal{D}_\lambda \tilde{S}_B \mathcal{N}_\lambda S_B u \|_{L^2_x} + \| S_B u_\lambda \|_{L^2_x}.
\end{equation*}
We still have to show that
\begin{equation*}
\mathcal{M}_\lambda \tilde{S}_B \mathcal{D}_\lambda \tilde{S}_B \mathcal{N}_\lambda S_B = P_\lambda S_B + E_\lambda
\end{equation*}
with $\| E_\lambda \|_{L^2_x \to L^2_x} \lesssim \lambda^{\frac{1}{3}}$.
We apply symbol composition to find that $E_\lambda = \lambda^{-\frac{2}{3}-\varepsilon} OP(e_\lambda)$ with $e_\lambda \in S^{1-\varepsilon}_{\frac{2}{3}+\varepsilon,\frac{2}{3}}$. Hence, the proof is concluded by applying the Calderon--Vaillancourt theorem.
\end{proof}
With Lemma \ref{lem:AuxDiagonalizationLipschitz} at hand, we can prove \eqref{eq:DyadicEstimateLipschitzCoefficients} by similar means like in Subsection \ref{subsubsection:IntegrableCoefficients}. This concludes the proof of Theorem \ref{thm:PartiallyAnisotropicStrichartz}.
\end{proof}

\section{Improved local well-posedness for quasilinear Maxwell equations}
\label{section:LocalWellposednessQuasilinearMaxwell}
\subsection{The simplified Kerr model}
\label{subsection:SimplifiedKerr}
In the following we shall analyze the system of equations:
\begin{equation}
\label{eq:SimplifiedKerrComponent}
\partial_t^2 u + \nabla \times (\varepsilon(u) \nabla \times u) = 0, \quad \nabla_{x'} \cdot u  = 0
\end{equation}
for $u: \R \times \R^3 \to \R^3$ and $\varepsilon \in C^\infty(\R^3;\R_{>0})$.

In the first step we modify the proof of the Strichartz estimates for the first order system in case of isotropic permittivity to show the following:
\begin{theorem}
\label{thm:StrichartzEstimatesSimplifiedKerr}
Let $\varepsilon: \R \times \R^3 \to \R_{>0}$, $\varepsilon \in C^1(\R \times \R^3)$. Suppose there are $\Lambda_1, \Lambda_2 > 0$  such that for any $x \in \R^4$ we have $\Lambda_1 \leq \varepsilon (x) \leq \Lambda_2$. Let $P(x,D) = \partial_t^2 + \nabla \times (\varepsilon \nabla \times \cdot)$ with $u=(u_1,u_2,u_3):\R \times \R^3 \to \R^3$ and $\nabla_{x'} \cdot u = \rho_e$. Then, the following Strichartz estimates hold:
\begin{equation*}
\begin{split}
\| |D'|^{-\rho} \nabla_x u \|_{L_T^p L_{x'}^q} &\lesssim \nu^{\frac{1}{p}} \| \nabla_x u \|_{L^\infty L_{x'}^2} + \nu^{-\frac{1}{p'}} \| P(x,D) u \|_{L_T^1 L_{x'}^2} \\
&\quad + T^{\frac{1}{p}} ( \| |D'|^{\frac{1}{p}} \rho_e \|_{L_T^\infty L_{x'}^2} + \| |D'|^{\frac{1}{p}} \partial_t \rho_e \|_{L_T^1 L_{x'}^2})
\end{split}
\end{equation*}
provided that $(\rho,p,q,3)$ is Strichartz admissible, $\nu \geq 1$, and $T \| \partial_x^2 \varepsilon \|_{L_t^1 L_{x'}^\infty} \leq \nu^2$.
\end{theorem}
This gives the following corollary for coefficients in $L_T^2 L_{x'}^\infty$ via the usual paradifferential decomposition (cf. \cite{Tataru2002}). The proof is omitted.
\begin{corollary}
Assume that $\partial \varepsilon \in L_T^2 L_{x'}^\infty$ and $(\rho,p,q,3)$ be Strichartz admissible. Then, the following estimated holds for $T,\delta > 0$:
\begin{equation*}
\begin{split}
\| \langle D' \rangle^{-\rho - \frac{1}{3p}-\delta} \nabla_x u \|_{L_t^p(0,T; L^q)} &\lesssim_{T,\delta} \| \nabla_x u \|_{L_T^\infty L_{x'}^2} + \| P(x,D) u \|_{L_T^1 L_{x'}^2} \\
&\quad + \| \langle D' \rangle^{\frac{1}{p}} \rho_e \|_{L_T^\infty L_{x'}^2} + \| \langle D' \rangle^{\frac{1}{p}} \partial_t \rho_e \|_{L_T^1 L_{x'}^2}.
\end{split}
\end{equation*}
\end{corollary}

\begin{proof}[Proof~of~Theorem~\ref{thm:StrichartzEstimatesSimplifiedKerr}]
By the arguments of \cite{Tataru2002}, which apply for the coupled system of wave equations as well, we can reduce to the dyadic estimate
\begin{equation}
\label{eq:DyadicEstimateSimplifiedKerr}
\lambda^{1-\rho} \| S_\lambda S'_\lambda u \|_{L^p L^q} \lesssim \lambda \| S_\lambda S'_\lambda u \|_{L^\infty L^2} + \| P_\lambda(x,D) S_\lambda S'_\lambda u \|_{L^2} + \lambda^{\frac{1}{p}} \| S'_\lambda \rho_e \|_{L_T^\infty L_{x'}^2},
\end{equation}
where $P_\lambda$ denotes the operator with frequency truncated coefficients at $\lambda^{\frac{1}{2}}$, $\| \partial^2 \varepsilon \|_{L^1 L^\infty} \leq 1$, $T \leq 1$. The principal symbol of $P$ (the frequency truncation is omitted in the following to lighten the notation) is given by
\begin{equation*}
p(x,\xi) = -\xi_0^2 + \varepsilon(x) [ \| \xi' \|^2 1_{3 \times 3} - \xi' \otimes \xi'].
\end{equation*}

In the following we diagonalize the principal symbol like we did for first order Maxwell equations in the isotropic case. To this end, let $\xi_i^* = \xi_i / \| \xi' \|$ and $\xi_{ij}^2 = \xi_i^2 + \xi_j^2$ for $i,j=1,2,3$. Fix smooth functions $\phi_{i}: \mathbb{S}^{2} \to \R_{\geq 0}$ such that $\phi_{1} + \phi_2 + \phi_3 = 1$ and $\phi_i$ is supported in $|\xi_i^*| \gtrsim 1$. We define $s_{\lambda i}(\xi) = s_{\leq \lambda}(\xi) \beta( \| \xi' \| / \lambda) \phi_i(\xi^*)$ with $\beta$ like in \eqref{eq:LPBeta}. A variant of the analysis of Section \ref{section:StrichartzEstimatesIsotropic} yields:
\begin{lemma}
For $i=1,2,3$, there are invertible matrices $m^{i}(\xi)$ such that
\begin{equation*}
p(x,\xi) s_{\lambda i}(\xi) = m^{(i)}(\xi) d(x,\xi) (m^{(i)})^{-1}(\xi) s_{\lambda i}(\xi)
\end{equation*}
with $d(x,\xi) = \text{diag}(-\xi_0^2, - \xi_0^2 + \varepsilon(x) \|\xi' \|^2, - \xi_0^2 + \varepsilon(x) \|\xi' \|^2)$.
\end{lemma}
\begin{proof}
We let as first eigenvector (independently of $i$) $v_1 = \xi^*$.

\vspace*{0.3cm}
$\bullet |\xi_1^*| \gtrsim 1$: We let as second and third eigenvector perpendicular to $\xi^*$:
\begin{equation*}
v_2 = 
\begin{pmatrix}
\frac{\xi_2}{\xi_{12}} \\ - \frac{\xi_1}{\xi_{12}} \\ 0
\end{pmatrix}
, \quad v_3 = 
\begin{pmatrix}
\frac{\xi_3}{\xi_{13}} \\ 0 \\ -\frac{\xi_1}{\xi_{13}}
\end{pmatrix}
.
\end{equation*}
We set as conjugation matrix
\begin{equation*}
m^{(1)}(\xi) = 
\begin{pmatrix}
\xi_1^* &  \frac{\xi_2}{\xi_{12}} & \frac{\xi_3}{\xi_{13}} \\
\xi_2^* &  - \frac{\xi_1}{\xi_{12}} &0 \\
\xi_3^* & 0& - \frac{\xi_1}{\xi_{13}} 
\end{pmatrix}
\end{equation*}
and compute $\det m^{(1)}(\xi) = \frac{\xi_1 \| \xi' \|}{\xi_{12} \xi_{13}}$.

\vspace*{0.3cm}
\noindent $\bullet |\xi_2^*| \gtrsim 1$: We let as second and third eigenvector
\begin{equation*}
v_2 = 
\begin{pmatrix}
\frac{\xi_2}{\xi_{12}} \\ - \frac{\xi_1}{\xi_{12}} \\ 0
\end{pmatrix}, \quad
v_3 = 
\begin{pmatrix}
0 \\ \frac{\xi_3}{\xi_{23}} \\ - \frac{\xi_2}{\xi_{23}}
\end{pmatrix}
.
\end{equation*}
As second conjugation matrix we set
\begin{equation*}
m^{(2)}(\xi) = 
\begin{pmatrix}
\xi_1^* & \frac{\xi_2}{\xi_{12}} & 0 \\
\xi_2^* & - \frac{\xi_1}{\xi_{12}} & \frac{\xi_3}{\xi_{23}} \\
\xi_3^* & 0 & -\frac{\xi_2}{\xi_{23}}
\end{pmatrix}
\end{equation*}
and have $\det m^{(2)}(\xi) = \frac{\xi_2 \| \xi' \|}{\xi_{12} \xi_{23}}$.
\vspace*{0.3cm}

\noindent $\bullet |\xi_3^*| \gtrsim 1$: We choose the second and third eigenvector as
\begin{equation*}
v_2 = 
\begin{pmatrix}
\frac{\xi_3}{\xi_{13}} \\ 0 \\ - \frac{\xi_1}{\xi_{13}}
\end{pmatrix}
, \quad v_3= 
\begin{pmatrix}
0 \\ \frac{\xi_3}{\xi_{23}} \\ - \frac{\xi_2}{\xi_{23}}
\end{pmatrix}
.
\end{equation*}
We set
\begin{equation*}
m^{(3)}(\xi) = 
\begin{pmatrix}
\xi_1^* & \frac{\xi_3}{\xi_{13}} & 0 \\
\xi_2^* & 0 & \frac{\xi_3}{\xi_{23}} \\
\xi_3^* & - \frac{\xi_1}{\xi_{13}} & - \frac{\xi_2}{\xi_{23}}
\end{pmatrix}
\end{equation*}
and $\det m^{(3)}(\xi) = \frac{\xi_3 \| \xi' \|}{\xi_{13} \xi_{23}}$.
\end{proof}
We remark that the entries of $m^{(i)}$ are $L^p$-bounded Fourier multipliers because these are Riesz transforms in two or three variables. By Cramer's rule, so are the entries of $(m^{(i)})^{-1}$ because the determinant is an $L^p$-bounded multiplier within the support of $s_{\lambda i}$. The latter is a straight-forward consequence of the H\"ormander--Mikhlin theorem. For future reference, $(m^{(i)})^{-1}$ take the form
\begin{equation}
\label{eq:InverseConjugationMatricesSimplifiedKerr}
(m^{(i)})^{-1} (\xi) = 
\begin{pmatrix}
\xi_1^* & \xi_2^* & \xi_3^* \\
w_{21}^{(i)} & w_{22}^{(i)} & w_{23}^{(i)} \\
w_{31}^{(i)} & w_{32}^{(i)} & w_{33}^{(i)}
\end{pmatrix}
.
\end{equation}

 Since $m^{(i)}(\xi) s_{\lambda i}(\xi)$ and $(m^{(i)})^{-1}(\xi) s_{\lambda i}(\xi) \in S^0_{1,0}$\footnote{Note that these are just Fourier multiplier.}, we can quantize
\begin{equation*}
\mathcal{M}_\lambda^{(i)} = OP(m^{(i)}(\xi) \chi_{\lambda i}(\xi)), \quad \mathcal{N}_\lambda^{(i)} = OP((m^{(i)})^{-1}(\xi) \chi_{\lambda i}(\xi)),
\end{equation*}
and
\begin{equation*}
\mathcal{D}_\lambda = \text{diag}(\partial_t^2 , \partial_t^2 - \nabla_{x'} \cdot (\varepsilon(x) \nabla_{x'}), \partial_t^2 - \nabla_{x'} \cdot (\varepsilon(x) \nabla_{x'})).
\end{equation*}
Note that the dependence on $\lambda$ comes for $\mathcal{D}_\lambda$ from the frequency truncation of $\varepsilon$, which is suppressed in notation. We have the following proposition on diagonalization:

\begin{proposition}
\label{prop:SimplifiedKerr}
For $i \in \{1,2,3\}$ and $\lambda \in 2^{\N_0}$, $\lambda \gg 1$, we have the following decomposition:
\begin{equation}
P_\lambda S_{\lambda i} = \mathcal{M}_\lambda^{(i)} \mathcal{D}_\lambda \mathcal{N}_\lambda^{(i)} S_{\lambda i} + E_{\lambda i}
\end{equation}
with $\| E_{\lambda i} \|_{L^2 \to L^2} \lesssim \lambda$.
\end{proposition}

\begin{proof}
First we observe that we can write $\mathcal{D}_\lambda \tilde{S}_{\lambda i} = OP(d(x,\xi)) \tilde{S}_{\lambda i} + E_D \tilde{S}_{\lambda i}$ with $\| E_D \tilde{S}_{\lambda i} \|_{L^2 \to L^2} \lesssim \lambda$. By symbol composition and $\| \mathcal{M}^{(i)}_\lambda E^i_D \mathcal{N}^{(i)}_\lambda \|_{L^2 \to L^2} \lesssim \lambda$, we find
\begin{equation*}
\mathcal{M}^{(i)}_\lambda \mathcal{D} \mathcal{N}_\lambda^{(i)} S_{\lambda i} = P_\lambda S_{\lambda i} + R \tilde{S}_{\lambda i}.
\end{equation*}
We have to show that $\| R \tilde{S}_{\lambda i} \|_{L^2 \to L^2} \lesssim \lambda$. For this purpose we use symbol composition and the asymptotic expansion of
\begin{equation*}
\mathcal{M}^{(i)}_\lambda \mathcal{D} \mathcal{N}^{(i)}_\lambda S_{\lambda i} = P S_{\lambda i} + \lambda^2 O( \partial \varepsilon \frac{\partial a}{\partial \xi} \chi_{\lambda i}(\xi))
\end{equation*}
with $a$ denoting a component of $m^{(i)}$. Similar to Proposition \ref{prop:Diagonalization}, we verify that the leading order error term satisfies the bound
\begin{equation*}
\| O( \partial \varepsilon \frac{\partial a}{\partial \xi} \chi_{\lambda i}(\xi)) \|_{L^2 \to L^2} \lesssim \lambda^{-1}.
\end{equation*}
The reason is that the coefficients of $\varepsilon$ are still Lipschitz, so we have the bound for the truncated coefficients
\begin{equation*}
\| \partial_x^{\alpha} \varepsilon \|_{L^\infty} \lesssim \lambda^{\frac{(|\alpha|-1)_+}{2}}, \quad \alpha \in \N_0^4.
\end{equation*}
\end{proof}

To conclude the proof of Theorem \ref{thm:StrichartzEstimatesSimplifiedKerr} like in Section \ref{section:StrichartzEstimatesIsotropic}, we need the following estimates:
\begin{lemma}
Let $\mathcal{M}^{(i)}_\lambda$, $\mathcal{D}_\lambda$, and $\mathcal{N}^{(i)}_\lambda$ like above for $\lambda \in 2^{\N_0}$. The following estimates are true:
\begin{align}
\label{eq:SimplifiedKerrAuxI}
\lambda^{1-\rho} \| S_{\lambda i} u \|_{L_t^p L_{x'}^q} &\lesssim \lambda^{1-\rho} \| \tilde{S}_{\lambda i} \mathcal{N}_\lambda^{(i)} S_{\lambda i} u \|_{L_t^p L_{x'}^q} + \lambda \| S_{\lambda i} u \|_{L^2_x}, \\
\label{eq:SimplifiedKerrAuxII}
\lambda^{1-\rho} \| \tilde{S}_{\lambda i} \mathcal{N}^{(i)}_\lambda S_{\lambda i} u \|_{L_t^p L_{x'}^q} &\lesssim \lambda \| S_{\lambda i} u \|_{L_t^\infty L_{x'}^2} + \| \tilde{S}_{\lambda i} \mathcal{D}_\lambda \mathcal{N}_\lambda S_{\lambda i} u \|_{L^2_x} \\
&\quad + \lambda^{\frac{1}{p}} \| S'_\lambda \rho_e \|_{L_t^\infty L_{x'}^2}, \nonumber \\
\label{eq:SimplifiedKerrAuxIII}
\| \tilde{S}_{\lambda i} \mathcal{D}_\lambda \mathcal{N}_\lambda S_{\lambda i} u \|_{L^2_x} &\lesssim \| P_\lambda S_{\lambda i} u \|_{L^2_x} + \lambda \| S_{\lambda i} u \|_{L^2_x}.
\end{align}
\end{lemma}
\begin{proof}
For the proof of \eqref{eq:SimplifiedKerrAuxI} and \eqref{eq:SimplifiedKerrAuxIII} one can use the arguments of the proof of Lemma \ref{lem:AuxDiagonalizationIsotropic}. For the proof of \eqref{eq:SimplifiedKerrAuxII}, we use Sobolev embedding for the first component as several times above. For the second and third component, we apply \cite[Theorem~1.1]{Tataru2002}.
\end{proof}

We finish the proof of Theorem \ref{thm:StrichartzEstimatesSimplifiedKerr} by following the arguments of the proof of Theorem \ref{thm:IsotropicStrichartz} in Section \ref{section:StrichartzEstimatesIsotropic}.

\end{proof}
Now we turn to the proof of Theorem \ref{thm:LocalWellposednessSimplifiedKerr}: We prove local well-posedness in three steps via the strategy explained in detail in the survey by Ifrim--Tataru \cite{IfrimTataru2020} (see also \cite{SchippaSchnaubelt2021}):
\begin{itemize}
\item[1.] Energy estimates for solutions: For a suitably defined functional and a solution $u$ to  \eqref{eq:SimplifiedKerrComponent} we prove
\begin{equation}
\label{eq:EnergyEstimateSimplifiedKerr}
E^s[u](t) \lesssim_{\| u \|_{L_x^\infty}} e^{c(\| u \|_{L_x^\infty}) \int_0^t (1+ \| \nabla_x u(s) \|_{L_{x'}^\infty} ) ds} E^s[u](0).
\end{equation}
\item[2.] Lipschitz-continuous dependence for differences of solutions: For smooth solutions $u$ and $v$ to \eqref{eq:SimplifiedKerrComponent} we have the following estimate
\begin{equation}
\label{eq:LipschitzBound}
\| \nabla_x (u-v) \|_{L_t^\infty(0,T;L^2_{x'})} \leq c ( \| u \|_{L_T^\infty H^s_{x'}}, \| v \|_{L_T^\infty H^s_{x'}}, \| \nabla_x u \|_{L_T^2 L_{x'}^\infty}, \| \nabla_x v \|_{L_T^2 L_{x'}^\infty}) \| \nabla_x (u-v)(0) \|_{L^2_{x'}}.
\end{equation}
\item[3.] Continuous dependence via frequency envelopes.
\end{itemize}
\vspace*{0.3cm}

 We define the energy functional by
 \begin{equation*}
 E^s[u](t) = \| \langle D' \rangle^{s-1} u(t) \|_{L^2}^2 + \| \langle D' \rangle^{s-1} \partial_t u(t) \|_{L^2}^2 + \langle \langle D' \rangle^{s-1} \nabla \times u(t) , \varepsilon(u) \langle D' \rangle^{s-1} \nabla \times u(t) \rangle
 \end{equation*}
for $s \geq 1$. Observe by Helmholtz decomposition that $\| \nabla \times u(t) \|_{L^2}^2 = \| u(t) \|_{\dot{H}_{x'}^1}^2$ because we require $\nabla_{x'} \cdot u(t) = 0$. 
\begin{lemma}
Let $u$ be a smooth solution to \eqref{eq:SimplifiedKerrComponent}. The estimate \eqref{eq:EnergyEstimateSimplifiedKerr} holds true.
\end{lemma}
\begin{proof}
We compute
\begin{equation*}
\begin{split}
\frac{d}{dt} E^s[u] &= 2 \langle \langle D' \rangle^{s-1} \partial_t u, \langle D' \rangle^{s-1} u \rangle + 2 \langle \langle D' \rangle^{s-1} \partial_t^2 u , \langle D' \rangle^{s-1} \partial_t u \rangle \\
&\quad + 2 \langle \langle D' \rangle^{s-1} \nabla \times \partial_t u, \varepsilon(u) \langle D' \rangle^{s-1} \nabla \times u \rangle + \langle \langle D' \rangle^{s-1} \nabla \times u , (\partial_t \varepsilon(u)) \langle D' \rangle^{s-1} \nabla \times u \rangle \\
&= I + II + III + IV.
\end{split}
\end{equation*}
The first and last term are estimated by Cauchy-Schwarz and H\"older's inequality as
\begin{equation*}
\begin{split}
|\langle \langle D' \rangle^{s-1} \partial_t u, \langle D' \rangle^{s-1} u \rangle | &\lesssim E^s[u], \\
| \langle \langle D' \rangle^{s-1} \nabla \times u, (\partial_t \varepsilon) \langle D' \rangle^{s-1} \nabla \times u \rangle | &\lesssim \| \partial_t u \|_{L_{x'}^\infty} E^s[u].
\end{split}
\end{equation*}
We summarize the second and third component as
\begin{equation*}
(II + III)/2 = \langle \nabla \times (\varepsilon(u) \langle D' \rangle^{s-1} \nabla \times u) + \langle D' \rangle^{s-1} \nabla \times (\varepsilon(u) \nabla \times u, \langle D' \rangle^{s-1} \partial_t u \rangle.
\end{equation*}
Before we compute the commutator, we note that
\begin{equation*}
\nabla \times (\varepsilon(u) \nabla \times u) = O (\partial_{x'} \varepsilon) \nabla \times u - \varepsilon(u) \Delta u.
\end{equation*}
The first term is lower order because by the fractional Leibniz rule we find
\begin{equation}
\label{eq:FractionalLeibnizMoser}
\begin{split}
\| \langle D' \rangle^{s-1} (\partial_{x'} \varepsilon) (\nabla \times u) \|_{L_{x'}^2} &\lesssim \| \partial \varepsilon \|_{L_{x'}^\infty} \| \langle D' \rangle^{s-1} \nabla \times u \|_{L_{x'}^2} + \| \nabla \times u \|_{L_{x'}^\infty} \| \langle D' \rangle^{s-1} \partial \varepsilon \|_{L^2_{x'}} \\
&\lesssim \| \nabla_x u \|_{L_{x'}^\infty} \| \langle D' \rangle^{s-1} \partial u \|_{L^2_{x'}}.
\end{split}
\end{equation}
The ultimate inequality is a consequence of Moser estimates (assuming a priori bounds on $\| u \|_{L_{x'}^\infty}$). This reduces us to estimate
\begin{equation*}
[\varepsilon(u), \langle D' \rangle^{s-1} ] \Delta u_j = [\varepsilon(u), \partial_i \langle D' \rangle^{s-1} ] \partial_i u_j -  [ \partial_i \varepsilon(u), \langle D' \rangle^{s-1} ]\partial_i u_j.
\end{equation*}
The second term can be argued to be lower order like in \eqref{eq:FractionalLeibnizMoser}. For the first term by the Kato-Ponce commutator estimate and another application of Moser's inequality, we find
\begin{equation*}
\| [\varepsilon(u), \partial_i \langle D' \rangle^{s-1} ] \partial_i u_j \|_{L^2_{x'}} \lesssim \| \nabla u \|_{L_{x'}^\infty} \| u \|_{H^s_{x'}} \lesssim \| \nabla u \|_{L_{x'}^\infty} E^s[u].
\end{equation*}
The ultimate estimate follows from $E^s[u] \approx_{\| u \|_{L_{x'}^\infty}} \| u \|_{H^s_{x'}}$ for divergence-free functions. Applying Gr\"onwall's inequality yields
\begin{equation*}
E^s[u](t) \lesssim e^{C \int_0^t (1+ \| \nabla_x u \|_{L_{x'}^\infty}) ds} E^s[u](0).
\end{equation*}
\end{proof}
Combining the energy estimate with Strichartz estimates, we can show a priori estimates for $s>\frac{13}{6}$:
\begin{lemma}
Let $s>\frac{13}{6}$ and $u$ be a smooth solution to \eqref{eq:SimplifiedKerrComponent}. Then, there is a lower semicontinuous $T=T(\| u_0 \|_{H^s_{x'}})$ such that the following estimate holds:
\begin{equation*}
\sup_{t \in [0,T]} \| u(t) \|_{H^s_{x'}} \lesssim \| u(0) \|_{H^s_{x'}}.
\end{equation*}
\end{lemma}
\begin{proof}
Since $u$ is divergence free and for $s>\frac{13}{6}$ we have $\| u(t) \|_{L^\infty_{x'}} \lesssim \| u(t) \|_{H^s_{x'}}$, it suffices to show an a priori estimate for the energy functional $E^s[u]$. To this end, we control $\| \nabla_x u \|_{L_t^2 L_{x'}^\infty}$ by Strichartz estimates.

We define the auxiliary function $v = \langle D' \rangle^{s-1} u$ and apply Strichartz estimates to find
\begin{equation}
\label{eq:EndpointStrichartz}
\| \langle D' \rangle^{-\rho - \frac{1}{6 p} - \varepsilon} \nabla_x v \|_{L_t^2 L_{x'}^\infty} \lesssim \| \nabla_x v \|_{L_t^\infty L_{x'}^2} + \| P(x,u,D) v \|_{L_t^1 L_{x'}^2}.
\end{equation}
Since $P(x,u,D) v = [P(x,u,D), \langle D' \rangle^{s-1} ] u$, \eqref{eq:EndpointStrichartz} yields
\begin{equation*}
\| \nabla_x u \|_{L_T^2 L_{x'}^\infty} \lesssim_T \| \langle D' \rangle^{s-1} \nabla_x u \|_{L_t^\infty L_{x'}^2} + \| [P(x,u,D), \langle D' \rangle^{s-1}] u \|_{L^2_{x'}}.
\end{equation*}
By the commutator estimate from above, we find
\begin{equation}
\label{eq:StrichartzEstimateAPriori}
\| \nabla_x u \|_{L_T^2 L_{x'}^\infty} \lesssim_T (1 + T^{\frac{1}{2}} \| \nabla u \|_{L_T^2 L_{x'}^\infty} ) \| \langle D' \rangle^{s-1} \nabla_x u \|_{L_t^\infty L_{x'}^2}.
\end{equation}
Moreover, the energy estimate gives
\begin{equation}
\label{eq:EnergyEstimateAPriori}
E^s[u](t) \lesssim_{\| u \|_{L_x^\infty}} e^{C (t+ \int_0^t \| \nabla_x u(s) \|_{L_{x'}^\infty} ds)} E^s[u](0).
\end{equation}
\eqref{eq:StrichartzEstimateAPriori} and \eqref{eq:EnergyEstimateAPriori} can be bootstrapped for $T=T(\| u_0 \|_{H^s})$ and $s>\frac{13}{6}$. The proof is complete.
\end{proof}

We turn to estimates for differences of solutions in $L^2_x$. Here we follow the argument of \cite[Lemma~4.2]{Tataru2000}.
\begin{lemma}
Let $u,v$ be smooth solutions to \eqref{eq:SimplifiedKerrComponent} on $[0,T]$. Then the following estimate holds:
\begin{equation}
\label{eq:DifferenceEstimate}
\| \nabla_x (u-v) \|_{L_t^\infty(0,T;L^2_{x'})} \leq c ( \| u \|_{L_T^\infty H^s_{x'}}, \| v \|_{L_T^\infty H^s_{x'}}, \| \nabla_x u \|_{L_T^2 L_{x'}^\infty}, \| \nabla_x v \|_{L_T^2 L_{x'}^\infty}) \| \nabla_x (u-v)(0) \|_{L^2_{x'}}.
\end{equation}
\end{lemma}
\begin{proof}
The difference $w = u-v$ solves the equation
\begin{equation*}
\begin{split}
P(x,u,D) w &= \partial_t^2 (u-v) + \nabla \times [ \varepsilon(u) \nabla \times u - \varepsilon(u) \nabla \times v] \\
&= - [\varepsilon(v) - \varepsilon(u) ] \Delta v + O(\partial(\varepsilon(v) - \varepsilon(u))) \nabla \times v \\
&= A_0 w + A_1 \nabla w.
\end{split}
\end{equation*}
with 
\begin{equation*}
A_0 = B_1(u,v) D' \nabla_{x'} v + B_2(u,v) (\nabla_{x'} u, \nabla_{x'} v)^2, \quad A_1 = B_3(u,v) (\nabla_{x'} u , \nabla_{x'} v).
\end{equation*}
The above notation means that $A_1$ is linear in $(\nabla_{x'}u , \nabla_{x'} v)$ and $A_0$ is linear in $D'_x \nabla_{x'} v$ and quadratic in $ \nabla_{x'} u$ and $\nabla_{x'} v$. To prove \eqref{eq:DifferenceEstimate}, we shall carry out a fixed point argument for the Strichartz norm $\| w \|_S = \| \nabla_x w \|_{L_T^\infty L_{x'}^2} + \| \langle D' \rangle^{-s} \nabla_x w \|_{L_T^2 L_{x'}^\infty}$ for some $s>\frac{13}{6}$. To this end, we shall prove that 
\begin{equation*}
P(x,u,D) w = f \in L_T^2 L_{x'}^2.
\end{equation*}

 By Strichartz estimates and energy estimates as argued above, we have $\| \nabla_{x'} u \|_{L_T^2 L_{x'}^\infty} + \| \nabla_{x'} v \|_{L_T^2 L_{x'}^\infty} < \infty$. Therefore, $A_1 \in L_T^2 L_{x'}^\infty$.
For $A_2$ we use interpolation to bound $\langle D' \rangle \nabla_x v $. By Strichartz and energy estimates, we obtain $\nabla_x v \in L_T^2 L_{x'}^\infty$ and $\langle D' \rangle^{s-1} \nabla_x v \in L_T^\infty L_{x'}^2$. By interpolation, we find $\langle D' \rangle \nabla_x v \in L_{T}^{p_1} L_{x'}^{q_1}$ with $p_1$, $q_1$ chosen such that
\begin{equation}
\label{eq:CollinearI}
\begin{array}{ccc}
( 0 & \frac{1}{2} & 0 ) \\
(1 & \frac{1}{p_1} & \frac{1}{q_1} )\\ 
( s-1 & 0 & \frac{1}{2} )
\end{array}
\end{equation}
are collinear.

Secondly, we estimate $(\nabla_{x'} u, \nabla_{x'} v)^2 \in L_{T}^{p_1} L_{x'}^{q_1}$. Indeed, in the borderline case $s = \frac{13}{6}$ we obtain $q_1 = \frac{7}{3}$, $p_1 = 14$. This gives by H\"older and Sobolev embedding
\begin{equation*}
\begin{split}
\| (\nabla_{x'} u)^2 \|_{L_T^{14} L_{x'}^{\frac{7}{3}}} &\lesssim \| \nabla_{x'} u \|_{L_T^\infty L_{x'}^9} \| \nabla_{x'} u \|_{L_T^{14} L_{x'}^{\frac{63}{20}}} \\
&\lesssim \| \langle D' \rangle^{s-1} \nabla_{x'} u \|_{L_T^\infty L_{x'}^2} \| \langle D' \rangle^{s-1} u \|_{L_T^\infty L_{x'}^2}.
\end{split}
\end{equation*}
Other quadratic expressions $(\nabla_{x'} u, \nabla_{x'} v)$ and $(\nabla_{x'} v)^2$, which shows that $A_1 \in L_T^{p_1} L_{x'}^{q_1}$. Strichartz estimates for $f \in L_T^2 L_{x'}^2$ give
\begin{equation*}
\| w \|_S \lesssim \| (w(0),\dot{w}(0)) \|_{\dot{H}^1 \times L^2} + T^{\frac{1}{2}} \| f \|_{L_T^2 L_{x'}^2}.
\end{equation*}
In particular, we can estimate $\| w \|_{L_T^{p_2} L_{x'}^{q_2}}$ for collinear
\begin{equation}
\label{eq:CollinearII}
\begin{array}{ccc}
( 0  & 0 				& \frac{1}{2} ) \\
( -1 & \frac{1}{p_2}	& \frac{1}{q_2} ) \\
(1-s & \frac{1}{2} &	0 )
\end{array}
\end{equation}
\eqref{eq:CollinearI} and \eqref{eq:CollinearII} give
\begin{equation*}
\frac{1}{p_1} + \frac{1}{p_2} = \frac{1}{q_1} + \frac{1}{q_2} = \frac{1}{2}.
\end{equation*}
Therefore, for $\| w \|_S < \infty$, we have $f \in L_x^2$ and the claim follows from the estimate
\begin{equation*}
\| w \|_S \lesssim \| (w(0),\dot{w}(0)) \|_{H^1 \times L^2} + T^{\frac{1}{2}} \| w \|_S C(\| u \|_S, \| v \|_S, \| w \|_S).
\end{equation*}
\end{proof}
We conclude the argument by frequency envelopes:
\begin{definition}
$(c_N)_{N \in 2^{\N_0}} \in \ell^2$ is a frequency envelope for functions $(u_0,u_1) \in H^s_x \times H^{s-1}_x$ if we have the following two properties:
\begin{itemize}
\item[a)] Energy bound:
\begin{equation*}
\| S'_N (u_0,u_1) \|_{H^s \times H^{s-1}} \leq c_N,
\end{equation*}
\item[b)] Slowly varying property:
\begin{equation*}
\frac{c_N}{c_J} \lesssim \big[ \frac{N}{J} \big]^\delta.
\end{equation*}
We use the notation $\big[ \frac{N}{J} \big] = \min( \frac{N}{J}, \frac{J}{N} )$.
\end{itemize}
\end{definition}
$S'_N$ denote the spatial Littlewood-Paley projections. Envelopes are sharp if
\begin{equation*}
\| u \|^2_{H^s \times H^{s-1}} \approx \sum_N c_N^2.
\end{equation*}
To construct an envelope for $(u_0,u_1) \in H^s \times H^{s-1}$, we let
\begin{equation*}
\tilde{c}_N = \| S'_N (u_0,u_1) \|_{H^s} \text{ and } c_N = \sup_J \big[ \frac{N}{J} \big]^\delta c_J .
\end{equation*}
We use the following regularization: Let $(u_0,u_1) \in H^s \times H^{s-1}$ with size $L$ and let $(c_N)$ be a sharp frequency envelope. For $u_0$ we consider $(u_0,u_1)^M = S'_{\leq M} (u_0,u_1)$ frequency truncations as regularization. We note the following properties:
\begin{itemize}
\item Uniform bounds:
\begin{equation*}
\| S'_N (u_0^M,u_1^M) \|_{H^s \times H^{s-1}} \lesssim c_N,
\end{equation*}
\item High frequency bounds:
\begin{equation*}
\| (u_0^M,u_1^M) \|_{H^{s+j} \times H^{s+j-1}} \lesssim M^j c_M \quad (j \geq 0).
\end{equation*}
\item Difference bounds:
\begin{equation*}
\| (u_0,u_1)^{2M} - (u_0,u_1)^M \|_{H^1 \times L^2} \lesssim M^{-s} c_M.
\end{equation*}
\item Limit:
\begin{equation*}
(u_0,u_1) = \lim_{M \to \infty} (u_0^M,u_1^M) \text{ in } H^s \times H^{s-1}.
\end{equation*}
\end{itemize}
We obtain for the regularized initial data a family of smooth solutions. The existence depends only on $L = \| (u_0,u_1) \|_{H^s \times H^{s-1}}$. We have the following:
\begin{itemize}
\item[i)] A priori estimates at high regularity:
\begin{equation*}
\| u^M \|_{C(0,T;H^{s+j})} \lesssim M^j c_M, \quad j \geq 0,
\end{equation*}
\item[ii)] Difference bounds:
\begin{equation*}
\| u^{2M} - u^M \|_{C(0,T;H^1 \times L^2)} \lesssim M^{-s} c_M.
\end{equation*}
\end{itemize}
From the difference bounds and a telescoping sum argument, we have the convergence of $u^M$ as $M \to \infty$ in $C_T L^2$. Writing
\begin{equation*}
u-u^M = \sum_{K=M}^\infty u^{2K} - u^K
\end{equation*}
we can argue by estimates at higher regularity and difference bounds that $u^{2K} - u^K$ is essentially concentrated at frequencies $K$. This yields the estimate
\begin{equation*}
\| u - u^M \|_{C(0,T;H^s)} \lesssim c_{\geq M}
\end{equation*}
and convergence of $u^M$ in $C_T H^s$. A variant of the argument also gives continuity of the data-to-solution mapping. The proof of Theorem \ref{thm:LocalWellposednessSimplifiedKerr} is complete.
$ \hfill \Box$

\subsection{Partially anisotropic permittivity}

In this section we improve the local well-posedness for quasilinear Maxwell equations in the case of partially anisotropic permittivity $\varepsilon^{-1} = (\psi(|\D_1|^2),1,1)$. To prove energy estimates, we have to rewrite the Maxwell system
\begin{equation}
\label{eq:PartiallyAnisotropicMaxwellSection}
\left\{ \begin{array}{cl}
\partial_t \D &= \nabla \times \mathcal{H}, \\
\partial_t \mathcal{H} &= - \nabla \times (\varepsilon^{-1}(\D) \D)
\end{array} \right.
\end{equation}
into non-divergence form. We compute
\begin{equation*}
\nabla \times (\varepsilon^{-1}(\D) \D) =
\begin{pmatrix}
\partial_2 \D_3 - \partial_3 \D_2 \\
\big( \psi(|\D_1|^2) + 2 \psi'(|\D_1|^2) \D_1^2 \big) \partial_3 \D_1 - \partial_1 \D_3 \\
\partial_1 \D_2 - \big( \psi(|\D_1|^2) + 2 \psi'(|\D_1|^2) \D_1^2 \big) \partial_2 \D_1
\end{pmatrix}
.
\end{equation*}
This suggests to work with the modified permittivity 
\begin{equation}
\label{eq:ModifiedPermittivity}
\tilde{\varepsilon}^{-1}(\D) = (\psi(|\D_1|^2) + 2 \psi'(|\D_1|^2) \D_1^2, 1, 1),
\end{equation}
for which we prove Strichartz estimates in divergence form. It turns out that these yield suitable Strichartz estimates for the equation in non-divergence form.

We shall prove Theorem \ref{thm:PartiallyAnisotropicWellposedness} following the same steps like above.
We begin with a priori estimates for solutions for $s>9/4$: We consider the energy functional
\begin{equation}
\label{eq:EnergyFunctionalPartiallyAnisotropic}
E^s[u](t) = \langle \langle D' \rangle^s u(t), C(u) \langle D' \rangle^s u(t) \rangle \approx_{\| u \|_{L_x^\infty}} \| u(t) \|_{H^s_{x'}},
\end{equation}
for which we want to prove the estimate
\begin{equation}
\label{eq:EnergyEstimatePartiallyAnisotropic}
E^s[u](t) \leq C(\| u \|_{L^\infty_x}) e^{c(\| u \|_{L^\infty_x}) \int_0^t \| \nabla_{x'} u(s) \|_{L^\infty_{x'}} ds} E^s[u](0).
\end{equation}

To cancel the top-order terms, we define symmetric $C(u)$ such that we find the estimate
\begin{equation*}
\frac{d}{dt} E^s[u](t) \leq C(\| u \|_{L^\infty_x}) \| \nabla_x u \|_{L^\infty_{x'}} E^s[u](t)
\end{equation*}
to hold. To this end, we rewrite \eqref{eq:PartiallyAnisotropicMaxwellSection} as $\partial_t u = \mathcal{A}_j(u) \partial_j u$ and require
\begin{equation}
\label{eq:AnsatzPartiallyAnisotropicSymmetrization}
C(u) \mathcal{A}^j(u) = \mathcal{A}^j(u)^* C(u).
\end{equation}
The matrices $\mathcal{A}^j(u)$ take the form
\begin{equation*}
\mathcal{A}^j(u) = 
\begin{pmatrix}
0 & A^j_1(u) \\
A^j_2(u) & 0
\end{pmatrix}
\end{equation*}
and we have $(A_1^j)_{mn} = - \varepsilon_{jm n}$ with $\varepsilon$ denoting the Levi--Civita symbol. For $A_2^j$ we find $(A_2^1)_{mn} = \varepsilon_{1mn}$ and 
\begin{equation*}
A_2^2 = 
\begin{pmatrix}
0 & 0 & -1 \\
0 & 0 & 0 \\
\psi(|\D_1|^2) + 2 \psi'(|\D_1|^2) \D_1^2 & 0 & 0
\end{pmatrix}
, \quad A_2^3 = 
\begin{pmatrix}
0 & 1 & 0 \\
- \psi(|\D_1|^2) - 2 \psi'(|\D_1|^2) \D_1^2 & 0 & 0 \\
0 & 0 & 0
\end{pmatrix}
.
\end{equation*}
With the ansatz
\begin{equation*}
C(u) = 
\begin{pmatrix}
C_1(u) & 0 \\
0 & 1_{3 \times 3}
\end{pmatrix}
\end{equation*}
\eqref{eq:AnsatzPartiallyAnisotropicSymmetrization} becomes $A_2^j = (A_1^j)^t C_1(u)$. A straight-forward computation yields
\begin{equation*}
C_1(u) = 
\begin{pmatrix}
\psi(|\D_1|^2) + 2 \psi'(|\D_1|^2) \D_1^2 & 0 & 0\\
0 & 1 & 0 \\
0 & 0 & 1
\end{pmatrix}
.
\end{equation*}
We are ready for the proof of the following proposition:
\begin{proposition}
\label{prop:APrioriKerr}
Let $s \geq 0$ and $u=(\D,\mathcal{H})$ be a smooth solution to \eqref{eq:PartiallyAnisotropicMaxwellSection}. Then, \eqref{eq:EnergyEstimatePartiallyAnisotropic} holds true. For $s>9/4$, there is a time $T=T(\|u_0\|_{H^s})$, which is lower semicontinuous such that
\begin{equation}
\label{eq:APrioriPartiallyAnisotropic}
\sup_{t \in [0,T]} \| u(t) \|_{H^s_{x'}} \lesssim \| u_0 \|_{H^s_{x'}}.
\end{equation}
\end{proposition}
\begin{proof}
We compute
\begin{equation*}
\begin{split}
\frac{d}{dt} E^s[u](t) &= \langle \langle D' \rangle^s \sum_{j=1}^3 \mathcal{A}^j(u) \partial_j u, C(u) \langle D' \rangle^s u \rangle + \langle \langle D' \rangle^s u, ( \frac{d}{dt} C(u)) \langle D' \rangle^s u \rangle \\
&\quad + \langle \langle D' \rangle^s u, C(u) \langle D' \rangle^s \sum_{j=1}^3 \mathcal{A}^j(u) \partial_j u \rangle = I + II + III.
\end{split}
\end{equation*}
Clearly, by using the equation and H\"older's inequality, we find
\begin{equation*}
II \lesssim_{\| u \|_{L^\infty_x}} \| \nabla_{x'} u \|_{L^\infty} \| u \|^2_{H^s_{x'}}.
\end{equation*}
Via integration by parts, the Kato-Ponce commutator estimate, and Moser estimates, we find
\begin{equation*}
\begin{split}
|I + III| &\lesssim_{\| u \|_{L^\infty_x}} \| \nabla_{x'} u \|_{L^\infty_{x'}} \| u \|^2_{H^s_{x'}} \\
&\; + \big| \sum_{j=1}^3 (\langle \langle D' \rangle^s \partial_j u, \mathcal{A}^j(u) C(u) \langle D' \rangle^s u \rangle - \langle \langle D' \rangle^s \partial_j u, C(u) \mathcal{A}^j(u) \langle D' \rangle^s u \rangle \big|.
\end{split}
\end{equation*}
The term in the second line vanishes by \eqref{eq:AnsatzPartiallyAnisotropicSymmetrization}.
Taking the estimates together, we have
\begin{equation*}
\frac{d}{dt} E^s[u](t) \lesssim_{\| u \|_{L^\infty_x}} \| \nabla_{x'} u (t) \|_{L^\infty_{x'}} E^s[u](t)
\end{equation*}
and \eqref{eq:EnergyEstimatePartiallyAnisotropic} follows from Gr\o nwall's argument.

To prove a priori estimates for $s>9/4$, we use Strichartz estimates and a continuity argument. We require that $\| \nabla_{x'} u \|_{L^2(0,T_0;L^\infty_{x'})} \leq K$ for fixed $K > 0$ and a maximally defined $T_0 > 0$. Note that this gives
\begin{equation*}
\| \nabla_x \varepsilon \|_{L^2(0,T_0;L^\infty_{x'})} \lesssim_A K \text{ and } \| \nabla_x \varepsilon \|_{L^1(0,T_0;L^\infty_{x'})} \lesssim_A T_0^{\frac{1}{2}} K.
\end{equation*}
Hence, we have uniform constants in the energy and Strichartz estimates
\begin{equation*}
\| \langle D' \rangle^{-\alpha} w \|_{L^p(0,T;L^\infty)} \lesssim \| w \|_{L_t^\infty L_{x'}^2} + \| \tilde{P}(x,D) w \|_{L_t^1 L_{x'}^2}
\end{equation*}
for $\alpha > \rho + \frac{1}{3p}$ from Theorem \ref{thm:PartiallyAnisotropicStrichartz} in the charge-free case, where $\tilde{P}$ is the time-dependent Maxwell operator with permittivity $\tilde{\varepsilon}^{-1}$ as defined in \eqref{eq:ModifiedPermittivity}. This can be recast as (using $\| \partial \tilde{\varepsilon} \|_{L_T^1 L_{x'}^\infty} < \infty $)
\begin{equation}
\label{eq:StrichartzNonDivergenceForm}
\| \langle D' \rangle^{-\alpha} w \|_{L^p(0,T;L^q)} \lesssim \| w \|_{L_t^\infty L_{x'}^2} + \| Q(x,D) w \|_{L_t^1 L_{x'}^2}
\end{equation}
for $Q(x,D) = \partial_t 1_{6 \times 6} - \mathcal{A}^j(u) \partial_j$ denoting the time-dependent Maxwell operator in non-divergence form.

By applying this estimate to $w = \langle D' \rangle^{\alpha+1} u$ and the Kato--Ponce commutator estimate (for which it is necessary to change to the non-divergence form), we obtain the estimate
\begin{equation*}
\| \nabla_{x'} u \|_{L^4(0,T;L^\infty_{x'})} \lesssim_{\| u \|_{L^\infty_{x}}} (1+ T^{\frac{1}{2}} \| \nabla_{x'} u \|_{L^2_T L_{x'}^\infty}) \| u \|_{L_T^\infty H^s_{x'}}.
\end{equation*}
Together with \eqref{eq:EnergyEstimatePartiallyAnisotropic}, this can be bootstrapped to prove the claim.
\end{proof}
By similar arguments to \cite{SchippaSchnaubelt2021}, we prove the following $L^2$-bound for differences:
\begin{proposition}
Let $u^1$ and $u^2$ be two smooth solutions to \eqref{eq:PartiallyAnisotropicMaxwellSection}, and set $v = u^1 - u^2$. Then, the estimate
\begin{equation*}
\| v(t) \|_{2} \lesssim_A e^{c(A) \int_0^t B(s) ds} \| v(0) \|_{L^2}
\end{equation*}
holds true with $A= \| u^1 \|_{L^\infty_x} + \| u^2 \|_{L^\infty_x}$ and $B(t) = \| \nabla_{x'} u^1(t) \|_{L^\infty_{x'}} + \| \nabla_{x'} u^2(t) \|_{L^\infty_{x'}}$. For $s>9/4$, there is $T=T(\| u^i(0) \|_{H^s})$ such that $T$ is lower semicontinuous and
\begin{equation*}
\sup_{t \in [0,T]} \| v(t) \|_{L^2_{x'}} \lesssim_{\|u^i(0) \|_{H^s_{x'}}} \| v(0) \|_{L^2_{x'}}.
\end{equation*}
\end{proposition}
\begin{proof}
First we note
\begin{equation*}
\begin{split}
\frac{d}{dt} v(t) &= \sum_{j=1}^3 \mathcal{A}^j(u) \partial_j u^1 - \sum_{j=1}^3 \mathcal{A}^j(u^2) \partial_j u^2 \\
&= \sum_{j=1}^3 \mathcal{A}^j(u^1) \partial_j v + \sum_{j=1}^3 [\mathcal{A}^j(u^1) - \mathcal{A}^j(u^2)] \partial_j u^2 \\
&= \sum_{j=1}^3 \mathcal{A}^j(u^1) \partial_j v + \sum_{j=1}^3 \mathcal{B}^j(u^1,u^2) v \partial_j u^2.
\end{split}
\end{equation*}
Let $E^0[v](t) = \langle v(t), C(u^1) v(t) \rangle$ and compute
\begin{equation*}
\begin{split}
\frac{d}{dt} E^0[v](t) &= \langle \sum_{j=1}^3 \mathcal{A}^j(u^1) \partial_j v, C(u^1) v \rangle + \langle v, ( \frac{d}{dt} C(u^1)) v \rangle + \langle v, C(u^1) \sum_{j=1}^3 \mathcal{A}^j(u^1) \partial_j v \rangle \\
&\quad + \langle \sum_{j=1}^3 \mathcal{B}^j(u^1,u^2) v \partial_j u^2, C(u^1) v \rangle + \langle v, C(u^1) \sum_{j=1}^3 \mathcal{B}^j(u^1,u^2) v \partial_j u^2 \rangle \\
&= I + II + III + IV + V.
\end{split}
\end{equation*}
The main terms $I + III$ are estimated like in the proof of Proposition \ref{prop:APrioriKerr} via integration by parts and Moser estimates:
\begin{equation*}
|I + III | \lesssim_{\| u^1 \|_{L^\infty}} \| \nabla_{x'} u \|_{L^\infty_{x'}} \| v(t) \|_{L^2_{x'}}^2.
\end{equation*}
We find like above by H\"older's inequality
\begin{equation*}
|II| \lesssim_{\| u^1 \|_{L^\infty}} \| \nabla_{x'} u^1 \|_{L^\infty_{x'}} \| v(t) \|_{L^2_{x'}}^2,
\end{equation*}
and $IV$ and $V$ are directly estimated by H\"older's inequality:
\begin{equation*}
|IV| + |V| \lesssim_A \| \nabla_{x'} u^2 \|_{L^\infty_{x'}} \| v(t) \|^2_{L^2_{x'}}.
\end{equation*}
Taking the estimates together gives
\begin{equation*}
\frac{d}{dt} E^0[v](t) \lesssim_A B(t) E^0[v](t),
\end{equation*}
and the proof is concluded by Gr\o nwall's argument.
\end{proof}

The proof of Theorem \ref{thm:PartiallyAnisotropicWellposedness} is concluded with frequency envelopes. For the first order system, these are defined as follows:
\begin{definition}
$(c_N)_{N \in 2^{\N_0}} \in \ell^2$ is called a frequency enveolope for $u \in H^s_{x'}$ if it has the following properties:
\begin{itemize}
\item[a)] Energy bound:
\begin{equation*}
\| S'_N u \|_{H^s_{x'}} \leq c_N.
\end{equation*}
\item[b)] Slowly varying: There is $\delta > 0$ such that for all $N,J \in 2^{\N_0}$:
\begin{equation*}
\frac{c_N}{c_J} \lesssim \big[ \frac{N}{J} \big]^{-\delta}.
\end{equation*}
\end{itemize}
\end{definition}
The envelope is called sharp if $\| u \|^2_{H^s} \approx \sum_N c_N^2$. By regularizing the obvious choice $\tilde{c}_N = \| S'_N u \|_{H^s}$, one shows that envelopes always exist. With this definition the argument from \cite{SchippaSchnaubelt2021} can be followed verbatim up to the difference in regularity. This finishes the proof of Theorem \ref{thm:PartiallyAnisotropicWellposedness}.
$\hfill \Box$

\section*{Appendix: Quantization and decomposition of the conjugation matrices in the partially anisotropic case}

We give decompositions for the quantizations $\mathcal{M}$, $\mathcal{N}$ for the conjugation matrices $\tilde{m}(x,\xi)$ and $\tilde{m}(x,\xi)^{-1}$ defined in \eqref{eq:ConjugationMatrixPartiallyAnisotropic} and \eqref{eq:InverseConjugationMatrixPartiallyAnisotropic} up to acceptable error terms. Recall that we had localized frequencies $\{|\xi_0| \lesssim \| \xi ' \| \sim \lambda \}$ and $\{ |(\xi_2,\xi_3)| \gtrsim \lambda^{\beta} \}$ and truncated the coefficients $\varepsilon = \text{diag}(\varepsilon_1,\varepsilon_2,\varepsilon_2)$ to frequencies $\lambda^{\alpha}$ with $\beta \geq \alpha$. Therefore, up to a smoothing error, the frequency projection to $\{|\xi_0| \lesssim \| \xi ' \| \sim \lambda \}$ and $\{ |(\xi_2,\xi_3)| \gtrsim \lambda^{\beta} \}$ can be harmlessly included after every factor. This is implicit in the following like the error terms. The pseudo-differential operators, for which we have sharp $L^p_x$-bounds, will be separated with ``$\cdot$''.

Recall that the precise choice of $\beta$ and $\alpha$ depends on the regularity and structural assumptions (cf. Section \ref{section:PartiallyAnisotropicCase}). In the following we let 
\begin{equation*}
D_{ij} = Op((\xi_i^2 + \xi_j^2)^{\frac{1}{2}}), \quad D'= Op(\| \xi' \|), \quad D_\varepsilon = Op( (\xi_i \xi_j \varepsilon^{ij})^{\frac{1}{2}})
\end{equation*}
and denote $a = \varepsilon_{1}^{-1}$ and $b = \varepsilon_2^{-1}$.

We give the expressions for $\mathcal{M}$:
\small
\begin{align*}
&\mathcal{M}_{11} = 0, \quad \mathcal{M}_{12} = \frac{-i}{D_\varepsilon} \partial_1 (a^{-1} \cdot), \quad \mathcal{M}_{13} =0, \\
&\qquad \mathcal{M}_{14} = 0, \quad \mathcal{M}_{15} = \frac{1}{D_\varepsilon^{\frac{3}{2}}} D^{\frac{1}{2}} \cdot  D_{23}, \quad \mathcal{M}_{16} = - \frac{1}{D_\varepsilon^{\frac{3}{2}}} D^{\frac{1}{2}} \cdot  D_{23}, \\
&\mathcal{M}_{21} = 0, \quad \mathcal{M}_{22} = \frac{-i}{D_\varepsilon} \partial_2 (b^{-1} \cdot), \quad \mathcal{M}_{23} = \frac{i}{D^{\frac{1}{2}}} \cdot  D_\varepsilon^{\frac{1}{2}} \cdot \frac{\partial_3}{D_{23}} \cdot (\frac{1}{\sqrt{b}} \cdot), \\
&\qquad \mathcal{M}_{24} = - \frac{i}{D^{\frac{1}{2}}} \cdot D_\varepsilon^{\frac{1}{2}} \cdot  \frac{\partial_3}{D_{23}} \cdot  ( \frac{1}{\sqrt{b}} \cdot), \quad \mathcal{M}_{25} = \frac{1}{D_\varepsilon^{\frac{3}{2}}} \cdot D^{\frac{1}{2}} \partial_1 \cdot \frac{\partial_{2}}{D_{23}}, \quad \mathcal{M}_{26} = - \frac{1}{D_\varepsilon^{\frac{3}{2}}} \cdot D^{\frac{1}{2}} \partial_1 \cdot  \frac{\partial_{2}}{D_{23}}, \\
\end{align*}
\begin{align*}
&\mathcal{M}_{31} = 0, \quad \mathcal{M}_{32} = - \frac{i}{D_\varepsilon} \partial_3 \cdot  (b^{-1} \cdot), \quad \mathcal{M}_{33} = - \frac{i}{D^{\frac{1}{2}}} \cdot  D_\varepsilon^{\frac{1}{2}} \cdot  \frac{\partial_2}{D_{23}} \cdot  (\frac{1}{\sqrt{b}} \cdot), \\
&\qquad \mathcal{M}_{34} = \frac{i}{D^{\frac{1}{2}}} \cdot D_\varepsilon^{\frac{1}{2}} \cdot  \frac{\partial_2}{D_{23}} \cdot  (\frac{1}{\sqrt{b}} \cdot), \quad \mathcal{M}_{35} = \frac{1}{D_\varepsilon^{\frac{3}{2}}} \cdot  D^{\frac{1}{2}} \partial_1 \cdot \frac{\partial_{3}}{D_{23}}, \quad \mathcal{M}_{36} = - \frac{1}{D_\varepsilon^{\frac{3}{2}}} \cdot  D^{\frac{1}{2}} \partial_1 \cdot \frac{\partial_{3}}{D_{23}}, \\
&\mathcal{M}_{41} = - \frac{i}{D} \partial_1, \quad \mathcal{M}_{42} = 0, \quad \mathcal{M}_{43} = - D_\varepsilon^{\frac{1}{2}} \cdot \frac{1}{D^{\frac{3}{2}}} \cdot D_{23}, \\
&\qquad \mathcal{M}_{44} = - D_\varepsilon^{\frac{1}{2}} \cdot \frac{1}{D^{\frac{3}{2}}} \cdot D_{23}, \quad \mathcal{M}_{45} = \mathcal{M}_{46} = 0, \\
\end{align*}
\normalsize The remaining expressions are given by:
\small
\begin{align*}
&\mathcal{M}_{51} = - \frac{i}{D} \partial_2, \quad \mathcal{M}_{52} = 0, \quad \mathcal{M}_{53} = - D_\varepsilon^{\frac{1}{2}} \cdot \frac{\partial_1}{D^{\frac{3}{2}}} \cdot \frac{\partial_{2}}{ D_{23}}, \\
&\qquad \mathcal{M}_{54} = - D_\varepsilon^{\frac{1}{2}} \cdot \frac{\partial_1}{D^{\frac{3}{2}}} \cdot \frac{\partial_{2}}{ D_{23}}, \quad \mathcal{M}_{55} = \frac{i}{D_\varepsilon^{\frac{1}{2}}} \cdot D^{\frac{1}{2}} \cdot \frac{ \partial_3}{D_{23}}, \quad \mathcal{M}_{56} = \frac{i}{D_\varepsilon^{\frac{1}{2}}} \cdot D^{\frac{1}{2}} \cdot \frac{ \partial_3}{D_{23}}, \\
&\mathcal{M}_{61} = - \frac{i}{D} \partial_3, \quad \mathcal{M}_{62} = 0, \quad \mathcal{M}_{63} = - D_\varepsilon^{\frac{1}{2}} \cdot \frac{\partial_1}{D^{\frac{3}{2}}} \cdot \frac{\partial_{3}}{ D_{23}}, \\
&\qquad \mathcal{M}_{64} = - D_\varepsilon^{\frac{1}{2}} \cdot \frac{\partial_1}{D^{\frac{3}{2}}} \cdot \frac{\partial_{3}}{ D_{23}}, \quad \mathcal{M}_{65} = - \frac{i}{D_\varepsilon^{\frac{1}{2}}} \cdot D^{\frac{1}{2}} \cdot \frac{ \partial_2}{D_{23}}, \quad \mathcal{M}_{66} = - \frac{i}{D_\varepsilon^{\frac{1}{2}}} \cdot D^{\frac{1}{2}} \cdot \frac{\partial_2}{D_{23}}.
\end{align*}
%For the diagonal symbol $d(x,\xi)$ we consider
%\begin{equation*}
%\mathcal{D} = \text{diag} (\partial_t, \partial_t, \partial_t - i \sqrt{b(x)} D, \partial_t + i \sqrt{b(x)} D, \partial_t + i D_\varepsilon, \partial_t - i D_\varepsilon).
%\end{equation*}
\normalsize We associate operators to $\tilde{m}^{-1}$ as follows:
\small
\begin{align*}
&\mathcal{N}_{11} = \mathcal{N}_{12} = \mathcal{N}_{13} = 0, \\
&\quad \mathcal{N}_{14} = - i \partial_1 \frac{1}{D}, \quad \mathcal{N}_{15} = - i \partial_2 \frac{1}{D}, \quad \mathcal{N}_{16} = - i \partial_3 \frac{1}{D}, \\
&\mathcal{N}_{21} = - i ab \cdot \partial_1 \frac{1}{D_\varepsilon}, \quad \mathcal{N}_{22} = - iab \cdot \partial_2 \frac{1}{D_\varepsilon}, \quad \mathcal{N}_{23} = - i ab \cdot \partial_3 \frac{1}{D_\varepsilon}, \\
&\quad \mathcal{N}_{24} = \mathcal{N}_{25} = \mathcal{N}_{26} = 0, \\
&\mathcal{N}_{31} = 0, \quad \mathcal{N}_{32} = i \frac{\sqrt{b}}{2} \cdot \frac{\partial_3}{D_{23}} \cdot D^{\frac{1}{2}} \cdot \frac{1}{D_\varepsilon^{\frac{1}{2}}}, \quad \mathcal{N}_{33} = -i \frac{\sqrt{b}}{2} \cdot \frac{\partial_2}{ D_{23}} \cdot D^{\frac{1}{2}} \cdot \frac{1}{D_\varepsilon^{\frac{1}{2}}}, \\
&\quad \mathcal{N}_{34} = - D_{23} \cdot \frac{1}{2 D^{\frac{1}{2}}} \cdot \frac{1}{D_\varepsilon^{\frac{1}{2}}}, \quad \mathcal{N}_{35} = - \partial_1 \cdot \frac{\partial_{2}}{2 D_{23}} \cdot \frac{1}{D^{\frac{1}{2}}} \cdot \frac{1}{D_\varepsilon^{\frac{1}{2}}}, \quad \mathcal{N}_{36} = - \partial_1 \cdot \frac{\partial_{3}}{D_{23}} \cdot \frac{1}{D^{\frac{1}{2}}} \cdot \frac{1}{D_\varepsilon^{\frac{1}{2}}}.
\end{align*}
\normalsize The remaining expressions are given by
\small
\begin{align*}
&\mathcal{N}_{41} = 0, \quad \mathcal{N}_{42} = - i \frac{\sqrt{b}}{2} \cdot \frac{\partial_3}{D_{23}} D^{\frac{1}{2}} \cdot \frac{1}{D_\varepsilon^{\frac{1}{2}}}, \quad \mathcal{N}_{43} = i \frac{\sqrt{b}}{2}  \cdot \frac{\partial_2}{D_{23}} \cdot D^{\frac{1}{2}} \cdot \frac{1}{D_\varepsilon^{\frac{1}{2}}}, \\
&\quad \mathcal{N}_{44} = - D_{23} \cdot \frac{1}{2 D^{\frac{1}{2}}} \cdot \frac{1}{D_\varepsilon^{\frac{1}{2}}}, \quad \mathcal{N}_{45} = -\partial_1 \cdot \frac{\partial_{1}}{2 D_{23}} \cdot \frac{1}{D^{\frac{1}{2}}} \cdot \frac{1}{D_\varepsilon^{\frac{1}{2}}}, \quad \mathcal{N}_{46} = - \partial_1 \cdot \frac{\partial_{3}}{2 D_{23}} \cdot \frac{1}{D^{\frac{1}{2}}} \cdot \frac{1}{D_\varepsilon^{\frac{1}{2}}}, \\
&\mathcal{N}_{51} = \frac{a}{2} \cdot D_{23} \cdot \frac{1}{D^{\frac{1}{2}}} \cdot \frac{1}{D_\varepsilon^{\frac{1}{2}}}, \quad \mathcal{N}_{52} = \frac{b}{2} \cdot \partial_1 \cdot  \frac{\partial_{2}^2}{D_{23}} \cdot \frac{1}{D^{\frac{1}{2}} } \cdot \frac{1}{D_\varepsilon^{\frac{1}{2}}}, \quad \mathcal{N}_{53} = \frac{b}{2} \cdot \partial_1 \cdot  \frac{\partial_{2}^2}{D_{23}} \cdot \frac{1}{D^{\frac{1}{2}} } \cdot \frac{1}{D_\varepsilon^{\frac{1}{2}}}, \\
&\quad \mathcal{N}_{54} = 0, \quad \mathcal{N}_{55} = i \frac{\partial_3}{2 D_{23} } \cdot \frac{1}{D^{\frac{1}{2}}} \cdot D_\varepsilon^{\frac{1}{2}}, \quad \mathcal{N}_{56} = -i \frac{\partial_2}{2 D_{23}} \cdot \frac{1}{D^{\frac{1}{2}}} \cdot D_\varepsilon^{\frac{1}{2}}, \\
&\mathcal{N}_{61} = -\frac{a}{2} \cdot D_{23} \cdot \frac{1}{D^{\frac{1}{2}}} \cdot \frac{1}{D_\varepsilon^{\frac{1}{2}}}, \quad \mathcal{N}_{62} = -\frac{b}{2} \cdot \partial_1 \cdot \frac{\partial_{2}}{ D_{23}} \cdot \frac{1}{D^{\frac{1}{2}}} \cdot \frac{1}{D_\varepsilon^{\frac{1}{2}}}, \quad \mathcal{N}_{63} = - \frac{b}{2} \cdot \frac{\partial_1}{D^{\frac{1}{2}}} \cdot \frac{\partial_{3}}{ D_{23}} \cdot \frac{1}{D_\varepsilon^{\frac{1}{2}}},\\ 
&\quad \mathcal{N}_{64} = 0, \quad \mathcal{N}_{65} = i \frac{\partial_3}{2 D_{23} } \cdot \frac{1}{D^{\frac{1}{2}}} \cdot D_\varepsilon^{\frac{1}{2}}, \quad \mathcal{N}_{66}  = - i \frac{\partial_2}{2 D_{23} } \cdot \frac{1}{D^{\frac{1}{2}}} \cdot D_\varepsilon^{\frac{1}{2}}.
\end{align*}

\normalsize
\section*{Acknowledgements}

Funded by the Deutsche Forschungsgemeinschaft (DFG, German Research Foundation) -- Project-ID 258734477 -- SFB 1173. I would like to thank Rainer Mandel (KIT) and Roland Schnaubelt (KIT) for discussions on related topics.

\end{document}